\documentclass[a4paper]{amsart}
\usepackage{color}
\usepackage{amssymb,amsmath,amsthm,amstext,amsfonts}
\usepackage{graphicx}
\usepackage{psfrag}
\usepackage{url}
\usepackage{amsfonts}
\usepackage{tikz}

\usepackage{amsmath}
\usepackage{mathrsfs}

\usepackage{amsmath,amsthm,amssymb,amscd}
\usepackage{enumerate}
\usepackage[colorlinks,linkcolor=black,anchorcolor=black,citecolor=black,hyperindex=true,CJKbookmarks=true]{hyperref}
\usepackage{epsfig}
\usepackage[T1]{fontenc}

\makeatletter \@addtoreset{equation}{section} \makeatother

\renewcommand\thetable{\thesection.\@arabic\c@table}

\theoremstyle{plain}
\newtheorem{maintheorem}{Theorem}

\newtheorem{maincorollary}{Corollary}

\newtheorem{mainlemma}{Lemma}

\newtheorem{mainquestion}{Question}

\newtheorem{Thm}{Theorem}[section]
\newtheorem{Lem}[Thm]{Lemma}
\newtheorem{Prop}[Thm]{Proposition}
\newtheorem{Cor}[Thm]{Corollary}

\theoremstyle{remark}
\newtheorem{Def}[Thm] {Definition}
\newtheorem{Rem}[Thm] {Remark}

\makeatletter
\@namedef{subjclassname@2020}{2020 Mathematics Subject Classification}
\makeatother

\begin{document}

\title{Different Statistical behaviors of orbits}

\author{Yiwei Dong, Xiaobo Hou, Wanshan Lin and Xueting Tian}

\address{Yiwei Dong, School of Mathematical Sciences,  Fudan University\\Shanghai 200433, People's Republic of China}
\email{dongyiwei06@gmail.com}

\address{Xiaobo Hou, School of Mathematical Sciences,  Fudan University\\Shanghai 200433, People's Republic of China}
\email{20110180003@fudan.edu.cn}

\address{Wanshan Lin, School of Mathematical Sciences,  Fudan University\\Shanghai 200433, People's Republic of China}
\email{21110180014@m.fudan.edu.cn}

\address{Xueting Tian, School of Mathematical Sciences,  Fudan University\\Shanghai 200433, People's Republic of China}
\email{xuetingtian@fudan.edu.cn}

\begin{abstract}
In this paper, we will study the statistical behaviors of orbits. Firstly, we will show that for a dynamical systems have the shadowing property or almost specification property, the set of nonrecurrent points has full topological entropy. After that, we introduce a criteria for classification of dynamical orbits in order to study the complexity theory of dynamical systems. The criteria is to use upper and lower natural density, upper and lower Banach density to divide different statistical future of dynamical orbits into 56 cases, 28 cases for recurrent orbits and 28 cases for nonrecurrent orbits. We will show the existence of 50 cases and for topologically transitive topologically expanding or topologically transitive topologically Anosov dynamical systems, we will prove that 35 classes, including all the 28 cases for nonrecurrent orbits, can carry full topological entropy. Besides, we will prove that 9 cases can be observable in some differential dynamical systems. Finally, we will apply our results to $\beta$-shifts, $C^{1+\alpha}$ surface diffeomorphisms and Mañé diffeomorphisms.
\end{abstract}

\keywords{Nonrecurrent Points; Shadowing property; Topological entropy; Statistical $\omega$-limit set}
\subjclass[2020] {37D20;  37C50;  37B20;  37B40;  37C45}
\maketitle

\setcounter{tocdepth}{1}
\tableofcontents

\section{Introduction}
In the theory of  dynamical systems, i.e., the study of the asymptotic
behaviors of trajectories or orbits $\{f^n (x)\}_{n\geq 0}$ when $f:X\rightarrow X$  is a continuous map of a compact
metric space $X$ (called a \textit{topological dynamical system}), one may say that two fundamental problems are to
understand how to partition different asymptotic behaviors and how the points with same asymptotic behavior control or determine the complexity of system $(X,f)$. There are two main methods to study asymptotic behavior: one   is from the perspective of topology (or geometry) and another   is  from the perspective of statistics (or measure).
Stable manifold theorem is a classical way to study asymptotic behavior of dynamical orbits. Stable and unstable manifolds are well learned from the perspective of topology (or geometry) in the sense of product structure, smooth regularity and absolute continuity etc, especially in the study of smooth chaotic dynamics including hyperbolic systems, non-uniformly hyperbolic systems and partially hyperbolic systems, see for example \cite{HPS,Katok1,BP,Bowen2}.  In this paper we mainly pay attention to asymptotic behavior from the statistical perspective.

A classic partitioning method is to divide the points of $X$ into two parts, one is recurrent points (denoted by $Rec(f)$), the other is nonrecurrent points (denoted by $NRec(f)$). For $Rec(f)$, from the probabilistic perspective, by Poincaré recurrence theorem, it has full measure for any invariant measure, and thus has full topological entropy, see for example, \cite[Theorem 3.1]{T16}. In fact, in \cite{HTW} and \cite{T16}, a more detailed analysis relative to the topological entropy of Banach upper recurrent points, upper recurrent points, lower recurrent points and almost periodic points (which coincides with Banach lower recurrent points) was obtained for dynamical systems with specification-like property and expansiveness. A natural question is whether $NRec(f)$ can also carry full topological entropy.  For a general dynamical system, $NRec(f)$ may be empty, such as the minimal dynamical systems. A choice is to consider dynamical systems having shadowing property or specification-like properties. These properties were derived from the study of Axiom A diffeomorphisms and have been studied for a long time. Many interesting results were obtained for these properties and we refer to \cite{Kwietniak-Lacka-Oprocha-2016} for an introduction. Our first main result will show these dynamical systems having the property we want. More precisely, let $h_{top}(f)$ denote the topological entropy of $(X,f)$ and let $h_{top}(f,Z)$ denote the (Bowen) topological entropy of a subset $Z$ of $X$, see section \ref{topological-entropy} for a precise definition, then we have the following.
\begin{maintheorem}\label{Theorem A}
	Suppose that $(X, f)$ satisfies the shadowing property or the almost specification property, then   
	$$h_{top}(f,NRec(f))=h_{top}(f).$$
\end{maintheorem}
The precise definitions of the shadowing property and the almost specification property can be found in section \ref{shadowing} and section \ref{almost-spec} respectively. There are many dynamical systems satisfying the shadowing property or the almost specification property. For example, in the family of tent maps, the shadowing property is satisfied for almost all parameters \cite{CKY}; $C^0$ generic continuous self-maps (resp. homeomorphisms) of a compact smooth Riemannian manifold \cite[Theorem 1]{Mazur-Oprocha} (resp. \cite[Theorem 1]{Pilyugin-Plamenevskaya}) satisfy the shadowing property; every topologically mixing interval maps has the periodic specification property \cite[Theorem 6]{Bl} and thus have the almost specification property.

To obtain more information about the asymptotic behaviors of orbits, we will adopt a more refined hierarchical approach. We denote the sets of positive integers, integers and nonnegative integers by $\mathbb{N},  \mathbb{Z},  \mathbb{Z}^+$ respectively.
For any $x\in X$,   the orbit of $x$ is $\{f^n(x)\}_{n=0}^{\infty}$ which we denote by $\operatorname{orb}(x,  f)$.
The $\omega$-limit set of $x$ is defined as
$$\omega_f(x):=\bigcap_{n\geq1}\overline{\bigcup_{k\geq n}\{f^k(x)\}}=\{y\in X: \text{there exist }n_i\to\infty~\text{such that}~f^{n_i}(x)\to y\}.  $$
It is clear that $\omega_f(x)$ is a nonempty compact $f$-invariant set.  Let $S\subseteq \mathbb{N}$,   define
$$\bar{d} (S):=\limsup_{n\rightarrow\infty}\frac{|S\cap \{0,  1,  \cdots,  n-1\}|}n,  \,  \,   \underline{d} (S):=\liminf_{n\rightarrow\infty}\frac{|S\cap \{0,  1,  \cdots,  n-1\}|}n,  $$
where $|Y|$ denotes the cardinality of the set $Y$.   These two concepts are called {\it upper density} and {\it lower density} of $S$,   respectively.   If $\bar{d} (S)=\underline{d}(S)=d,  $ we call $S$ to have \emph{density} of $d.  $  Define $$\overline{B} (S):=\limsup_{|I|\rightarrow\infty}\frac{|S\cap I|}{|I|},  \,  \,   \underline{B} (S):=\liminf_{|I|\rightarrow\infty}\frac{|S\cap I|}{|I|},  $$
here $I\subseteq \mathbb{N}$ is taken from finite continuous integer intervals.   These two concepts are called {\it Banach upper density} and {\it Banach lower density} of $S$,   respectively. These concepts of density are basic and have played important roles in the field of dynamical systems,   ergodic theory and number theory,   etc.
Let $U,  V\subseteq X$ be two nonempty open subsets and $x\in X.  $
Define
sets of visiting time
$$N (U,  V):=\{n\geq 0:  U\cap f^{-n} (V)\neq \emptyset\} \,  \,  \text{ and } \,  \,  N (x,  U):=\{n\ge 0:  f^n (x)\in U\}.  $$

\begin{Def} (Statistical $\omega$-limit sets)
	For $x\in X$ and $\xi=\overline{d},   \,  \underline{d},   \,  \overline{B},  \,   \underline{B}$,   a point $y\in X$ is called $x$-$\xi$-accessible,   if for any $\varepsilon>0,  \,  N (x,  B(y,\varepsilon))\text{ has positive   density w.  r.  t.   }\xi,  $ where     $B(y,\varepsilon)$ denotes  the  ball centered at $y$ with radius $\varepsilon$.
	Let $$\omega_{\xi}(x):=\{y\in X:   y\text{ is } x\text{-}\xi\text{-}\text{accessible}\}.$$  For convenience,   it is called {\it $\xi$-$\omega$-limit set of $x$} or {\it $\xi$-center of $x$}.
\end{Def}

\begin{Rem}
	We learned from \cite{AAN} for maps and \cite{AF} for flows  that $\omega_{\overline{d}}(x)$ is called essential $\omega$-limit set of $x$.
\end{Rem}
Note that
\begin{equation*}\label{X-three-omega}
	\emptyset\subseteq\omega_{\underline{B}}(x)\subseteq \omega_{\underline{d}}(x)\subseteq \omega_{\overline{d}}(x)\subseteq \omega_{\overline{B}}(x)\subseteq \omega_f(x)\subseteq\overline{\operatorname{orb}(x,  f)}.
\end{equation*} 
\subsection{Basic characterization of $\xi$-$\omega$-limit sets} It is easy to check that $\xi$-$\omega$-limit set is compact and invariant (with possibility that some sets are empty). Now, let $M(X)$ denote the space of Borel probability measures on $X$ endowed with weak$^*$ topology. The sets of probability measures, $f$-invariant measures and $f$-ergodic measures supported on $Y\subseteq X$ are denoted by $M(Y),$ $M(f,  Y)$ and $M_{erg}(f,  Y)$ respectively. For $x\in X$ and $n\in\mathbb{N}$,   we define the empirical measure of $x$ as
\begin{equation*}
	\mathcal{E}_{n}(x):=\frac{1}{n}\sum_{j=0}^{n-1}\delta_{f^{j}(x)},
\end{equation*}
where $\delta_{x}$ is the Dirac mass at $x$.    We denote the set of limit points of $\{\mathcal{E}_{n}(x)\}$ by $V_{f}(x)$.   As is known,   $V_f(x)$ is  a nonempty compact connected subset of $M(f,  X)$ \cite[Proposition 3.8]{DGS}.
{ For any two nonnegative integers $a_k<b_k$,   denote
	$[a_k,  b_k]=\{a_k,  a_k+1,  \cdots,  b_k\}$ and $[a_k,  b_k)=[a_k,  b_k-1],  (a_k,  b_k)=[a_k+1,  b_k-1],  (a_k,  b_k]=[a_k+1,  b_k]$.
	A point $x$ is called {\it quasi-generic} for some measure $\mu,  $ if there is a sequence of positive integer intervals $I_k=[a_k,  b_k)$ with $b_k-a_k\to\infty$ such that
	$$\lim_{k\rightarrow\infty}\frac{1}{b_k-a_k}\sum_{j=a_k}^{b_k-1}\delta_{f^j(x)}=\mu$$
	in weak$^*$ topology.
}
Let $V_f^*(x)=\{\mu\in M(f,  X): \,  x \text{ is quasi-generic for } \mu\}$.
This concept is from \cite{Furst} and from \cite[Lemma 2.2]{Wang-Huang-2018}, it is known $V_f^*(x)$ is always nonempty, compact and connected.   Note that $V_f(x)\subseteq V_f^*(x).  $ Let $S_\mu$ denote the support of $\mu.$ We give the basic characterization of $\xi$-$\omega$-limit sets.

\begin{Thm}\label{thm-density-basic-property} Suppose $(X,  f)$ is a topological dynamical system.
	\begin{description}
		\item[(1)] For any $x\in X,   $ $\omega_{\underline{d}}(x)= \bigcap_{\mu\in V_f(x)} S_\mu$.
		\item[(2)] For any $x\in X,   $  $\omega_{\overline{d}}(x)=\overline{\bigcup_{\mu\in V_f(x)} S_\mu}\neq \emptyset$.
		\item[(3)] For any $x\in X,   $  $\omega_{\underline{B}}(x)= \bigcap_{\mu\in V^*_f(x)} S_\mu =  {\bigcap_{\mu\in M(f,  \omega_f(x))}S_{\mu}}= {\bigcap_{\mu\in M_{erg}(f,  \omega_f(x))}S_{\mu}}.  $\,  \,  \,   If \,
		$\omega_{\underline{B}}(x)\neq \emptyset$,   then $ \omega_{\underline{B}}(x)$ is minimal.
		\item[(4)] For any $x\in X,   $ $\omega_{\overline{B}}(x)= \overline{\bigcup_{\mu\in V^*_f(x)} S_\mu}=\overline{\bigcup_{\mu\in M(f,  \omega_f(x))}S_{\mu}}=\overline{\bigcup_{\mu\in M_{erg}(f,  \omega_f(x))}S_{\mu}}\neq \emptyset;$
		
		\item[(5)] For any invariant measure $\mu$ and $\mu$-a.e. $x\in X,$
		$ \omega_{\underline{d}}(x)= \omega_{\overline{d}}(x)= \omega_{\overline{B}}(x)= \omega_f(x)=\overline{\operatorname{orb}(x,f)}.  $ If further $\mu$ is ergodic, then $\mu$-a.e.   $x\in X,   $
		$ \omega_{\underline{d}}(x)= \omega_{\overline{d}}(x)= \omega_{\overline{B}}(x)= \omega_f(x)=\overline{\operatorname{orb}(x,f)}=S_\mu.  $
	\end{description}
\end{Thm}

\begin{Def}
	Given $\alpha_1\alpha_2\cdots\alpha_6\in\{0,1\}^6$ and $x\in X$, we say that $x$ satisfies \textbf{Case~}($\alpha_1\alpha_2\cdots\alpha_6$) if 
	$$\emptyset\star_1\omega_{\underline{B}}(x)\star_2\omega_{\underline{d}}(x)\star_3 \omega_{\overline{d}}(x)\star_4\omega_{\overline{B}}(x)\star_5 \omega_f(x)\star_6\overline{\operatorname{orb}(x,  f)},$$ where 
	\[\star_i=
	\begin{cases}
		\subsetneq &\text{if }\alpha_i=0,\\
		= &\text{if }\alpha_i=1.
	\end{cases}\]
\end{Def}

Given $x\in X$, by Theorem \ref{thm-density-basic-property}, we always have that $\omega_{\overline{d}}(x)\neq\emptyset$ and thus there are $56$ cases at most. Denote $$\mathscr{C}(\alpha_1\alpha_2\cdots\alpha_6):=\{x\in X: x\text{ satisfies \textbf{Case}~}(\alpha_1\alpha_2\cdots\alpha_6)\}.$$Then $$\mathscr{C}(111\alpha_4\alpha_5\alpha_6)=\emptyset\text{ for any }111\alpha_4\alpha_5\alpha_6.$$
Denote $$\mathfrak{A}:=\{0,1\}^6\setminus\{111\alpha_4\alpha_5\alpha_6:\alpha_4\alpha_5\alpha_6\in\{0,1\}^3\}.$$
It is clear that $x\in Rec(f)$ if and only if $x$ satisfies \textbf{Case} ($\alpha_1\alpha_2\alpha_3\alpha_4\alpha_51$) for some $\alpha_1\alpha_2\alpha_3\alpha_4\alpha_51\in\mathfrak{A}$; $x\in NRec(f)$ if and only if $x$ satisfies \textbf{Case} ($\alpha_1\alpha_2\alpha_3\alpha_4\alpha_50$) for some $\alpha_1\alpha_2\alpha_3\alpha_4\alpha_50\in\mathfrak{A}$. The following is a natural question.
\begin{mainquestion}\label{Question A}
 Given one of the $56$ cases, does there always exist a dynamical system $(X,f)$ and $x\in X$ satisfying this case?
\end{mainquestion}
\subsection{The existence of $50$ cases}
It can be easily checked by Theorem \ref{thm-density-basic-property} that $x$ is a point of a minimal dynamical system if and only if it satisfies \textbf{Case} (011111). Since every dynamical system has a minimal subsystem, we can always find points satisfying \textbf{Case} (011111) in any dynamical system. Katznelson and Weiss \cite{Katznelson-Weiss-1981} constructed a dynamical system $(X,f)$ has a unique minimal subsystem $(Y,f)$ with $Y\subsetneq X$ and an ergodic measure $\mu$ with $Y\subsetneq S_\mu$, then by Theorem \ref{thm-density-basic-property}(5), for $\mu$-a.e. $x\in X$, one has $x$ satisfies \textbf{Case} (001111). For non-uniquely ergodic dynamical system satisfying $g$-almost product property, uniform separation property and having an invariant measure with full support, Huang, Tian and Wang \cite{HTW} proved that for any $1\alpha_2\alpha_3\alpha_411\in\mathfrak{A}$, $\mathscr{C}(1\alpha_2\alpha_3\alpha_411)$ has full topological entropy. For one-sided full shifts, Jiang and Tian \cite{Jiang-Tian-2024} find that for any $1\alpha_2\alpha_3\alpha_401\in\mathfrak{A}$,  $\mathscr{C}(1\alpha_2\alpha_3\alpha_401)$ is dense and contains an uncountable DC-1 scrambled subset, see \cite[Lemma 5.1 and Lemma 5.2]{Jiang-Tian-2024}. Overall, the existence of the above $14$ cases was already known. Our second main result will show the existence of other 36 cases.
\begin{maintheorem}\label{Theorem A-new}
	There exists a dynamical system such that for any $0\alpha_2\alpha_3\alpha_401\in\mathfrak{A}$ and any $\beta_1\beta_2\beta_3\beta_4\beta_50\in\mathfrak{A}$, there always exists a point satisfying \textbf{Case} $(0\alpha_2\alpha_3\alpha_401)$ and a point satisfying \textbf{Case} $(\beta_1\beta_2\beta_3\beta_4\beta_50)$.
\end{maintheorem}
\begin{Rem}
	To the best of our knowledge, the existence of \textbf{Case} $(011011)$, \textbf{Case} $(010111)$, \textbf{Case} $(010011)$, \textbf{Case} $(001011)$, \textbf{Case} $(000111)$ and \textbf{Case} $(000011)$ remains unknown. Our results show the existence of all the 28 cases for nonrecurrent orbits. In other words, if we don't consider the relationship between $\omega_f(x)$ and $\overline{\operatorname{orb}(x,f)}$, then all the possible cases do exist.
\end{Rem}
At the end of this subsection, we put a more general question than Question \ref{Question A}. Given a dynamical system $(X,f)$, there exists $\{011111\}\in\mathfrak{A}'\subset \mathfrak{A}$  such that $X=\bigcup_{\alpha_1\alpha_2\cdots\alpha_6\in\mathfrak{A}'}\mathscr{C}(\alpha_1\alpha_2\cdots\alpha_6).$ So from this point of view, there are at most $2^{55}$ kinds of dynamical systems. From the inverse direction, we have the following question.
\begin{mainquestion}
	Suppose that $\mathfrak{A}'$ containing $\{011111\}$ is a subset of $\mathfrak{A}$, then does there exist a dynamical system $(X,f)$ such that $$X=\bigcup_{\alpha_1\alpha_2\cdots\alpha_6\in\mathfrak{A}'}\mathscr{C}(\alpha_1\alpha_2\cdots\alpha_6).$$
\end{mainquestion}

\subsection{The dynamical complexity of $\mathscr{C}(\alpha_1\alpha_2\cdots\alpha_6)$}
In this subsection, we aim to study the dynamical complexity of $\mathscr{C}(\alpha_1\alpha_2\cdots\alpha_6)$ for $\alpha_1\alpha_2\cdots\alpha_6\in\mathfrak{A}$. Due to Theorem \ref{thm-density-basic-property}(5), the set $$\mathscr{C}(001111)\cup\mathscr{C}(011111)\cup\mathscr{C}(101111)$$ always has full measure for any invariant measure and thus has full topological entropy. Hence, the other cases are not detectable from the point of view of any invariant measure. Inspired by the long-standing study of irregular sets, we can also study $\mathscr{C}(\alpha_1\alpha_2\cdots\alpha_6)$ in the sense of topological entropy, Lebesgue measure,  winning set for Schmidt's game \cite{Schmidt-1966} and so on. 

From the point of view of topological entropy, we have the following question.
\begin{mainquestion}
	For which case, can we find a dynamical system such that the set of points satisfying this case is nonempty and has full topological entropy? 
\end{mainquestion}
We will give a partial answer to this question by considering a classical class of dynamical systems. 
\begin{Def}
	A topological dynamical system $(X,f)$ is called \emph{topologically expanding} if $X$ has infinitely many points,  $(X,f)$ is positively expansive and satisfies the shadowing property. When $f$ is a homeomorphism, $(X,f)$ is called \emph{topologically Anosov} if $X$ has infinitely many points,  $(X,f)$ is expansive and satisfies the two-sided shadowing property.
\end{Def}
\begin{Rem}
	From \cite[Corollary 4]{Moot} if  a dynamical system with shadowing property has a recurrent but not minimal point, then the system has positive topological entropy.  Since topologically expanding and topologically Anosov both imply the existence of periodic orbits, hence for topologically transitive topologically expanding dynamical systems and topologically transitive topologically Anosov dynamical systems, every transitive point is recurrent but not minimal, and thus these dynamical systems have positive topological entropy.
\end{Rem}
\begin{Rem} 
	There are topologically Anosov but not Anosov diffeomorphisms. From \cite{Go} we know the existence of $C^{1+Lip}$ non-Anosov diffeomorphisms, which are conjugated to a transitive Anosov diffeomorphism and the conjugation and its inverse is H\"older  continuous.
\end{Rem}
Denote $$\mathfrak{A}_h:=\{011111\}\cup\{1\alpha_2\alpha_3\alpha_411:1\alpha_2\alpha_3\alpha_411\in\mathfrak{A}\}\cup\{\alpha_1\alpha_2\alpha_3\alpha_4\alpha_50:\alpha_1\alpha_2\alpha_3\alpha_4\alpha_50\in\mathfrak{A}\}.$$
\begin{maintheorem}\label{Theorem B}
	Suppose that $(X,f)$ is topologically transitive topologically expanding or topologically transitive topologically Anosov, then for any nonempty open set $U\subset X$ and any  $\alpha_1\alpha_2\cdots\alpha_6\in\mathfrak{A}_h$, we have $$h_{top}(f,\mathscr{C}(\alpha_1\alpha_2\cdots\alpha_6)\cap U)=h_{top}(f)>0.$$
\end{maintheorem}
From the point of view of Lebesgue measure, for differential dynamical systems, we have the following question.
\begin{mainquestion}
	For which case, can we find a differential dynamical system such that the Lebesgue measure of the set of points satisfying this case is positive?
\end{mainquestion}
Denote $$\mathfrak{A}_L:=\{011111, 011110, 101100, 011100, 101111, 101110, 101011,100100,110100\}$$
\begin{maintheorem}\label{Theorem D}
For any  $\alpha_1\alpha_2\cdots\alpha_6\in\mathfrak{A}_L$, there is a differential dynamical system such that $\mathscr{C}(\alpha_1\alpha_2\cdots\alpha_6)$ has positive Lebesgue measure.
\end{maintheorem}

From the point of view of winning set for Schmidt's game, we have the following question. 
\begin{mainquestion}
	For which cases, can we find a dynamical system  such that the set of points  satisfying these cases is a winning set for Schmidt's game? Moreover, for topologically transitive topologically expanding or topologically transitive topologically Anosov dynamical system, what is the minimum cardinality of the subset $\mathfrak{A}''\subset\mathfrak{A}$ satisfying that $$\bigcup_{\alpha_1\alpha_2\cdots\alpha_6\in\mathfrak{A}''}\mathscr{C}(\alpha_1\alpha_2\cdots\alpha_6)$$ is winning?
\end{mainquestion}

\subsection{Applications} Recall from \cite{Walters2} a subshift satisfies shadowing property if and only if it is a subshift of finite type. As a subsystem of one-sided full shift (resp. two-sided full shift), it is positively expansive (resp. expansive). Hence, the results in Theorem \ref{Theorem B} hold for every topologically transitive subshift of finite type with infinite points. In most cases, the dynamical system itself is not topologically transitive topologically expanding or topologically transitive topologically Anosov. If a dynamical system has a sequence of topologically transitive topologically expanding or topologically transitive topologically Anosov subsystems whose topological entropy can be made arbitrarily close to  that of the original system, then results similar to Theorem \ref{Theorem B} still hold. We use this idea to analysis $\beta$-shifts, $C^{1+\alpha}$ surface diffeomorphisms and Mañé diffeomorphisms.
\subsubsection{$\beta$-shifts} Since all $\beta$-shifts have $g$-almost product property \cite{PS2007} and thus have almost specification property, by Theorem \ref{Theorem A}, we have that their nonrecurrent points have full topological entropy. In fact, the results in Theorem \ref{Theorem B} also hold.
\begin{maincorollary}\label{Corollary A}
	Suppose that $(X,f)$ is a $\beta$-shift, then for any nonempty open set $U\subset X$ and any  $\alpha_1\alpha_2\alpha_3\alpha_4\alpha_5\alpha_6\in\mathfrak{A}_h$, we have $$h_{top}(f,\mathscr{C}(\alpha_1\alpha_2\alpha_3\alpha_4\alpha_5\alpha_6)\cap U)=h_{top}(f)>0.$$
\end{maincorollary}

\subsubsection{$C^{1+\alpha}$ surface diffeomorphisms} Let $M$ be a compact smooth Riemannian manifold with $\dim M=2$, then a diffeomorphism on $M$ is called a surface diffeomorphism, for surface diffeomorphism, we have  
\begin{maincorollary}\label{Corollary B}
	Suppose that $f$ is a $C^{1+\alpha}$ surface diffeomorphism on a compact smooth Riemannian manifold $M$ with $\dim M=2$, then for any  $\alpha_1\alpha_2\alpha_3\alpha_4\alpha_5\alpha_6\in\mathfrak{A}_h$, we have $$h_{top}(f,\mathscr{C}(\alpha_1\alpha_2\alpha_3\alpha_4\alpha_5\alpha_6))=h_{top}(f).$$
\end{maincorollary}

\subsubsection{Mañé diffeomorphisms} Mañé diffeomorphisms, which are derived from a hyperbolic toral automorphism on $\mathbb{T}^d$ ($d\geq3$), were introduced by Ma\~{n}\'{e} \cite{Mane-1978}, we also refer to \cite{Climenhaga-Fisher-Thompson-2019} to find a detailed explanation for the construction. We use the notations in \cite{Climenhaga-Fisher-Thompson-2019}, let $f_M:\mathbb{T}^d:\to\mathbb{T}^d$ be a Ma\~{n}\'{e} diffeomorphisms derived from a hyperbolic toral automorphism $f_A$; let $\rho,r>0$ be two parameters controlled in the construction in \cite[section 4]{Climenhaga-Fisher-Thompson-2019}; let $q$ be the fixed point of $f_A$ satisfying that $f_M=f_A$ on $\mathbb{T}^d\setminus B(q,\rho)$ and if an orbit spends a proposition at least $r$ of its time outside $B(q,\rho)$, then it contracts vectors in the one-dimensional center bundle; let $\mathcal{U}_{\rho,r}$ be the sufficient small $C^1$ open neighborhood of $f_M$ in the space of $C^1$ diffeomorphisms on $\mathbb{T}^d$; let $h$ denote the topological entropy of $f_A$ and $L$ be a constant depending on $f_A$ and $g\in\mathcal{U}_{\rho,r}$ and $H(r)=-r\log r-(1-r)\log (1-r)$. 
\begin{maincorollary}\label{Corollary C}
	Suppose that $g\in\mathcal{U}_{\rho,r}$ satisfying that $$r(h+\log L)+H(2r)<h_{top}(g).$$ Then for any  $\alpha_1\alpha_2\alpha_3\alpha_4\alpha_5\alpha_6\in\mathfrak{A}_h$, we have $$h_{top}(g,\mathscr{C}(\alpha_1\alpha_2\alpha_3\alpha_4\alpha_5\alpha_6))=h_{top}(g)>0.$$
\end{maincorollary}

Finally, we pose an inverse question.
\begin{mainquestion}
	If a dynamical system has points satisfying each case of 56 cases, what information can we obtain for the dynamical complexity of this dynamical system?
\end{mainquestion}
\textbf{Organization of this paper.} In section \ref{section 2}, we will introduce some preliminaries. In section \ref{section 3}, we will prove Theorem \ref{Theorem A} assuming Theorem \ref{Theorem B}. In section \ref{section 4}, we will give the basic characterization of $\xi$-$\omega$-limit sets and prove Theorem \ref{thm-density-basic-property}. In section \ref{section 4 new}, we will prove Theorem \ref{Theorem A-new} assuming Theorem \ref{Theorem B}. In section \ref{section 6}, we will study the saturated property, locally-saturated property and locally-star-saturated property. In section \ref{section 7}, we will give the proof of Theorem \ref{Theorem B} based on the results obtained in section \ref{section 4} and section \ref{section 6}. In section \ref{section leb}, we will give the proof of Theorem \ref{Theorem D}.  In section \ref{section 8}, we will give the applications of our main results and prove Corollary \ref{Corollary A}, Corollary \ref{Corollary B} and Corollary \ref{Corollary C}.

\section{Preliminaries}\label{section 2}
\subsection{Metric compatible with the weak$^*$ topology}
The space of Borel probability measures on $X$ is denoted by $M(X)$ and the set of continuous functions on $X$ by $C(X)$.  The set of probability measures, $f$-invariant measures and $f$-ergodic measures supported on $Y\subseteq X$ are denoted by $M(Y),$ $M(f,  Y)$ and $M_{erg}(f,  Y)$ respectively. We endow $\varphi\in C(X)$ the norm $\|\varphi\|=\max\{|\varphi(x)|:x\in X\}$.
Let ${\{\varphi_{j}\}}_{j\in\mathbb{N}}$ be a dense subset of $C(X)$ with $\|\varphi_j\|\neq0$,   then
$$\rho(\xi,  \tau)=\sum_{j=1}^{\infty}\frac{|\int\varphi_{j}d\xi-\int\varphi_{j}d\tau|}{2^{j}\|\varphi_{j}\|}$$
defines a metric on $M(X)$ for the weak$^{*}$ topology  \cite{Walters}.
For $\nu\in M(X)$ and $r>0$,   we denote a ball in $M(X)$ centered at $\nu$ with radius $r$ by
$$\mathcal{B}(\nu,  r):=\{\rho(\nu,  \mu)<r:\mu\in M(X)\}.  $$
One notices that
\begin{equation}\label{diameter-of-Borel-pro-meas}
	\rho(\xi,  \tau)\leq2~~\textrm{for any}~~\xi,  \tau\in M(X).
\end{equation}
It is also well known that the natural embedding $j:x\mapsto \delta_x$ is continuous.   Since $X$ is compact and $M(X)$ is Hausdorff,   one sees that there is a homeomorphism between $X$ and its image $j(X)$.   Therefore,   without loss of generality
we will assume that
\begin{equation}\label{metric-on-X}
	d(x,  y)=\rho(\delta_x,  \delta_y).
\end{equation}
For $x\in X$ and $\varepsilon>0$,   we denote a ball in $X$ centered at $x$ with radius $\varepsilon$ by
$$B(x,\varepsilon):=\{d(x,y)<\varepsilon:y\in X\}.$$
A straight calculation using \eqref{diameter-of-Borel-pro-meas} and \eqref{metric-on-X} gives
\begin{Lem}\label{lem:prohorov}
For any $\varepsilon > 0,\delta >0$, and any two sequences $\{x_i\}_{i=0}^{n-1},\{y_i\}_{i=0}^{n-1}$ of $X$, if $d(x_i,y_i)<\varepsilon$ holds for any $i\in [0,n-1]$, then for any $J\subseteq \{0,1,\cdots,n-1\}$ with $\frac{n-|J|}{n}<\delta$, one has:
	\begin{description}
		\item[(a)] $\rho(\frac{1}{n}\sum_{i=0}^{n-1}\delta_{x_i},\frac{1}{n}\sum_{i=0}^{n-1}\delta_{y_i})<\varepsilon.$
		\item[(b)] $\rho(\frac{1}{n}\sum_{i=0}^{n-1}\delta_{x_i},\frac{1}{|J|}\sum_{i\in J}\delta_{y_i})<\varepsilon+2\delta.$
	\end{description}
\end{Lem}
Lemma \ref{lem:prohorov} is easy to be verified, similar to \cite[Lemma 4.3]{Hou-Tian-Yuan-2023} and shows us that if any two orbit of $x$ and $y$ in finite steps are close in the most time, then the two empirical measures induced by $x,y$ are also close.

\subsection{Notions and notations}
Consider a metric space $(X,  d).  $ Let $A,  B$ be two nonempty subsets,   then the distance from $x\in X$ to $B$ is defined as
$\operatorname{dist}(x,  A):=\inf_{y\in A}d(x,  y).  $
Furthermore,   the distance from $A$ to $B$ is defined as
$\operatorname{dist}(A,  B):=\sup_{x\in A}\operatorname{dist}(x,  B).  $
Finally,   the Hausdorff distance between $A$ and $B$ is defined as
$$d_H(A,  B):=\max\{\operatorname{dist}(A,  B),  \operatorname{dist}(B,  A)\}.  $$

Now consider a topological dynamical system $(X,  f).$ A point $x\in X$ is called {\it recurrent},   if  $x\in \omega_f (x).  $ Otherwise,   $x$ is called {\it non-recurrent}. A point $x\in X$ is called {\it transitive} if $\overline{\operatorname{orb}(x,f)}=X$, the set of transitive points is denoted by $Tran(f)$. A point $x\in X$ is called {\it almost periodic},   if for every open neighborhood $U$ of $x$,   there exists $N\in\mathbb{N}$ such that for every $n\in\mathbb{N},$ $f^k (x)\in U$ for some $k\in [n,  n+N]$, the set of almost periodic points is denoted by $AP(f)$, then $AP(f^n)=AP(f)$ for any $n\in\mathbb{N}$ \cite[Theorem I]{Erdos-Stone1945}. If for every pair of nonempty open sets $U$ and $V$, there is an nonnegative integer $n$ such that $f^{-n}(U)\cap V\neq \emptyset$ then we call $(X,  f)$ \textit{topologically transitive}.
Furthermore,   if for every pair of nonempty open sets $U$ and $V$, there exists an nonnegative integer $N$ such that $f^{-n}(U)\cap V\neq \emptyset$ for every $n>N$,   then we call $(X,  f)$ \textit{topologically mixing}. We say that $(X,  f)$ is \emph{positively expansive} if there exists a constant $c>0$ such that for any $x,  y\in X$,   $d(f^i(x),  f^i(y))> c$ for some $i\in\mathbb{Z}^+$ and we call $c$ the expansive constant. When $f$ is a homeomorphism,  we say that $(X,  f)$ is \emph{expansive} if there exists a constant $c>0$ such that for any $x,  y\in X$,   $d(f^i(x),  f^i(y))> c$ for some $i\in\mathbb{Z}$ and we also call $c$ the expansive constant. We say that a subset $Y$ of $X$ is \emph{$f$-invariant} (or simply \emph{invariant}) if $f(Y)\subseteq Y.$ If $Y$ is a closed $f$-invariant subset of $X,$ then $(Y,f|_Y)$ also is a dynamical system. We will call it a subsystem of $(X,f).$ It is not hard to check that $Rec(f|_Y)=Rec(f)\cap Y$.   Consequently,   $NRec(f|_Y)=NRec(f)\cap Y$. 

A finite sequence $\mathfrak{C}=\langle x_1,  \cdots,  x_l\rangle,  l\in\mathbb{N}$ is called a \emph{chain}.   Furthermore,   if $d(f(x_i),  x_{i+1})<\varepsilon,  1\leq i\leq l-1$,   we call $\mathfrak{C}$ an $\varepsilon$-chain with length $l.$
For any $m\in\mathbb{N}$,   if there are $m$ $\varepsilon$-chains $\mathfrak{C}_i=\langle x_{i,  1},  \cdots,  x_{i,  l_i}\rangle$,   $l_i\in\mathbb{N},  1\leq i\leq m$ satisfying that $d(f(x_{i,  l_i}),x_{i+1,  1})<\varepsilon,   1\leq i\leq m-1$,   then we can concatenate $\mathfrak{C}_i$s to constitute a new $\varepsilon$-chain
$$\langle x_{1,  1},  \cdots,  x_{1,  l_1},  x_{2,  1},  \cdots,  x_{2,  l_2},  \cdots,  x_{m,  1},  \cdots,  x_{m,  l_m}\rangle,$$
which we denote by $\mathfrak{C}_1\mathfrak{C}_2\cdots\mathfrak{C}_m$.

\begin{Def}
	Let $A \subseteq X$ be a nonempty closed invariant set. We say a chain $\mathfrak{C}=\left\langle x_{1}, \cdots, x_{l}\right\rangle$ is in $A$ if $\left\{x_{i}\right\}_{i=1}^{l} \subseteq A$. We say an $\varepsilon$-chain $\mathfrak{C}=\left\langle x_{1}, \cdots, x_{l}\right\rangle$ connects $a$ and $b$ if $x_{1}=a$ and $d\left(f (x_{l}), b\right)<\varepsilon$.
	We call $A$ \emph{internally chain transitive} if for any $a, b \in A$ and any $\varepsilon>0$, there is an $\varepsilon$-chain $\mathfrak{C}$ in $A$ connecting $a$ and $b$.
\end{Def}

\begin{Lem}\label{Lem-omega-naturally-in-ICT} \cite[Lemma 2.1]{HSZ} For any $x\in X,$ $\omega_f(x)$ is internally chain transitive.
	
\end{Lem}

\subsection{Topological entropy and metric entropy}\label{topological-entropy}
\subsubsection{Topological entropy of subsets}
For the topological entropy of a subset, Bowen developed a satisfying definition via dimension language \cite{Bowen1973} which we now illustrate. For $x,  y\in X$ and $n\in\mathbb{N}$,   the Bowen distance between $x,  y$ is defined as
$$d_n(x,  y):=\max\{d(f^i(x),  f^i(y)):i=0,  1,  \cdots,  n-1\}$$
and the Bowen ball centered at $x$ with radius $\varepsilon>0$ is defined as
$$B_n(x,  \varepsilon):=\{y\in X:d_n(x,  y)<\varepsilon\}.  $$

Let $E\subseteq X$,   and $\mathcal {G}_{n}(E,  \sigma)$ be the collection of all finite or countable covers of $E$ by sets of the form $B_{u}(x,  \sigma)$ with $u\geq n$.   We set
$$C(E;t,  n,  \sigma,  f):=\inf_{\mathcal {C}\in \mathcal {G}_{n}(E,  \sigma)}\sum_{B_{u}(x,  \sigma)\in \mathcal {C}}e^{-tu} \,\,\,\text{   and }
C(E;t,  \sigma,  f):=\lim_{n\rightarrow\infty}C(E;t,  n,  \sigma,  f).  $$
Then we define
$$h_{top}(E;\sigma,  f):=\inf\{t:C(E;t,  \sigma,  f)=0\}=\sup\{t:C(E;t,  \sigma,  f)=\infty\}$$
The \textit{(Bowen) topological entropy} of $E$ is
\begin{equation}\label{definition-of-topological-entropy}
	h_{top}(f,  E):=\lim_{\sigma\rightarrow0} h_{top}(E;\sigma,  f).
\end{equation}
In particular, $h_{top}(f)=h_{top}(f,X)$. 
\begin{Lem}\label{Lemma-basic-property-Bowen-entropy-1}\cite[Proposition 2]{Bowen1973}
	Suppose that $(X,f)$ is a dynamical system, $Y, Y_1, Y_2,\cdots\subset X$, then we have
	\begin{enumerate}[(1)]
		\item $h_{top}(f,f(Y))=h_{top}(f,Y)$;
		\item $h_{top}(f,\bigcup_{i\geq1}Y_i)=\sup_{i\geq1}h_{top}(Y_i)$;
		\item $h_{top}(f^n,Y)=nh_{top}(f,Y)$ for any $n\geq1$.
	\end{enumerate}
\end{Lem}
\begin{Lem}\label{Lemma-basic-property-Bowen-entropy-2}\cite[Theorem 3.11]{Feng-Huang-2012}
	Suppose that $\pi:(X,f)\to(Y,g)$ is a factor between two dynamical systems. Then for any $E\subset X$, one has $$h_{top}(f,E)\geq h_{top}(g,\pi(E)).$$
\end{Lem}
\subsubsection{Metric entropy}
We call $(X,  \mathcal{B},  \mu)$ a probability space if $\mathcal{B}$ is a Borel $\sigma$-algebra on $X$ and $\mu$ is a probability measure on $X$.   For a finite measurable partition $\xi=\{A_1,  \cdots,  A_n\}$ of a probability space $(X,  \mathcal{B},  \mu)$,   define
$$H_\mu(\xi)=-\sum_{i=1}^n\mu(A_i)\log\mu(A_i).  $$
Let $f:X\to X$ be a continuous map preserving $\mu$.   We denote by $\bigvee_{i=0}^{n-1}f^{-i}\xi$ the partition whose element is the set $\bigcap_{i=0}^{n-1}f^{-i}A_{j_i},  1\leq j_i\leq n$.   Then the following limit exists:
$$h_\mu(f,  \xi)=\lim_{n\to\infty}\frac1n H_\mu\left(\bigvee_{i=0}^{n-1}f^{-i}\xi\right)$$
and we define the metric entropy of $\mu$ as
$$h_{\mu}(f):=\sup\{h_\mu(f,  \xi):\xi~\textrm{is a finite measurable partition of X}\}. $$
For convenience, we write $h_\mu$ to denote $h_\mu(f)$.
\begin{Lem}\label{Lemma-basic-property-Bowen-entropy-3}\cite[Theorem 1]{Bowen1973}
	Suppose that $(X,f)$ is a dynamical system, $\mu\in M(f,X)$ and $Y\subset X$ satisfies that $\mu(Y)=1$, then $h_{top}(f,Y)\geq h_\mu$.
\end{Lem}

\subsection{Shadowing property}\label{shadowing}

Bowen \cite{Bowen1975} proved that every Anosov diffeomorphism of a compact manifold has the shadowing property and he used this notion efficiently in the study of  $\omega$-limit sets. Since its introduction,   shadowing property has attracted various attentions. Meanwhile,   shadowing property is also generalized to various other forms.   For example,   there are studies on limit-shadowing \cite{Pilyugin2007},   s-limit-shadowing \cite{Mazur-Oprocha,  Sakai2012}, and more forms \cite{Kwietniak-Oprocha,DTY,BMR,Dastjerdi-Hosseini,  ODH,Fakhari-Gane}.

\begin{Def}\label{def-shadowing}
	Suppose $f:X\to X$ is a homeomorphism on compact metric space. For any $\delta>0$,   a sequence $\{x_n\}_{n\in \mathbb{Z}}$ is called a \textit{$\delta$-pseudo-orbit} if
	$d(f(x_n),  x_{n+1})<\delta~~\textrm{for}~~n\in\mathbb{Z}.$ $\{x_n\}_{n\in \mathbb{Z}}$ is \textit{$\varepsilon$-shadowed} by some $y\in X$ if
	$d(f^n(y),  x_n)<\varepsilon~~\textrm{for any}~~n\in\mathbb{Z}.$
	We say that $(X,  f)$ has the \textit{two-sided shadowing property} if for any $\varepsilon>0$,   there exists  $\delta>0$ such that any $\delta$-pseudo-orbit is $\varepsilon$-shadowed by some point in $X$.
\end{Def}

\begin{Def}\label{def-limitshadowig}
	Suppose $f:X\to X$ is a homeomorphism on compact metric space. A sequence $\{x_n\}_{n\in \mathbb{Z}}$ is called a \textit{limit-pseudo-orbit} if
	$\lim_{n\to\pm\infty} d(f(x_n),  x_{n+1})=0.$ $\{x_n\}_{n\in \mathbb{Z}}$ is \textit{limit-shadowed} by $y\in X$ if
	$\lim_{n\to\pm\infty}d(f^n(y),  x_n)=0.$
	If $\{x_n\}_{n\in \mathbb{Z}}$  is both a $\delta$-pseudo-orbit and a limit-pseudo-orbit, then $\{x_n\}_{n\in \mathbb{Z}}$ is called a \textit{$\delta$-limit-pseudo-orbit}. $\{x_n\}_{n\in \mathbb{Z}}$ is \textit{$\varepsilon$-limit-shadowed} by some $y\in X$ if
	$\{x_n\}_{n\in \mathbb{Z}}$ is both $\varepsilon$-shadowed and limit-shadowed by $y$.
	We say that $(X,  f)$ has the \textit{two-sided s-limit-shadowing property} if for any $\varepsilon>0$,   there exists  $\delta>0$ such that any $\delta$-pseudo-orbit $\{x_n\}_{n\in \mathbb{Z}}$ is $\varepsilon$-shadowed by some point $y$ in $X,$ and, if in addition, $\{x_n\}_{n\in \mathbb{Z}}$ is a $\delta$-limit-pseudo-orbit, then it is $\varepsilon$-limit-shadowed by $y.$
\end{Def}
When $f$ is just a continuous map of compact metric space $X$, we say that $(X,  f)$ has the shadowing property or s-limit-shadowing property if Definition \ref{def-shadowing} or Definition \ref{def-limitshadowig} holds for $\{x_n\}_{n\in \mathbb{Z}^+}.$ It can be checked that every dynamical system satisfying two-sided shadowing property has shadowing property and every dynamical system satisfying two-sided s-limit-shadowing property has s-limit-shadowing property.

\begin{Lem}\cite[Theorem 4.3]{Li-Oprocha-2018}\label{Lemma 6.5-old}
	Suppose that $(X,f)$ satisfies the shadowing property and $\mu\in M(f,X)$, if $S_\mu$ is internally chain transitive, then for any $0\leq c\leq h_\mu$, there exists a sequence of minimal subsystem $\{(Y_n,f)\}_{n\geq1}$ such that $\liminf_{n\to\infty}h_{top}(f,Y_n)\geq c$. 
\end{Lem}
Since the support of an ergodic measure is always internally chain transitive, by the variational principle and Lemma \ref{Lemma 6.5-old}, we have
\begin{Lem}\label{Lemma 6.5}
	Suppose that $(X,f)$ satisfies the shadowing property, then for any $\alpha<h_{top}(f)$, there exists a minimal subsystem $(Y,f)$ such that $h_{top}(f,Y)>\alpha$.
\end{Lem}

\begin{Lem}\label{lem-shadowing-full-support}\cite[Proposition 3.5]{Li-Oprocha-2018}
	Suppose that $(X,f)$ is topologically transitive and has the shadowing property, then the set of ergodic measures with full support is residual in $M(f,X)$.
\end{Lem}

\begin{Lem}\label{lem-shadowing-s-limit-shadowing} \cite[Theorem]{LSaka2005}
	If $(X,  f)$ is topologically Anosov, then $(X,  f)$ has the two-sided s-limit shadowing property.
\end{Lem}
\begin{Lem}\label{lem-shadowing-s-limit-shadowing-2} \cite[Theorem 1]{Sakai}
	If $(X,  f)$ is topologically expanding, then $(X,  f)$ has the s-limit shadowing property.
\end{Lem}

\subsection{Specification-like properties}\label{almost-spec}
\begin{Def}\label{definition of specification}
	We say a dynamical system $(X,f)$ has \textit{periodic specification property} if for
	any $\varepsilon > 0$ there is a positive integer $M_\varepsilon$ such that for any points $x_{1}, \dots, x_{k},$ any positive integers $n_{1}, \dots, n_{k} \geq 1$ and any positive integers $p_{1}, \dots, p_{k} \geq M_\varepsilon$, there is a point $x$ in $X$ such that $d(f^{j}(x),f^{j}(x_{1}))\leq \varepsilon$ for every $0\leq j \leq n_{1}-1$ and $$d(f^{j+n_{1}+p_{1}+\dots +n_{i-1}+p_{i-1}}(x),f^{j}(x_{i}))\leq \varepsilon$$ for all $0\leq j \leq n_{i}-1,$ $2 \leq i \leq k,$ and $f^{n_{1}+p_{1}+\dots +n_{k}+p_{k}}(x)=x.$
\end{Def}

Pfister and Sullivan generalized the specification property to the $g$-almost product property in the study of large deviation \cite{PS2005}.   Later on,   Thompson  renamed a weaker version of it as the almost specification property in the study of irregular points \cite{Thompson2012}.   The only difference is that the blowup function $g$ can depend on $\varepsilon$ in the latter case.

\begin{Def}
	Let $\varepsilon_0>0$.   A function $g:\mathbb{N}\times (0,  \varepsilon_0)\to \mathbb{N}$ is called a \emph{mistake function} if for all $\varepsilon\in (0,  \varepsilon_0)$ and all $n\in \mathbb{N}$,   $g(n,  \varepsilon)\leq g(n+1,  \varepsilon)$ and
	$\lim_{n\to\infty}\frac{g(n,  \varepsilon)}{n}=0.$
\end{Def}
If $\varepsilon\geq \varepsilon_0$,   we define $g(n,  \varepsilon)=g(n,  \varepsilon_0)$.
For $n\in \mathbb{N}$ large enough such that $g(n,  \varepsilon)<n$,   let $\Lambda_n=\{0,  1,  \cdots,  n-1\}$.   Define
the $(g;n,  \varepsilon)$-Bowen ball centered at $x$ as the closed set
$$B_n(g;x,  \varepsilon):=\{y\in X:\exists~\Lambda\subseteq \Lambda_n,  ~|\Lambda_n\setminus\Lambda|\leq g(n,  \varepsilon)~\text{and}~\max\{d(f^j(x),  f^j(y)):j\in\Lambda\}\leq\varepsilon\}.  $$
\begin{Def}
	The dynamical system $(X,  f)$ has the \emph{almost specification} property with mistake function $g$,   if there exists a function $k_g:(0,  +\infty)\to \mathbb{N}$ such that for any
	$\varepsilon_{1}>0,  \cdots,  \varepsilon_{m}>0$,   any points $x_{1},  \cdots,  x_{m}\in X$,   and any integers $n_{1}\geq k_g(\varepsilon_{1}),  \cdots,  n_{m}\geq k_g(\varepsilon_{m})$,   we can find a point $z\in X$ such that
	\begin{equation*}
		f^{l_{j}}(z)\in B_{n_{j}}(g;x_{j},  \varepsilon_{j}),  ~j=1,  \cdots,  m,
	\end{equation*}
	where $n_{0}=0~\textrm{and}~l_{j}=\sum_{s=0}^{j-1}n_{s}$.
\end{Def}

Finally, we introduce two other specification-like properties.
\begin{Def}
	We say that $(X,f)$ has the \textit{periodic gluing orbit property} if for any $\varepsilon >0$ there exists a positive integer $K_\varepsilon$ such that for any points $x_{1}, \dots, x_{k}$ and positive integers $n_{1}, \dots, n_{k} \geq 1$ there are $p_{1}, \dots, p_{k} \leq K_\varepsilon$ and $x\in X$ so that $d(f^{j}(x),f^{j}(x_{1}))\leq \varepsilon$ for every $0\leq j \leq n_{1}-1$ and $$d(f^{j+n_{1}+p_{1}+\dots +n_{i-1}+p_{i-1}}(x),f^{j}(x_{i}))\leq \varepsilon$$ for all $0\leq j \leq n_{i}-1$ and $2 \leq i \leq k,$ and $f^{n_{1}+p_{1}+\dots +n_{k}+p_{k}}(x)=x.$
\end{Def}
\begin{Def}
	We say that $(X,f)$ has the \emph{approximate product property}, if for any $\varepsilon>0$, $\delta_1>0$ and $\delta_2>0$, there exists $N=N(\varepsilon,\delta_1,\delta_2)\in\mathbb{N}$ such that for any $n\geq N$ and any sequence $\{x_i\}_{i=1}^{+\infty}$ of $X$, there exist a sequence of integers $\{h_i\}_{i=1}^{+\infty}$ and $x\in
	X$ satisfying $h_1=0$, $n\leq h_{i+1}-h_i\leq n(1+\delta_2)$ and $$|\{0\leq j\leq n-1\mid d(f^{h_i+j}x,f^jx_i)>\varepsilon\}|\leq\delta_1n.$$
\end{Def}

\subsection{Uniform separation property}
For $\delta>0$,   $\varepsilon>0$ and $n\in\mathbb{N}$,   two points $x$ and $y$ are $(\delta,  n,  \varepsilon)$-separated if
$$|\{j:d(f^{j}(x),  f^{j}(y))>\varepsilon,  ~0\leq j\leq n-1\}|\geq\delta n.  $$
A subset $E$ is $(\delta,  n,  \varepsilon)$-separated if any pair of different points of $E$ are $(\delta,  n,  \varepsilon)$-separated. 
\begin{Lem}\cite{PS2005}\label{separated-for-nu}
	Let $\mu$ a $f$-ergodic measure and $\eta<h_\mu$.   Then there exist $\delta^*>0$ and $\varepsilon^*>0$ so that for each neighborhood $F$ of $\mu$ in $M(X)$,   there exists $n_F^*\in \mathbb{N}$ such that for any $n\geq n_F^*$, there is a $(\delta^*,  n,  \varepsilon^*)$-separated set $\Gamma_n\subseteq X_{n,  F}$ with
	$$|\Gamma_n|\geq e^{n\eta},$$
	where $X_{n,  F}:=\{x\in X:\mathcal{E}_{n}(x)\in F\}.$
\end{Lem}
Furthermore,   if the above $\delta^*$ and $\varepsilon^*$ can be chosen to be independent of $\mu$,   one has the following definition.
\begin{Def}\label{def-uniformseparation}
	We say $(X,  f)$ has the uniform separation property if the following holds.   For any $\eta>0$,   there exist $\delta^*>0$ and $\varepsilon^*>0$ so that for any $f$-ergodic measure $\mu$ and any neighbourhood $F\subseteq M(X)$ of $\mu$,   there exists $n^*_{F,  \mu,  \eta}\in\mathbb{N}$ such that for any $n\geq n^*_{F,  \mu,  \eta}$,   there is a $(\delta^*,  n,  \varepsilon^*)$-separated set $\Gamma_n\subseteq X_{n,  F}\cap S_\mu$ with
	$$|\Gamma_n|\geq e^{n(h_{\mu}-\eta)}. $$
\end{Def}
\begin{Rem}
	The key observation of Definition \ref{def-uniformseparation} is that the selection of $\delta^*,  \varepsilon^*$ does not depend on $\mu$ and $F$.   This is exactly what `uniform' means. The original definition in  \cite[Definition 3.1]{PS2007} require that ``$|\Gamma_n|\geq 2^{n(h_{\mu}-\eta)}$'' since ``$\log_2$'' rather than ``$\log$'' is used in the definition of metric entropy. Also, the original definition only require that ``$\Gamma_n\subseteq X_{n,  F}$''. From the remark after \cite[Definition 3.1]{PS2007}, we know that uniform separation property implies that $h_{top}(f)$ is finite.
\end{Rem}
Uniform separation property is satisfied by some typical dynamical systems,   the following results still holds when we use Definition \ref{def-uniformseparation} as the definition of uniform separation property.
\begin{Prop}\cite[Theorem 3.1]{PS2007}\label{Prop-PS2007}
	If $(X,  f)$ is positively expansive (resp. expansive),   $h$-expansive or asymptotic $h$-expansive,   then $(X,  f)$ has the uniform separation property.
\end{Prop}

For $\varepsilon>0$ and $n\in\mathbb{N}$,   two points $x$ and $y$ are $(n,  \varepsilon)$-separated if
$d_n(x,y)>\varepsilon.$
A subset $E$ is $(n,  \varepsilon)$-separated if any pair of different points of $E$ are $(n,  \varepsilon)$-separated.
\begin{Lem}\cite{DOT}\label{useful-Proposition}
	Let $\mu$ be an ergodic measure and $\eta<h_\mu$.   Then there exists $\varepsilon>0$ such that for each neighbourhood $F$ of $\mu$ in $M(X)$,   there exists $n_F\in\mathbb{N}$ such that for any $n\geq n_F$,   there exists an $(n,  \varepsilon)$-separated set {$\Gamma_n\subset X_{n,  F}\cap S_\mu$}with
	$$|\Gamma_n|\geq e^{n\eta}.  $$
\end{Lem}

\subsection{Entropy-dense properties}
In this subsection,   we introduce several  entropy-dense properties,   which shall serve for our needs in the future.
Eizenberg,   Kifer and Weiss  \cite{EKW} proved for systems with the specification property that any $f$-invariant probability measure $\nu$ is the weak limit of a sequence of ergodic measures
$\{\nu_n\}$,   such that the entropy of $\nu$ is also the limit of the entropies of the $\{\nu_n\}$.   This is a central point in
large deviations theory,   which was first emphasized in \cite{FO}.   Meanwhile,   this also plays a crucial part in the computing of Billingsley dimension \cite{Billingsley1960,  Billingsley1961} on shift spaces \cite{PS2003}.
Pfister and Sullivan refer to this property as the \emph{entropy-dense} property \cite{PS2005}.

\begin{Def}
	We say $\Lambda$ satisfies the  entropy-dense property (or $M_{erg}(f,  \Lambda)$ is  entropy-dense in $M(f,  \Lambda)$),  if for any $\mu\in M(f,  \Lambda)$,   for any neighborhood $G$ of $\mu$ in $M(\Lambda)$,   and for any $\eta<h_{\mu}(f)$,   there exists a $\nu\in M_{erg}(f,  \Lambda)$ such that $h_{\nu}(f)>\eta$ and $\nu \in G$.
\end{Def}

\begin{Def}
	We say $\Lambda$ satisfies the refined entropy-dense property (or $M_{erg}(f,  \Lambda)$ is refined entropy-dense in $M(f,  \Lambda)$),  if for any $\mu\in M(f,  \Lambda)$,   for any neighborhood $G$ of $\mu$ in $M(\Lambda)$,
	and for any $\eta<h_{\mu}(f)$,   there exists a closed $f$-invariant set $\Lambda_{\mu}\subseteq \Lambda $ such that $M(f,  \Lambda_{\mu})\subseteq G$ and $h_{top}(f,  \Lambda_{\mu})>h_{\mu}(f)-\eta$. By classical variational principle,
	it is equivalent that for any neighborhood $G$ of $\mu$ in $M(\Lambda)$,   and for any $\eta>0$,   there exists a $\nu\in M_{erg}(f,  \Lambda)$ such that $h_{\nu}(f)>h_{\mu}(f)-\eta$ and $M(f,  S_{\nu})\subseteq G$.
\end{Def}

Of course,   $\textrm{refined entropy-dense} \Rightarrow\textrm{entropy-dense}\Rightarrow $ ergodic measures are dense in the space of invariant measures.   For systems with the approximate product property,  Pfister and Sullivan in fact had obtained the refined entropy-dense properties by showing the following lemma.

\begin{Lem}\cite[Proposition 2.3]{PS2005}\label{Lemma-PS}
	Suppose that $(X,  f)$ has the approximate product property and that $\nu\in M(X,f)$
	verifies the conclusions of Lemma \ref{separated-for-nu}.   Let $h<h_{\nu}(f)$.
	Then,   there exists $\varepsilon>0$ such that for any neighborhood $G$ of $\nu$,   there exists a closed $f$-invariant subspace $Y\subseteq X$ and an $n_G\in\mathbb{N}$ with the following properties:
	\begin{enumerate}
		\item $\mathcal{E}_n(y)\in G$ whenever $n\geq n_G$ and $y\in Y$.
		\item For all $l\in\mathbb{N}$ there exists a $(l\cdot n_G,  \varepsilon)$-separated subset of $Y$ with cardinality greater than $\exp(l\cdot n_G\cdot h)$.
	\end{enumerate}
	In particular,   $h_{top}(f,  Y)\geq h$.
\end{Lem}

Note that topologically transitivity plus shadowing property implies approximate product property \cite[Theorem 30]{Kwietniak-Lacka-Oprocha-2016}. And from the definitions, we know that periodic gluing orbit property also implies approximate product property. Then by Lemma \ref{Lemma-PS}, we have 
\begin{Cor}\label{prop-entropy-dense-for-shadowing}
	Suppose that $(X,  f)$  is one of the following dynamical system:
	\begin{enumerate}[(1)]
		\item topologically transitive dynamical systems satisfying the shadowing property;
		\item dynamical systems with periodic gluing orbit property.
	\end{enumerate}
	Then $(X,  f)$ has the refined entropy-dense property.
\end{Cor}

\section{Topological entropy of nonrecurrent points: proof of Theorem \ref{Theorem A}}\label{section 3}
In this section, we prove Theorem \ref{Theorem A} by using Theorem \ref{Theorem B}. Let $(\Sigma_{m}^+,  \sigma)$ (resp. $(\Sigma_m,\sigma)$) denote the one-sided (resp. two-sided) full shift over a finite alphabet $\mathcal{A}$ with $|\mathcal{A}|=m$.
\begin{Lem}\label{horseshoe}\cite[Lemma 3.3]{DOT}
	Suppose $(X,  f)$ is a topological dynamical system. If $(X,  f)$ has the shadowing property and $h_{top}(f)>0,$ then for any $0<\alpha<h_{top}(f)$ there are $m,  k\in \mathbb{N}$,   $\frac{\log m}{k}>\alpha$ and a closed set $\Delta\subseteq X$
	invariant under $f^k$ such that there is a factor map $\pi: (\Delta,   f^k)\to (\Sigma_{m}^+,  \sigma)$.  
\end{Lem}
For dynamical systems satisfying specification property, we can obtain a similar result.

\begin{Lem}\label{horseshoe2}
	Suppose $(X,  f)$ satisfies the almost specification property with $h_{top}(f)>0$.   Then for any $0<\alpha<h_{top}(f)$,   there are $m,  k\in \mathbb{N}$,   $\frac{\log m}{k}>\alpha$ and a closed set $\Lambda\subseteq X$
	invariant under $f^k$ such that there is a factor map $\pi: (\Lambda,   f^k)\to (\Sigma_{m}^+,  \sigma)$.
\end{Lem}
\begin{proof}
	By the variational principle,   choose some invariant measure $\nu$ such that $0<\alpha<h_{\nu}\leq h_{top}(f)$.   According to Lemma \ref{separated-for-nu},   there exist $\delta>0$ and $\varepsilon>0$ so that for each neighborhood $F$ of $\nu$ in $M(X)$,   there exists $n_F\in \mathbb{N}$ such that for any $n\geq n_F$,   there exists $\Gamma_n\subseteq X_{n,  F}$ which is $(\delta,  n,  \varepsilon)$-separated and satisfies $\log|\Gamma_n|> n\alpha$.   Now choose an arbitrary neighbourhood $F$ of $\nu$.   Since $\lim_{n\to\infty}\frac{g(n,  \varepsilon/3)}{n}=0$,   there exists an $N\in \mathbb{N}$ such that for any $n\geq N$,   $g(n,  \frac{\varepsilon}{3})\leq \frac{\delta}{3}$.   Let $k=\max\{N, n_F,  k_g(\frac{\varepsilon}{3}),  [\frac{3}{\delta}]+1\}$.   Enumerate the elements of $\Gamma_k$ as $\{p_1,  \cdots,  p_m\}$ where $m=|\Gamma_k|$.
	
	Given $l\in\mathbb{N}$, let $\Sigma^+_m$ (resp. $\Sigma^+_{m,l}$) be the set whose element is $(a_0a_1\cdots)$ (resp. $(a_0a_1\cdots a_{l-1})$) such that $a_i\in\{p_1,  \cdots,  p_m\}$,   $i\in \mathbb{Z}^+$ (resp. $0\leq i\leq l-1$).   For every $\xi\in \Sigma^+_m$ and $l\in\mathbb{N}$,   denote $$Y_{\xi,l}=\{z\in X: f^{ik}(z)\in B_k(g;\xi_i,  \frac{\varepsilon}{3})~\text{for}~0\leq i\leq l-1\}$$ and 
	$$Y_{\xi}=\left\{z\in X: f^{ik}(z)\in B_k(g;\xi_i,  \frac{\varepsilon}{3})~\text{for}~i\in \mathbb{Z}^+\right\}.  $$ Then both $Y_{\xi,l}$ and $Y_{\xi}$ are closed. 
	By the almost specification property,   the selection of $k$ and the compactness of $X$,   $Y_{\xi}$ is nonempty.   Note that if $\xi\neq \psi$ then there is $t\in\mathbb{Z}^+$ such that $\xi_t \neq \psi_t$.   For any $x\in Y_{\xi}$ and $y\in Y_{\psi}$,
	since $\delta-2g(k,  \frac{\varepsilon}{3})\geq\frac{\delta}{3}$,   $f^{tk}x$ and $f^{tk}y$ are $(\frac{\delta}{3},  k,  \frac{\varepsilon}{3})$-separated.   Moreover,   $\frac{\delta}{3}\cdot k\geq 1$.   So $x,  y$ are $((t+1)k,  \frac{\varepsilon}{3})$-separated which implies that $x\neq y$.   Therefore,   $Y_{\xi}\cap Y_{\psi}=\emptyset$.   So we define $\Lambda$ as the disjoint union of $Y_{\xi}$:
	$$\Lambda=\bigsqcup_{\xi\in\Sigma_m^+}Y_{\xi}=\bigcap_{l=1}^\infty\bigcup_{(\xi_0\cdots\xi_{l-1})\in\Sigma^+_{m,l}}Y_{\xi,l}.  $$ Then $\Lambda$ is closed.
	Note that $f^k(Y_{\xi})\subseteq Y_{\sigma(\xi)}$.   So $\Lambda$ is $f^k$-invariant.   It is not hard to see that if $x\in Y_{\xi}$, $y\in Y_{\psi}$ and $d(f^l(x),  f^l(y))<\varepsilon/3$
	for $l=0,1,\cdots,   ks-1$ then $\xi_i=\psi_i$ for $i=0,  \ldots,  s-1$.
	Therefore,   if we define $\pi: \Lambda \to \Sigma_{m}^+$ as
	$$\pi(x):=\xi~~\textrm{if}~~x\in Y_{\xi},  $$
	then $\pi$
	is a continuous surjection.   It is  clearly that $\sigma \circ \pi =\pi \circ f^k$.
	
	Finally,    observe that
	$
	\frac{\log m}{k}=\frac{\log|\Gamma_k|}{k}>\alpha.
	$
	The proof is completed.
\end{proof}
Now, we give the proof of Theorem \ref{Theorem A} by using Theorem \ref{Theorem B}.
\subsection{Proof of Theorem \ref{Theorem A}} Suppose that $(X, f)$ satisfies the shadowing property or the almost specification property. When $h_{top}(f)=0$,   there is nothing to prove.   So we suppose $h_{top}(f)>0$.   By Lemma \ref{horseshoe} or Lemma \ref{horseshoe2},   for any $0<\alpha<h_{top}(f)$,   there are $m,  k\in \mathbb{N}$,   $\log (m)/k>\alpha$ and a closed and $f^k$-invariant set $\Lambda\subseteq X$ with a semiconjugation $\pi:(\Lambda,  f^k)\to(\Sigma_m^+,  \sigma)$. 
Since $(\Sigma_m^+,  \sigma)$ is topologically expanding and transitive,   by Theorem \ref{Theorem B},   we obtain that $h_{top}(\sigma,NRec(\sigma))=h_{top}(\sigma,  \Sigma_m^+)>k\alpha$.

On the other hand,  $\pi(Rec(f^k,  \Lambda))\subseteq Rec(\sigma)$,   so $NRec(f^k,  \Lambda)\supseteq \pi^{-1} NRec(\sigma)$.   Meanwhile,   note that $NRec(f^k)=NRec(f)$.   Therefore,   since $\pi$ is a semiconjugation, by Lemma \ref{Lemma-basic-property-Bowen-entropy-1} and Lemma \ref{Lemma-basic-property-Bowen-entropy-2}, one has
$$h_{top}(f,  NRec(f))=\frac1kh_{top}(f^k,  NRec(f))=\frac1kh_{top}(f^k,  NRec(f^k))\geq\frac1kh_{top}(\sigma,  NRec(\sigma))>\alpha.  $$
By the arbitrariness of $\alpha$,   we see that
$h_{top}(f,NRec(f))=h_{top}(f).  $  \qed

\section{Basic characterization of $\xi$-$\omega$-limit sets: proof of Theorem \ref{thm-density-basic-property}}\label{section 4}
In this section, we give the basic characterization of $\xi$-$\omega$-limit sets, which will help us find the points satisfying different statistical behaviours. 
To prove Theorem \ref{thm-density-basic-property}, we need some preparations. For any $x\in X$,   we define the measure center of $x$ as
$$C^*_x:=\overline{\bigcup_{\mu\in M(f,  \omega_f(x))}S_{\mu}}.  $$
Furthermore,   we define the measure center of an invariant set $\Lambda\subseteq X$ as
$$C^*_\Lambda:=\overline{\bigcup_{\mu\in M(f,  \Lambda)}S_{\mu}}.$$
\begin{Lem}\label{lem-cap-cup}
	Let $\Lambda \subseteq X$ be compact and $f$-invariant. We have the following relations:
	\begin{description}
		\item[(1)] $\bigcap_{\mu\in M_{erg}(f,  \Lambda)}S_{\mu}=\bigcap_{\mu\in M(f,  \Lambda)}S_{\mu}$.
		\item[(2)] $C^*_\Lambda=\overline{\bigcup_{\mu\in M_{erg}(f,  \Lambda)}S_{\mu}}$.
	\end{description}
\end{Lem}
\begin{proof}
	(1) It is clear that $\bigcap_{\mu\in M_{erg}(f,  \Lambda)}S_{\mu}\supseteq \bigcap_{\mu\in M(f,  \Lambda)}S_{\mu}$.   So we only need to prove that $$\bigcap_{\mu\in M_{erg}(f,  \Lambda)}S_{\mu}\subseteq\bigcap_{\mu\in M(f,  \Lambda)}S_{\mu}.$$   Indeed,   for any $x\in \bigcap_{\mu\in M_{erg}(f,  \Lambda)}S_{\mu}$ and $\varepsilon>0$,   one has
	\begin{equation}\label{eq-mu-ergodic}
		\mu(B(x,  \varepsilon))>0~\textrm{for any}~\mu\in M_{erg}(f,  \Lambda).
	\end{equation}
	If $\nu(B(x,  \varepsilon))=0$ for some $\nu\in M(f,  \Lambda)$,   then by the ergodic decomposition Theorem \cite{Walters},   there is a unique measure $\tau$ on the Borel subsets of the compact metrizable space $M(f,  \Lambda)$ such that $\tau(M(f,  \Lambda))=1$ and
	$$0=\nu(B(x,  \varepsilon))=\int_{M(f,  \Lambda)}\mu(B(x,  \varepsilon))d\tau(\mu).  $$
	Therefore,   for $\tau$-a. e. $\mu\in M_{erg}(f,  \Lambda)$,   $\mu(B(x,  \varepsilon))=0$,   contradicting \eqref{eq-mu-ergodic}.   Thus
	$x\in \bigcap_{\mu\in M(f,  \Lambda)}S_{\mu}$ which implies that $\bigcap_{\mu\in M_{erg}(f,  \Lambda)}S_{\mu}\subseteq\bigcap_{\mu\in M(f,  \Lambda)}S_{\mu}$.

	(2) It is clear that $\overline{\bigcup_{\mu\in M_{erg}(f,  \Lambda)}S_{\mu}}\subseteq\overline{\bigcup_{\mu\in M(f,  \Lambda)}S_{\mu}}$.   So it is sufficient to prove that $\overline{\bigcup_{\mu\in M_{erg}(f,  \Lambda)}S_{\mu}}\supseteq\overline{\bigcup_{\mu\in M(f,  \Lambda)}S_{\mu}}$.   Indeed,   for any $\mu\in M(f,  \Lambda)$ and any $x\in S_{\mu}$,   one has that
	$$\mu(B(x,  \varepsilon))>0~\textrm{for any}~\varepsilon>0.  $$
	By the ergodic decomposition Theorem,   there is a $\mu_{\varepsilon}\in M_{erg}(f,  \Lambda)$ with $\mu_{\varepsilon}(B(x,  \varepsilon))>0$.   This implies that $B(x,  \varepsilon)\cap S_{\mu_{\varepsilon}}\neq\emptyset$.   Since $\varepsilon>0$ is arbitrary,   $x\in\overline{\bigcup_{\mu\in M_{erg}(f,  \Lambda) S_\mu}}$,   which yields that $\overline{\bigcup_{\mu\in M_{erg}(f,  \Lambda)}S_{\mu}}\supseteq\overline{\bigcup_{\mu\in M(f,  \Lambda)}S_{\mu}}$.
\end{proof}
\begin{Lem}\label{lemma-V*-contains-ergodic}
	For $(X,  f)$ and $x\in X$,
	$M_{erg}(f,  \omega_f(x))\subseteq V^*_f(x)\subseteq M(f,  \omega_f(x)).  $
	If $M_{erg}(f,  \omega_f(x))$ is dense in $M(f,  \omega_f(x))$,   then    $V^*_f(x)=M(f,  \omega_f(x))$.
\end{Lem}
\begin{proof}
	From \cite [Proposition 3.9,   Page 65] {Furst} we know that for a point $x_0$ and an ergodic measure $\mu_0\in M(f,  \omega_f(x_0))$,   $x_0$ is quasi-generic for $\mu_0.  $ So   $M_{erg}(f,  \omega_f(x))\subseteq V_f^*(x)\subseteq M(f,  \omega_f(x))$. If  $M_{erg}(f,  \omega_f(x))$ is dense in $ M(f,  \omega_f(x))$, then $V^*_f(x)=M(f,  \omega_f(x))$.
\end{proof}

{
	\begin{Lem}\label{lem-X-B-lower-star}
		For $(X,  f)$ and $x\in X$,   $z\in \omega_{\underline{B}}(x)$ if and only if for any $\varepsilon>0$,   there exists an $N\in\mathbb{N}$ such that for any $N$ consecutive positive integers $n+1,  \cdots,  n+N$,   $d(f^{n+i}(x),  z)<\varepsilon$ for some $1\leq i\leq N$.
	\end{Lem}
	\begin{proof}
		Sufficiency: Let $I_k=[a_k,  b_k)$ be a sequence of positive integer intervals with $b_k-a_k\to\infty$ such that
		$$\lim_{k\to\infty}\frac{|N(x,  B(z,\varepsilon))\cap I_k|}{|I_k|}=\liminf_{|I|\to\infty}\frac{|N(x,  B(z,\varepsilon))\cap I|}{|I|}.  $$
		Let $b_k-a_k=l_kN+r_k,  0\leq r_k\leq N-1$ for each $k\in\mathbb{N}$.   Then for $k$ large enough such that $l_k\geq 1$,   one has
		$$\frac{|N(x,  B(z,\varepsilon)))\cap I_k|}{|I_k|}\geq\frac{l_k}{l_kN+r_k}\geq\frac{1}{2N}.  $$
		This implies that
		$$\liminf_{|I|\to\infty}\frac{|N(x,  B(z,\varepsilon)))\cap I|}{|I|}=\lim_{k\to\infty}\frac{|N(x,  B(z,\varepsilon)))\cap I_k|}{|I_k|}\geq \frac{1}{2N}>0.  $$
		Since $\varepsilon$ is arbitrary,   we see that $z\in \omega_{\underline{B}}(x)$.
		
		Necessity: Otherwise, there is $\varepsilon_0>0$ and for any $k\in\mathbb{N}$,   there exists an $n_k\in\mathbb{N}$ such that $d(f^{n_k+i}(x),  z)\geq\varepsilon_0$ for $1\leq i\leq k$.   Now consider the sequence of intervals $I_k=[n_k+1,  n_k+k]$.   Then for any $k\in\mathbb{N}$,
		$$|N(x,  B(z,\varepsilon_0))\cap I_k|=0.  $$
		This implies that
		$$\liminf_{|I|\to\infty}\frac{|N(x,  B(z,\varepsilon_0))\cap I|}{|I|}=\lim_{k\to\infty}\frac{|N(x,  B(z,\varepsilon_0))\cap I_k|}{|I_k|}=0,  $$
		contradicting the fact that $z\in \omega_{\underline{B}}(x)$.
	\end{proof}
	\begin{Cor}\label{cor-X-B-lower-star-is-minimal}
		Fix any $x\in X$.   Then for any $y\in \omega_f(x)$,   $\omega_{\underline{B}}(x)\subseteq \overline{\operatorname{orb}(y,  f)}$.   As a result,   $\omega_{\underline{B}}(x)$ is either empty or a minimal set.
	\end{Cor}
	\begin{proof}
		If $\omega_{\underline{B}}(x)=\emptyset$,   there is nothing to prove.   So we suppose $\omega_{\underline{B}}(x)\neq\emptyset$.   Select an arbitrary $z\in \omega_{\underline{B}}(x)$.   For any $y\in\omega_f(x)$,   if $z\notin \overline{\operatorname{orb}(y,  f)}$,   then let $2\varepsilon:=\operatorname{dist}(z,  \overline{\operatorname{orb}(y,  f)})>0$.   By Lemma \ref{lem-X-B-lower-star},   there is an $N\in\mathbb{N}$ such that for any $n\in\mathbb{N}$,   $d(f^{n+i}(x),  z)<\varepsilon$ for some $1\leq i\leq N$.   Now choose a $0<\delta<\varepsilon$ such that $d(u,  v)<\delta$ implies that $d(f^j(u),  f^j(v))<\varepsilon$,   $1\leq j\leq N$.   Moreover,   since $y\in\omega_f(x)$,   we select a $p\in\mathbb{N}$ such that $d(f^p(x),  y)<\delta$.   Then there exist a $1\leq q\leq N$ such that $d(f^{p+q}(z),  z)<\varepsilon$.   However,   the selection of $\delta$ and the condition $1\leq q\leq N$ indicate that $d(f^{p+q}(x),  f^q(y))<\varepsilon$.   So we have $d(z,  f^q(y))<2\varepsilon$,   contradicting the definition of $\varepsilon$.   Therefore,   $z\in \overline{\operatorname{orb}(y,  f)}$.   Since $z\in \omega_{\underline{B}}(x)$ is arbitrary,   we have $\omega_{\underline{B}}(x)\subseteq \overline{\operatorname{orb}(y,  f)}$.
		
		Moreover,   it is well know that (by Zorn Lemma) every topological dynamical system has a minimal system.   So we select some $z$ from a minimal subset of $\omega_f(x)$.   Therefore,   $\overline{\operatorname{orb}(z,  f)}$ is minimal.   Moreover,   since $\omega_{\underline{B}}(x)\subseteq \overline{\operatorname{orb}(z,  f)}$ and $\omega_{\underline{B}}(x)$ is nonempty,   closed and $f$-invariant,   $\omega_{\underline{B}}(x)$ is minimal,   which completes the proof.
	\end{proof}	
	Remark that if $\omega_{\underline{B}}(x)\neq \emptyset,$ then by Corollary \ref{cor-X-B-lower-star-is-minimal} there is a unique minimal subset in $\omega_f(x)$ which is exactly  $\omega_{\underline{B}}(x)$.

Now, we give the proof of Theorem \ref{thm-density-basic-property}.
\subsection{Proof of Theorem \ref{thm-density-basic-property}. }
		(1) On one hand,   consider an arbitrary $y\in \omega_{\underline{d}}(x)$.   For any $\mu\in V_f(x)$,   let $m_k\to\infty$ be such that $\lim_{k\to\infty}\mathcal{E}_{m_k}(x)=\mu$.   Then for any $\varepsilon>0$,   one has
		\begin{eqnarray*}
			\mu(B(y,4\varepsilon))\geq \mu(\overline{B(y,2\varepsilon)}) &\geq& \limsup_{k\to\infty}\mathcal{E}_{m_k}(\overline{B(y,2\varepsilon)}) \\
			&=&  \limsup_{k\to\infty}\frac{1}{m_k}\sum_{i=0}^{m_k-1}\delta_{f^i(x)}(\overline{B(y,2\varepsilon)})\\
			&\geq&  \liminf_{n\to\infty}\frac1n\sum_{j=0}^{n-1}\delta_{f^j(x)}(\overline{B(y,2\varepsilon)})\\
			&\geq&  \liminf_{n\to\infty}\frac1n\sum_{j=0}^{n-1}\delta_{f^j(x)}(B(y,\varepsilon))>0.
		\end{eqnarray*}
		This implies that $y\in S_{\mu}$ and thus $\omega_{\underline{d}}(x)\subseteq \bigcap_{\mu\in V_f(x)} S_\mu$.
		
		On the other hand,   consider an arbitrary $y\in\bigcap_{\mu\in V_f(x)} S_\mu$.   For any $\varepsilon>0$,   let $n_k\to\infty$ be such that
		$$\lim_{k\to\infty}\frac{1}{n_k}\sum_{i=0}^{n_k-1}\delta_{f^i(x)}(B(y,\varepsilon))=
		\liminf_{n\to\infty}\frac1n\sum_{j=0}^{n-1}\delta_{f^j(x)}(B(y,\varepsilon)).  $$
		Then choose a subsequence $n_{k_l}$ of $n_k$ such that $\lim_{l\to\infty}\mathcal{E}_{n_{k_l}}(x)=\tau$ for some $\tau\in V_f(x)$.   Then
		$$\liminf_{n\to\infty}\frac1n\sum_{j=0}^{n-1}\delta_{f^j(x)}(B(y,\varepsilon))=\lim_{k\to\infty}\frac{1}{n_k}\sum_{i=0}^{n_k-1}\delta_{f^i(x)}(B(y,\varepsilon))=\lim_{l\to\infty}\frac{1}{n_{k_l}}\sum_{i=0}^{n_{k_l}-1}\delta_{f^i(x)}(B(y,\varepsilon))
		\geq\tau(B(y,\varepsilon))>0.  $$
		Therefore,   $y\in \omega_{\underline{d}}(x)$ and thus $\omega_{\underline{d}}(x)\supseteq \bigcap_{\mu\in V_f(x)} S_\mu$.

		(2)  On one hand,   consider an arbitrary $y\in \omega_{\overline{d}}(x)$.   Then for any $\varepsilon>0$,   one has
		$$\limsup_{n\to\infty}\frac1n\sum_{i=0}^{n-1}\delta_{f^i(x)}(B(y,\varepsilon))>0.  $$
		Now choose a sequence $n_k\to\infty$ such that
		$$\lim_{k\to\infty}\frac{1}{n_k}\sum_{i=0}^{n_k-1}\delta_{f^i(x)}(\overline{B(y,2\varepsilon)})=
		\limsup_{n\to\infty}\frac1n\sum_{j=0}^{n-1}\delta_{f^j(x)}(\overline{B(y,2\varepsilon)}).  $$
		Then choose a subsequence $n_{k_l}$ of $n_k$ such that $\lim_{l\to\infty}\mathcal{E}_{n_{k_l}}(x)=\tau$ for some $\tau\in V_f(x)$.   Then
		\begin{eqnarray*}
			\tau(B(y,4\varepsilon))\geq\tau(\overline{B(y,2\varepsilon)}) &\geq& \lim_{l\to\infty}\frac{1}{n_{k_l}}\sum_{i=0}^{n_{k_l}-1}\delta_{f^i(x)}(\overline{B(y,2\varepsilon)}) \\
			&=&  \lim_{k\to\infty}\frac{1}{n_k}\sum_{i=0}^{n_k-1}\delta_{f^i(x)}(\overline{B(y,2\varepsilon)})\\
			&=&  \limsup_{n\to\infty}\frac1n\sum_{j=0}^{n-1}\delta_{f^j(x)}(\overline{B(y,2\varepsilon)})\\
			&\geq&  \limsup_{n\to\infty}\frac1n\sum_{j=0}^{n-1}\delta_{f^j(x)}(B(y,\varepsilon))>0.
		\end{eqnarray*}
		Therefore,   $B(y,4\varepsilon)\cap S_{\tau}\neq\emptyset$.   So for each $k\in\mathbb{N}$,   we obtain a $y_k\in B(y,1/k)\cap S_{\mu_k}$ for some $\mu_k\in V_f(x)$.   Thus $y_k\to y$ as $k\to\infty$,   which implies that $y\in \overline{\bigcup_{\mu\in V_f(x)} S_\mu}$ and thus $\omega_{\overline{d}}(x)\subseteq\overline{\bigcup_{\mu\in V_f(x)} S_\mu}$.
		
		On the other hand,   consider an arbitrary $y\in \overline{\bigcup_{\mu\in V_f(x)} S_\mu}$.   For each $k\in\mathbb{N}$,   choose a $y_k\in B(y,1/k)\cap S_{\mu_k}$ for some $\mu_k\in V_f(x)$.   Then $B(y,1/k)$ is a neighborhood of $y_k\in S_{\mu_k}$.   This implies that
		$\mu_k(B(y,1/k))>0$.   Since $\mu_k\in V_f(x)$,   we choose a sequence $k_l\to\infty$ such that $\lim_{l\to\infty}\mathcal{E}_{k_l}(x)=\mu_k$.   Then one sees that
		$$\limsup_{n\to\infty}\frac1n\sum_{i=0}^{n-1}\delta_{f^i(x)}(B(y,1/k))
		\geq\lim_{l\to\infty}\frac{1}{k_l}\sum_{i=0}^{k_l-1}\delta_{f^i(x)}(B(y,1/k))\geq\mu_k(B(y,1/k))>0.  $$
		Since $k\in\mathbb{N}$ is arbitrary,   one has
		$y\in \omega_{\overline{d}}(x)$.   Since $y\in \overline{\bigcup_{\mu\in V_f(x)} S_\mu}$ is arbitrary,   we see $\omega_{\overline{d}}(x)\supseteq\overline{\bigcup_{\mu\in V_f(x)} S_\mu}$.

		(3) By Corollary \ref{cor-X-B-lower-star-is-minimal}, for any $x\in X$ $\omega_{\underline{B}}(x)$ is either empty or a minimal set. By Lemma \ref{lemma-V*-contains-ergodic}, $M_{erg}(f,  \omega_f(x))\subseteq V^*_f(x)\subseteq M(f,  \omega_f(x)).$ Thus, by  Lemma \ref{lem-cap-cup} $$ \bigcap_{\mu\in V^*_f(x)} S_\mu =  {\bigcap_{\mu\in M(f,  \omega_f(x))}S_{\mu}}= {\bigcap_{\mu\in M_{erg}(f,  \omega_f(x))}S_{\mu}}.  $$  So we only need to prove $\omega_{\underline{B}}(x)= \bigcap_{\mu\in V^*_f(x)} S_\mu$.

		On one hand,   consider an arbitrary $y\in \omega_{\underline{B}}(x)$.   For any $\mu\in V_f^*(x)$,   let $I_k=[a_k,  b_k)$ be a sequence of positive integer intervals with $\lim_{k\to\infty}b_k-a_k=\infty$ such that
		$$\lim_{k\rightarrow\infty}\frac{1}{b_k-a_k}\sum_{j=a_k}^{b_k-1}\delta_{f^j(x)}=\mu.  $$
		Then for any $\varepsilon>0$,   one has
		\begin{eqnarray*}
			\mu(B(y,4\varepsilon))\geq\mu(\overline{B(y,2\varepsilon)}) &\geq& \limsup_{k\to\infty}\frac{1}{b_k-a_k}\sum_{j=a_k}^{b_k-1}\delta_{f^j(x)}(\overline{B(y,2\varepsilon)}) \\
			&=& \limsup_{k\to\infty}\frac{|N(x,  \overline{B(y,2\varepsilon)})\cap I_k|}{|I_k|} \\
			&\geq& \liminf_{|I|\to\infty}\frac{|N(x,  \overline{B(y,2\varepsilon)})\cap I|}{|I|}\\
			&\geq& \liminf_{|I|\to\infty}\frac{|N(x,  B(y,\varepsilon))\cap I|}{|I|}>0.
		\end{eqnarray*}
		This implies that $y\in S_{\mu}$ and thus $\omega_{\underline{B}}(x)\subseteq \bigcap_{\mu\in V_f^*(x)} S_\mu$.
		
		On the other hand,   consider an arbitrary $y\in \bigcap_{\mu\in V_f^*(x)} S_\mu$.   Let $I_k=[a_k,  b_k)$ be a sequences of positive integer intervals with $\lim_{k\to\infty}b_k-a_k=\infty$ such that
		$$\lim_{k\to\infty}\frac{|N(x,  B(y,\varepsilon))\cap I_k|}{|I_k|}
		=\liminf_{|I|\to\infty}\frac{|N(x,  B(y,\varepsilon))\cap I|}{|I|}.  $$
		Then choose a subsequence $I_{k_l}$ of $I_k$ such that $\lim_{l\rightarrow\infty}\frac{1}{b_{k_l}-a_{k_l}}\sum_{j=a_{k_l}}^{b_{k_l}-1}\delta_{f^j(x)}=\tau$ for some $\tau\in V_f^*(x)$.   We have
		\begin{eqnarray*}
			\liminf_{|I|\to\infty}\frac{|N(x,  B(y,\varepsilon))\cap I|}{|I|} &=& \lim_{k\to\infty}\frac{|N(x,  B(y,\varepsilon))\cap I_k|}{|I_k|} \\
			&=& \lim_{l\to\infty}\frac{|N(x,  B(y,\varepsilon))\cap I_{k_l}|}{|I_{k_l}|}\\
			&=& \lim_{l\rightarrow\infty}\frac{1}{b_{k_l}-a_{k_l}}\sum_{j=a_{k_l}}^{b_{k_l}-1}\delta_{f^j(x)}(B(y,\varepsilon)) \\
			&\geq& \tau(B(y,\varepsilon))>0.
		\end{eqnarray*}
		Therefore,   $y\in \omega_{\underline{B}}(x)$ and thus $\omega_{\underline{B}}(x)\supseteq \bigcap_{\mu\in V^*_f(x)} S_\mu$.

		(4) By Lemma \ref{lemma-V*-contains-ergodic},
		$M_{erg}(f,  \omega_f(x))\subseteq V^*_f(x)\subseteq M(f,  \omega_f(x)).  $
		Thus, by  Lemma \ref{lem-cap-cup}
		$$ \overline{\bigcup_{\mu\in V^*_f(x)} S_\mu}=
		\overline{\bigcup_{\mu\in M(f,  \omega_f(x))}S_{\mu}}=\overline{\bigcup_{\mu\in M_{erg}(f,  \omega_f(x))}S_{\mu}}\neq \emptyset.$$
		So we only need to prove $\omega_{\overline{B}}(x)= \overline{\bigcup_{\mu\in V^*_f(x)} S_\mu}$.
		
		On one hand,   consider an arbitrary $y\in \omega_{\overline{B}}(x)$.   Then for any $\varepsilon>0$,   one has
		$$\limsup_{|I|\to\infty}\frac{|N(x,  B(y,\varepsilon))\cap I|}{|I|}>0.  $$
		Now choose a sequence of positive integer intervals $I_k=[a_k,  b_k)$ with $b_k-a_k\to\infty$ such that
		$$\lim_{k\to\infty}\frac{|N(x,  B(y,\varepsilon))\cap I_k|}{|I_k|}=
		\limsup_{|I|\to\infty}\frac{|N(x,  B(y,\varepsilon))\cap I|}{|I|}.  $$
		Then choose a subsequence $I_{k_l}$ of $I_k$ such that $\lim_{l\rightarrow\infty}\frac{1}{b_{k_l}-a_{k_l}}\sum_{j=a_{k_l}}^{b_{k_l}-1}\delta_{f^j(x)}=\tau$ for some $\tau\in V_f^*(x)$.   We have
		\begin{eqnarray*}
			\tau(B(y,4\varepsilon))\geq\tau(\overline{B(y,2\varepsilon)}) &\geq& \lim_{l\rightarrow\infty}\frac{1}{b_{k_l}-a_{k_l}}\sum_{j=a_{k_l}}^{b_{k_l}-1}\delta_{f^j(x)}(\overline{B(y,2\varepsilon)}) \\
			&=& \lim_{k\to\infty}\frac{1}{b_k-a_k}\sum_{i=a_k}^{b_k-1}\delta_{f^i(x)}(\overline{B(y,2\varepsilon)}) \\
			&=& \lim_{k\to\infty}\frac{|N(x,  \overline{B(y,2\varepsilon)})\cap I_k|}{|I_k|}\\
			&=& \limsup_{|I|\to\infty}\frac{|N(x,  \overline{B(y,2\varepsilon)})\cap I|}{|I|}\\
			&\geq& \limsup_{|I|\to\infty}\frac{|N(x,  B(y,\varepsilon))\cap I|}{|I|}>0.
		\end{eqnarray*}
		Therefore,   $B(y,4\varepsilon)\cap S_{\tau}\neq\emptyset$.   So for each $k\in\mathbb{N}$,   we obtain a $y_k\in B(y,1/k)\cap S_{\mu_k}$ for some $\mu_k\in V_f(x)$.   Thus $y_k\to y$ as $k\to\infty$,   which implies that $y\in \overline{\bigcup_{\mu\in V_f(x)} S_\mu}$ and thus $\omega_{\overline{B}}(x)\subseteq\overline{\bigcup_{\mu\in V_f^*(x)} S_\mu}$.
		
		On the other hand,   consider an arbitrary $y\in \overline{\bigcup_{\mu\in V_f^*(x)} S_\mu}$.   For each $k\in\mathbb{N}$,   choose a $y_k\in B(y,1/k)\cap S_{\mu_k}$ for some $\mu_k\in V_f^*(x)$.   Then $B(y,1/k)$ is a neighborhood of $y_k\in S_{\mu_k}$.   This implies that
		$\mu_k(B(y,1/k))>0$.   Since $\mu_k\in V_f^*(x)$,   we choose a sequence $k_l\to\infty$ such that $$\lim_{l\rightarrow\infty}\frac{1}{b_{k_l}-a_{k_l}}\sum_{j=a_{k_l}}^{b_{k_l}-1}\delta_{f^j(x)}=\mu_k.  $$ Then one sees that
		\begin{eqnarray*}
			\limsup_{|I|\to\infty}\frac{|N(x,  B(y,1/k))\cap I|}{|I|} &\geq& \lim_{l\to\infty}\frac{|N(x,  B(y,1/k))\cap I_{k_l}|}{|I_{k_l}|} \\
			&=& \lim_{l\rightarrow\infty}\frac{1}{b_{k_l}-a_{k_l}}\sum_{j=a_{k_l}}^{b_{k_l}-1}\delta_{f^j(x)}(B(y,1/k)) \\
			&\geq& \mu_k(B(y,1/k))>0.
		\end{eqnarray*}
		Therefore,   for any $\varepsilon>0$,   let $k\in\mathbb{N}$ be such that $0<1/k<\varepsilon$,   we see that
		$$\limsup_{|I|\to\infty}\frac{|N(x,  B(y,\varepsilon))\cap I|}{|I|}\geq\limsup_{|I|\to\infty}\frac{|N(x,  B(y,1/k))\cap I|}{|I|}>0$$
		This indicates that $y\in \omega_{\overline{B}}(x)$ and thus $\omega_{\overline{B}}(x)\supseteq\overline{\bigcup_{\mu\in V_f^*(x)} S_\mu}$.

		(5) By ergodic decomposition Theorem,   we only prove
		the case that $\mu$ is ergodic.   By ergodicity,   for $\mu$-a.e.   $x\in X,   $
		any $y\in S_\mu$ and  any $\varepsilon>0,  \,  N (x,  B(y,\varepsilon))$
		has positive   density  (equal to   $\mu(B(y,\varepsilon))$)  w.  r.  t.
		$\underline{d}$ and  $\overline{d}.  $ So for $\mu$-a.e.   $x\in X,   $ $S_\mu\subseteq \omega_{\underline{d}}(x)$.   By ergodicity,   we also have that for $\mu$-a.e.   $x\in X,   $
		$S_\mu=\omega_f(x)=\overline{\operatorname{orb}(x,f)}.  $  So we conclude the result.   \qed

\section{The existence of 50 cases: Proof of Theorem \ref{Theorem A-new}}\label{section 4 new}
In this section, we will prove Theorem \ref{Theorem A-new} assuming Theorem \ref{Theorem B}. Fix $m\geq2$, we consider the one-sided full shift $(\Sigma_m^+,\sigma)$, it is clear that $(\Sigma_m^+,\sigma)$ is topologically transitive topologically expanding and thus by Theorem \ref{Theorem B}, for any $\beta_1\beta_2\beta_3\beta_4\beta_50\in\mathfrak{A}$, there always exists a point in $(\Sigma_m^+,\sigma)$ satisfying \textbf{Case }($\beta_1\beta_2\beta_3\beta_4\beta_50$). Hence, to show Theorem \ref{Theorem A-new}, it is enough to show that for any $0\alpha_2\alpha_3\alpha_401\in\mathfrak{A}$, there exists a point in $(\Sigma_m^+,\sigma)$ satisfying \textbf{Case }($0\alpha_2\alpha_3\alpha_401$). We need the following results obtained in \cite{Jiang-Tian-2024}.
\begin{Lem}\label{Lemma-map-T-1}\cite[Proposition 3.1 and Proposition 3.2]{Jiang-Tian-2024}
	Given $x\in(\Sigma_m^+,\sigma)$ with $\overline{\operatorname{orb}(x,\sigma)}\subsetneq\Sigma_m^+$, there exists $y\in\Sigma_m^+$ such that
	\begin{enumerate}[(1)]
		\item $\omega_{\overline{B}}(x)\subset\omega_{\overline{B}}(y)\subset\overline{\operatorname{orb}(x,\sigma)}\subsetneq\omega_\sigma(y)=\overline{\operatorname{orb}(y,\sigma)}$;
		\item $\omega_{\overline{d}}(x)=\omega_{\overline{d}}(y)$ and $\omega_{\underline{d}}(x)=\omega_{\underline{d}}(y)$.
	\end{enumerate}
\end{Lem}
\begin{Lem}\label{Lemma-map-T-2}
	Given $x\in(\Sigma_m^+,\sigma)$ with $\overline{\operatorname{orb}(x,\sigma)}\subsetneq\Sigma_m^+$, let $y\in\Sigma_m^+$ be the point obtained by Lemma \ref{Lemma-map-T-1}, then $\omega_{\underline{B}}(x)=\omega_{\underline{B}}(y)$.
\end{Lem}
\begin{proof}
	By Theorem \ref{thm-density-basic-property}(4), we have that $$\omega_{\overline{B}}(x)=\overline{\bigcup_{\mu\in M(\sigma,  \omega_\sigma(x))}S_\mu}\text{ and }\omega_{\overline{B}}(y)=\overline{\bigcup_{\mu\in M(\sigma,  \omega_\sigma(y))}S_\mu}.$$
	Since $M(\sigma,\omega_\sigma(x))=M(\sigma,\overline{\operatorname{orb}(x,\sigma)})$, by Lemma \ref{Lemma-map-T-1}(1), one has $\omega_{\overline{B}}(y)\subset\omega_\sigma(x)$ and $\omega_\sigma(x)\subset\omega_\sigma(y)$, as a result, $M(\sigma,\omega_\sigma(x))=M(\sigma,\omega_\sigma(y))$. Finally, by Theorem \ref{thm-density-basic-property}(3), we have that $\omega_{\underline{B}}(x)=\omega_{\underline{B}}(y)$.
\end{proof}
Now, given $0\alpha_2\alpha_3\alpha_401\in\mathfrak{A}$, we choose $x\in\Sigma_m^+$ satisfying \textbf{Case }$(0\alpha_2\alpha_3\alpha_410)$, then it is clear that $\overline{\operatorname{orb}(x,\sigma)}\subsetneq\Sigma_m^+$. By Lemma \ref{Lemma-map-T-1} and Lemma \ref{Lemma-map-T-2}, we can find $y$ satisfying \textbf{Case }$(0\alpha_2\alpha_3\alpha_401)$. Now we complete the proof of Theorem \ref{Theorem A-new}.

\section{Saturated set and entropy estimation}\label{section 6}
As suggested by Theorem \ref{thm-density-basic-property}, the study of $V_f(x)$ and $\omega(x)$ will help us judge the statistical behaviour of $x$. An important tool for this is saturated set. Let $(X,  f)$ be a topological dynamical system.   Recall that, for $x\in X$, $V_f(x)$ is  a nonempty compact connected subset of $M(f,  X)$. So for any nonempty compact connected subset $K$ of $M(f,  X)$,   it is logical to define the following set
$$G_K=G_K(f):=\{x\in X:V_f(x)=K\}.  $$
We call $G_K$ the saturated set of $K$.   Particularly,   if $K=\{\mu\}$ for some ergodic measure $\mu$,   then $G_{\mu}=G_{\{\mu\}}$ is just the generic points of $\mu$.   Saturated sets are studied by Pfister and Sullivan in \cite{PS2007}.
Furthermore,   we define  $G^T_{K}=\{x\in Tran(f): V_f(x)=K\}$  and call $G^T_{K}$ the transitively-saturated set of $K$.
\begin{Def}
	We say that  the system $f$ has {\it saturated} property or $f$ is {\it saturated},   if  for any  nonempty compact connected set $K \subseteq M (f,  X ),  $
	\begin{eqnarray} \label{eq-saturated-definition}
		G_K\neq\emptyset \text{ and } h_{top}(f,  G_K)=\inf\{h_\mu (f):  \mu\in K\}.
	\end{eqnarray}
	We say that  the system $f$ has {\it locally-saturated} property or $f$ is {\it locally-saturated},   if  for any nonempty open set $U\subseteq X$ and any  nonempty compact connected set $K \subseteq M (T,  X ),  $
	\begin{eqnarray} \label{eq-saturated-definition-locally}
		G_K\cap U\neq\emptyset \text{ and } h_{top} (f,  G_K\cap U)=\inf\{h_\mu (f): \mu\in K\}.
	\end{eqnarray}
	In parallel,   one can define  {\it(locally-)transitively-saturated},   just replacing $ G_K$ by $ G^T_K $ in (\ref{eq-saturated-definition}) (respectively,   (\ref{eq-saturated-definition-locally})).   On the other hand,   (locally-,   transitively-)single-saturated means (\ref{eq-saturated-definition}) (respectively,   (\ref{eq-saturated-definition-locally})) holds when $K$ is a singleton.
\end{Def}
When the dynamical system $f$ satisfies $g$-almost product and uniform separation property,   Pfister and Sullivan proved in \cite{PS2007} that $f$ is saturated, and Huang,  Tian and Wang proved in \cite{HTW} that $f$ is transitively-saturated if further there is an invariant measure with full support.

Let $\Lambda\subseteq X$ be a closed invariant subset and $K$ is a nonempty compact connected subset of $M(f,  \Lambda)$.   Define
$$G_K^{\Lambda}:=G_K\cap\{x\in X:\omega_f(x)=\Lambda\}.$$

\begin{Def}
	(1)  We say that $\Lambda$ is star-saturated,   if for any nonempty connected compact set $K\subseteq M(f,  \Lambda)$,   one has
	\begin{eqnarray*}
		G_K^{\Lambda}\neq\emptyset \text{ and } h_{top} (f,  G_K^{\Lambda})=\inf\{h_\mu (f):  \mu\in K\}.
	\end{eqnarray*}
	We say that  the system $f$ has {\it star-saturated} property or $f$ is {\it star-saturated},
	if for any internally chain transitive closed invariant subset $\Lambda\subseteq X$,
	$\Lambda$ is star-saturated.
	
	(2)   We say that    $\Lambda$ is {  locally-star-saturated},   if for any nonempty open set $U\subseteq X$ and any  compact connected nonempty set $K \subseteq M (f,  \Lambda ),  $
	\begin{eqnarray*} 
		G_K^{\Lambda}\cap U\neq\emptyset \text{ and } h_{top} (f,  G_K^{\Lambda}\cap U)=\inf\{h_\mu (f):  \mu\in K\}.
	\end{eqnarray*}
	We say that  the system $f$ has {\it locally-star-saturated} property or $f$ is {\it locally-star-saturated},   if  for any internally chain transitive closed invariant subset $\Lambda\subseteq X$,
	$\Lambda$ is locally star-saturated.
\end{Def}

Note that $X$ is  star-saturated $\Leftrightarrow$ $f$ is transitively-saturated,   and for a closed invariant subset $\Lambda\subsetneq X,  $   $f|_\Lambda$ is transitively-saturated $\Rightarrow$ $\Lambda$ is  star-saturated.   However,   remark that $\Lambda$ is  star-saturated   does not imply $f|_\Lambda$ is transitively-saturated,   since some points of $G_K^{\Lambda}$ may not lie in $\Lambda$.

Now we state one result on locally-star-saturated property.
\begin{Thm}\label{G-K-star}
	Suppose $(X,  f)$ is topologically transitive and satisfies s-limit-shadowing property and uniform separation property. 
	Then $f$ is locally-star-saturated.
	\end{Thm}}
	
To prove Theorem \ref{G-K-star}, we need some preparations. First of all, let's recall a result from Bowen \cite{Bowen1973}, which will help us obtain the upper bound of the topological entropy of $G_K$.
\begin{Lem}\cite[Theorem 2]{Bowen1973}\label{lem-Bowen}
	Let $f:X\rightarrow X$ be a continuous map on a compact metric space.   Set
	$$QR(t)=\{x\in X~|~\text{there exists }\tau\in V_{f}(x)~\textrm{with}~h_{\tau}(f)\leq t\}.  $$
	Then
	$h_{top}(f,  QR(t))\leq t.  $
\end{Lem}
\begin{Lem}\label{Mainlemma-prop-mu-gamma-n}
	Suppose $(X,  f)$ satisfies the shadowing property.   Let $\Lambda\subseteq X$ be closed $f$-invariant and internally chain transitive.   Then for any $\eta>0$, any $\alpha>0$, any $\mu\in M(f,  \Lambda)$ and its neighborhood $F_{\mu}$,    there exist $\varepsilon^*_1=\varepsilon^*_1(\eta,  \mu, F_{\mu})> \varepsilon^*_2=\varepsilon^*_2(\eta,  \mu, F_{\mu})>0$ such that for any $0<\varepsilon\leq\varepsilon^*_2$, for any $x\in\Lambda$,  for any $N\in\mathbb{N}$,   there exist $n=n(\eta,  \mu, F_{\mu}, \varepsilon,x)\geq N$ such that for any $p\in\mathbb{N}$,   there exists a $(pn,  \frac{\varepsilon^*_1}{3})$-separated set $\Gamma_{pn}$ so that
	\begin{description}
		\item[(a)] $\Gamma_{pn}\subseteq X_{pn,  F_{\mu}}\cap B(x,  \varepsilon)\cap f^{-pn}B(x,  \varepsilon)$;
		\item[(b)] $\frac{\log |\Gamma_{pn}|}{pn}>\min\{h_{\mu},\alpha\}-\eta$;
		\item[(c)] $d_H(\{f^i(y)\}_{i=0}^{pn-1},  \Lambda)<2\varepsilon$ for any $y\in\Gamma_{pn}.$
	\end{description}
\end{Lem}
\begin{proof}
	Since $F_{\mu}$ is a neighborhood of $\mu$,   there is an $a>0$ such that $\mathcal{B}(\mu,  a)\subseteq F_{\mu}$.   By the ergodic decomposition theorem,   there exists a finite convex combination of ergodic measures $\sum_{i=1}^mc_i\nu_i\in \mathcal{B}(\mu,  \frac a4)$ with $h_{\sum_{i=1}^mc_i\nu_i}>\min\{h_{\mu},\alpha\}-\frac{\eta}{3}$ and $S_{\nu_i}\subset \Lambda$, see also \cite[Lemma 2.5]{Li-Oprocha-2018}.   Moreover,   by the denseness of the rational numbers,   we can choose each $c_i=\frac{b_i}{b}$ with $b_i\in\mathbb{N}$ and $\sum_{i=1}^mb_i=b$.
	
	By Lemma \ref{useful-Proposition},   for each $i$,   there exist $\varepsilon_i>0$ and $n_i\in\mathbb{N}$ such that for any $n\geq n_i$,   there exists an $(n,  \varepsilon_i)$-separated set $\Gamma_n^{\nu_i}\subseteq X_{n,  \mathcal{B}(\nu_i,  \frac{a}{4})}\cap\Lambda$ such that $\frac{\log|\Gamma_n^{\nu_i}|}{n}>h_{\nu_i}-\frac{\eta}{3}$.   Let $\varepsilon^*_1=\min\{\varepsilon_i:1\leq i\leq m\}$ and $\varepsilon^*_2=\min\{\frac{\varepsilon^*_1}{4},  \frac{a}{4}\}$.   By the shadowing property,   for any $0<\varepsilon\leq\varepsilon^*_2$,   there exists a $0<\delta<\varepsilon$ such that any $\delta$-pseudo-orbit can be $\varepsilon$-shadowed.   Moreover,   since $\Lambda$ is compact,   we can cover $\Lambda$ by finite number of open balls $\{B(x_i,  \frac{\delta}{2})\}_{i=1}^l$ with $\{x_i\}_{i=1}^l\subseteq\Lambda$.   Since $\Lambda$ is internally chain transitive,   for each $1\leq i\leq l$,   there exist a $\frac{\delta}{2}$-chain $\mathfrak{C}_{xx_{i}}$ with length $l_{0i}$ connecting $x$ and $x_i,$ and a $\frac{\delta}{2}$-chain $\mathfrak{C}_{x_ix}$ with length $l_{i0}$ connecting $x_i$ and $x$.   Let $L=\max\{l_{0i},  l_{i0}:1\leq i\leq l\}$.   Now choose $k\in\mathbb{N}$ large enough such that
	
	\begin{equation}\label{k-Gamma}
		kb_i\geq n_i~\textrm{for}~1\leq i\leq m~\textrm{and}~kb\geq N;
	\end{equation}
	\begin{equation*}
		\frac{2(m+l)L}{kb+2(m+l)L}\cdot2<\frac{a}{4};
	\end{equation*}
	\begin{equation}\label{entropy}
		\frac{kb(\min\{h_{\mu},\alpha\}-\frac{2}{3}\eta)-2m\log l}{kb+2(m+l)L}>\min\{h_{\mu},\alpha\}-\eta.
	\end{equation}

By \eqref{k-Gamma},   for each $1\leq i\leq m$,   we obtain a $(kb_i,  \varepsilon_i)$-separated set $\Gamma_{kb_i}^{\nu_i}\subseteq X_{kb_i,  \mathcal{B}(\nu_i,  \frac{a}{4})}\cap\Lambda$ such that $\frac{\log|\Gamma_{kb_i}^{\nu_i}|}{kb_i}>h_{\nu_i}-\frac{\eta}{3}$.   Moreover,   by the pigeonhole principle,   there exist $1\leq i_1,  i_2\leq l$ and $\widetilde{\Gamma}_{kb_i}^{\nu_i}\subset \Gamma_{kb_i}^{\nu_i}$ such that $\widetilde{\Gamma}_{kb_i}^{\nu_i}\subseteq B(x_{i_1}, \frac{\delta}{2}),$ $f^{kb_i}\widetilde{\Gamma}_{kb_i}^{\nu_i}\subseteq B(x_{i_2},\frac{\delta}{2})$ and
\begin{equation}\label{pigeonhole}
	|\widetilde{\Gamma}_{kb_i}^{\nu_i}|\geq \frac{|\Gamma_{kb_i}^{\nu_i}|}{l^2}.
\end{equation}

Now let $\Gamma=\widetilde{\Gamma}_{kb_1}^{\nu_1}\times\widetilde{\Gamma}_{kb_2}^{\nu_2}\times\cdots\times\widetilde{\Gamma}_{kb_m}^{\nu_m}$ whose element is $\underline{y}=(y_1,  \cdots,  y_m)$ with $y_i\in\widetilde{\Gamma}_{kb_i}^{\nu_i}$,   $1\leq i\leq m$.   Then we define the following pseudo-orbit:
$$\mathfrak{C}_{\underline{y}}=\mathfrak{C}_{xx_{1_1}}\langle y_1,  \cdots,  f^{kb_1-1}y_1\rangle\mathfrak{C}_{x_{1_2}x}\cdots \mathfrak{C}_{xx_{m_1}}\langle y_m,  \cdots,  f^{kb_m-1}y_m\rangle
\mathfrak{C}_{x_{m_2}x}  \mathfrak{C}_{xx_1}  \mathfrak{C}_{x_1x} \cdots  \mathfrak{C}_{xx_l}  \mathfrak{C}_{x_lx}.$$
It is clear that $\mathfrak{C}_{\underline{y}}$ is a $\delta$-pseudo-orbit.   Moreover,   one notes that we can freely concatenate such $\mathfrak{C}_{\underline{y}}$s to constitutes a $\delta$-pseudo-orbit.   So if we let $\Gamma^p=\Gamma\times\cdots\times\Gamma$ whose element is $\theta=(\theta_1,  \cdots,  \theta_p)$ with $\theta_j\in\Gamma$,   $1\leq j\leq p$,   we can define the following two $\delta$-pseudo-orbits: $$\widetilde{\mathfrak{C}}_{\theta}=\mathfrak{C}_{\theta_1}\mathfrak{C}_{\theta_2}\cdots\mathfrak{C}_{\theta_p}~\textrm{and}~
\mathfrak{C}_{\theta}=\widetilde{\mathfrak{C}}_{\theta}\widetilde{\mathfrak{C}}_{\theta}\cdots\widetilde{\mathfrak{C}}_{\theta}\cdots.  $$
Hence $\mathfrak{C}_{\theta}$ can be $\varepsilon$-shadows by some point in $X$.   Therefore,   if we let
$$n=kb+\sum_{i=1}^m(l_{0i_1}+l_{i_20})+\sum_{j=1}^l(l_{0j}+l_{j0}),  $$
we can define the following nonempty set
\begin{equation*}
	\Gamma_{pn}:=\{y\in X: y~\varepsilon-\textrm{shadows some pseudo orbit}~\mathfrak{C}_{\theta}~\textrm{with}~\theta\in\Gamma^p\}.	
\end{equation*}
By \eqref{k-Gamma},   $n\geq N$.   To see that $\Gamma_{pn}$ is a $(pn,  \frac{\varepsilon^*_1}{3})$-separated set,   consider any two points $y,  y'\in\Gamma_{pn}$.   Suppose $y$ $\varepsilon$-shadows $\mathfrak{C}_{\theta}$ and $y'$ $\varepsilon$-shadows $\mathfrak{C}_{\theta'}$ with $\theta_i\neq \theta'_i$ for some $1\leq i\leq p$.   Suppose that $\theta_i=(y_1,  \cdots,  y_m)$ and $\theta'_i=(y'_1,  \cdots,  y'_m)$.   Then there exists a $1\leq j\leq m$ such that $y_j\neq y'_j$.   This implies that $d_{pn}(y,  y')\geq \varepsilon_j-2\varepsilon\geq\varepsilon^*_1-2\varepsilon^*_2> \frac{\varepsilon^*_1}{3}.$   Now let us prove that $\Gamma_{pn}$ satisfies (a)(b)(c).

(a) For any $y\in\Gamma_{pn}$,   $y$ $\varepsilon$-shadows some $\mathfrak{C}_{\theta}=\langle w_u\rangle_{u=0}^{\infty}$ with $\theta=(\theta_1,  \cdots,  \theta_p)$ and $\theta_i=(\theta_{i,  1},  \cdots,  \theta_{i,  m})$,   $\theta_{i,  j}\in\widetilde{\Gamma}_{kb_j}^{\nu_j}$,   $1\leq i\leq p$,   $1\leq j\leq m$.   We have the following estimations:
\begin{enumerate}
	\item $\rho\left(\sum_{j=1}^m\frac{b_j}{b}\nu_j,  \mu\right)<\frac{a}{4}$.
	\item $\rho\left(\sum_{j=1}^m\frac{b_j}{b}\mathcal{E}_{kb_j}(\theta_{i,  j}),  \sum_{j=1}^m\frac{b_j}{b}\nu_j\right)<\frac{a}{4}$ for each $1\leq i\leq p$.
	\item $\rho\left(\frac{1}{n}\sum_{t=(i-1)n}^{in-1}\delta_{w_t},  \sum_{j=1}^m\frac{b_j}{b}\mathcal{E}_{kb_j}(\theta_{i,  j})\right)\leq \frac{n-kb}{kb+n-kb}\cdot2\leq \frac{2(m+l)L}{kb+2(m+l)L}\cdot2<\frac{a}{4}$ for each $1\leq i\leq p$.
	\item $\rho\left(\mathcal{E}_n(f^{(i-1)n}(y)),  \frac{1}{n}\sum_{t=(i-1)n}^{in-1}\delta_{w_t}\right)<\varepsilon\leq \frac{a}{4}$ for each $1\leq i\leq p$.
\end{enumerate}
Item 3 follows from Lemma \ref{lem:prohorov}.
So we have $\rho(\mathcal{E}_n(f^{(i-1)n}(y)),  \mu)<\frac{a}{4}+\frac{a}{4}+\frac{a}{4}+\frac{a}{4}=a$ for each $1\leq i\leq p$.   Therefore,
$$\rho(\mathcal{E}_{pn}(y),  \mu)=\rho\left(\frac{1}{p}\sum_{i=1}^p\mathcal{E}_n(f^{(i-1)n}(y)),  \mu\right)\leq \frac1p\sum_{i=1}^p\rho(\mathcal{E}_n(f^{(i-1)n}(y)),  \mu)<a,  $$
which implies that $y\in X_{pn,  \mathcal{B}(\mu,  a)}\subseteq X_{pn,  F_{\mu}}$.   Consequently,   $\Gamma_{pn}\subseteq X_{pn,  F_{\mu}}$.   The fact that $\Gamma_{pn}\subseteq B(x,  \varepsilon)\cap f^{-pn}B(x,  \varepsilon)$ is obvious by the construction.

(b) Note that $|\Gamma_{pn}|=\prod_{i=1}^m|\widetilde{\Gamma}_{kb_i}^{\nu_i}|^p\geq \prod_{i=1}^m\left(\frac{e^{kb_i(h_{\nu_i}-\frac{\eta}{3})}}{l^2}\right)^p$ by \eqref{pigeonhole}.   Then by the affinity of the metric entropy and \eqref{entropy},   we have
$$\frac{\log |\Gamma_{pn}|}{pn}\geq \frac{kb(\min\{h_{\mu},\alpha\}-\frac{2}{3}\eta)-2m\log l}{n}\geq \frac{kb(\min\{h_{\mu},\alpha\}-\frac{2}{3}\eta)-2m\log l}{kb+2(m+l)L}>\min\{h_{\mu},\alpha\}-\eta.  $$

(c) For any $y\in\Gamma_{pn}$,   $y$ $\varepsilon$-shadows some $\mathfrak{C}_{\theta}\subseteq \Lambda$ with $\theta\in\Gamma^p$.   Then $\{f^i(y)\}_{i=0}^{pn-1}\subseteq B(\Lambda,  \varepsilon)$.   Meanwhile,   by the construction of $\mathfrak{C}_{\theta}$,   for any $x_i$,   $1\leq i\leq l$,   there exists a point $z\in\{f^i(y)\}_{i=0}^{pn-1}$ which is $\varepsilon$ close to $x_i$.   However,   $\{x_i\}_{i=1}^l$ is $\frac{\delta}{2}$-dense in $\Lambda$.   So $\{f^i(y)\}_{i=0}^{pn-1}$ is $(\frac{\delta}{2}+\varepsilon)$-dense in $\Lambda$.   Since $\frac{\delta}{2}+\varepsilon<2\varepsilon$,   we have $\Lambda\subseteq B(\{f^i(y)\}_{i=0}^{pn-1},  2\varepsilon)$.
\end{proof}

\begin{Cor}\label{Maincor-mu-gamma-uniform} Suppose $(X,  f)$ satisfies shadowing property and uniform separation property.
Let $\Lambda\subseteq X$ be compact $f$-invariant and internally chain transitive.
Then for any $\eta>0$,   there exists an $\varepsilon^*_1=\varepsilon^*_1(\eta)>0$ such that for any $\mu\in M(f,  \Lambda)$ and its neighborhood $F_{\mu}$, there exists an $0<\varepsilon^*_2=\varepsilon^*_2(\eta,\mu,F_\mu)<\varepsilon^*_1$ such that  for any $x\in\Lambda$,   for any $0<\varepsilon\leq\varepsilon^*_2$,   for any $N\in\mathbb{N}$,   there exist an $n=n(\eta,\mu,F_\mu,  \varepsilon,x)\geq N$ such that for any $p\in\mathbb{N}$,   there exists an $(pn,  \frac{\varepsilon^*_1}{3})$-separated set $\Gamma_{pn}$ so that
\begin{description}
	\item[(1)] $\Gamma_{pn}\subseteq X_{pn,  F_{\mu}}\cap B(x,  \varepsilon)\cap f^{-pn}B(x,  \varepsilon)$;
	\item[(2)] $\frac{\log |\Gamma_{pn}|}{pn}>h_{\mu}-\eta$;
	\item[(3)] $d_H(\{f^i(y)\}_{i=0}^{pn-1},  \Lambda)<2\varepsilon$ for any $y\in\Gamma_{pn}.$
\end{description}	
\end{Cor}
\begin{proof}
Since $(X,  f)$ has uniform separation property,  it has finite topological entropy and $f|_\Lambda$ has uniform separation property.   Therefore, we can fix $\alpha=h_{top}(f)$ in the statement of Lemma \ref{Mainlemma-prop-mu-gamma-n} and 
the $\varepsilon^*_1$ in the proof of Lemma \ref{Mainlemma-prop-mu-gamma-n} can be chosen independent of $\mu$ and $F_{\mu}$ replacing using Lemma \ref{useful-Proposition} by  using uniform separation property of $f|_\Lambda$.
\end{proof}

For $k\in\mathbb{N}$,   denote $P_k(f):=\{x\in X:f^k(x)=x\}$.
\begin{Cor}\label{Maincor-mu-gamma-n-expansive} 
	Suppose that $(X,  f)$ is topologically expanding (resp. topologically Anosov).
	Let $\Lambda\subseteq X$ be compact $f$-invariant and internally chain transitive.
	Then for any $\eta>0$,   there exists an $\varepsilon^*_1=\varepsilon^*_1(\eta)>0$ such that for any $\mu\in M(f,  \Lambda)$ and its neighborhood $F_{\mu}$, there exists an $0<\varepsilon^*_2=\varepsilon^*_2(\eta,\mu,F_\mu)<\varepsilon^*_1$ such that  for any $x\in\Lambda$,   for any $0<\varepsilon\leq\varepsilon^*_2$,   for any $N\in\mathbb{N}$,   there exist an $n=n(\eta,\mu,F_\mu,  \varepsilon,x)\geq N$ such that for any $p\in\mathbb{N}$,   there exists an $(pn,  \frac{\varepsilon^*_1}{3})$-separated set $\Gamma_{pn}$ so that
	\begin{description}
		\item[(1)]\label{lem-B-aACor} $\Gamma_{pn}\subseteq X_{pn,  F_{\mu}}\cap B(x,  \varepsilon)\cap P_{pn}(f)$;
		\item[(2)]\label{lem-B-bACor} $\frac{\log |\Gamma_{pn}|}{pn}>h_{\mu}-\eta$;
		\item[(3)]\label{lem-C-cAcor} $d_H(\{f^i(y)\}_{i=0}^{pn-1},  \Lambda)<2\varepsilon$ for any $y\in\Gamma_{pn}.$
	\end{description}	
\end{Cor}
\begin{proof}
	Since $(X,  f)$ is positively expansive (resp. expansive),   $f|_\Lambda$ is positively expansive (resp. expansive) so that by Proposition \ref{Prop-PS2007}, $f|_\Lambda$ has the uniform separation property.   Therefore,
	the $\varepsilon^*_1$  in the proof of Lemma \ref{Mainlemma-prop-mu-gamma-n} can be chosen independent of $\mu$ and $F_{\mu}$ replacing using Lemma \ref{useful-Proposition} by  using uniform separation of $f|_\Lambda$.   
	
	Let $c>0$ be the expansive constant. 
	Since $(X,  f)$ is topologically expanding (resp. topologically Anosov),  then modify the proof of Lemma \ref{Mainlemma-prop-mu-gamma-n} by letting $\varepsilon^*_2=\min\{\frac{\varepsilon^*_1}{4},  \frac{a}{4},\frac{c}{2}\}$  and $$\mathfrak{C}_{\theta}=\widetilde{\mathfrak{C}}_{\theta}\cdots\widetilde{\mathfrak{C}}_{\theta}\cdots.$$ $$(\text{resp. } \mathfrak{C}_{\theta}=\cdots\widetilde{\mathfrak{C}}_{\theta}\cdots\widetilde{\mathfrak{C}}_{\theta}\widetilde{\mathfrak{C}}_{\theta}\cdots\widetilde{\mathfrak{C}}_{\theta}\cdots.)$$   From the construction of the $\delta$-pseudo-orbit $\mathrm{\mathfrak{C}_{\theta}}$,   one sees that for any $w\in\Gamma_{pn}$,   $d(f^{pn+t}(w),  f^{t}(w))<\varepsilon+\varepsilon=2\varepsilon\leq c$ for any $t\in\mathbb{Z}^+$ (resp. $t\in\mathbb{Z}$).   Therefore,   $f^{pn}(w)=w$,   which implies that $\Gamma_{pn}\subseteq P_{pn}(f)$.
\end{proof}

Now, we give the proof of Theorem \ref{G-K-star}.
\subsection{Proof of Theorem \ref{G-K-star}}

We need to show that for any internally chain transitive closed invariant subset $\Lambda\subseteq X$,   any nonempty open set $U\subseteq X$ and any nonempty compact connected subset $K\subseteq M(f,  \Lambda)$,   one has
$h_{top}(f,  G_K^{\Lambda}\cap U)=\inf\{h_\mu (f): \mu\in K\}. $ 
Since $G_K^{\Lambda}\subseteq G_K$ by definition and $h_{top}(f,G_K)\leq\inf\{h_\mu (f):  \mu\in K\}$ by Lemma \ref{lem-Bowen},   one has $h_{top}(f,G_K^{\Lambda}\cap U)\leq \inf\{h_\mu (f): \mu\in K\}.  $ 
So it remains to show that
\begin{equation}\label{eq-G-K-U-geq-inf}
	h_{top}(f,G_K^{\Lambda}\cap U)\geq\inf\{h_\mu (f): \mu\in K\}.
\end{equation}
To begin with,   we recall the following conclusion which was used in the proof of \cite[Proposition 21.  14]{DGS}.   A proof can be be found,   for example,   in \cite{DTY}.
\begin{Lem}
	For any nonempty compact connected set $K\subseteq \mathcal {M}_{f}(X)$,   there exists a sequence of open balls $B_{n}$ in $\mathcal {M}_{f}(X)$ with radius $\frac{\zeta_{n}}{2}$ such that the following holds:
	\begin{description}
		\item[(1)] $B_{n}\bigcap B_{n+1}\bigcap K\neq\emptyset$;
		\item[(2)] ${\bigcap}_{N=1}^{\infty}{\bigcup}_{n\geq N}B_{n}=K$;
		\item[(3)] $\lim_{n\rightarrow \infty}\zeta_{n}=0$.
	\end{description}
\end{Lem}
This allows us to pick an $\alpha_n\in B_{n}\bigcap B_{n+1}\bigcap K$ for each $n$.   Then
\begin{equation*}
	\rho(\alpha_n,  \alpha_{n+1})<\zeta_{n+1}~\text{for each}~n.
\end{equation*}
Moreover,   $\{\alpha_n\}_{n\geq1}$ are dense in $K$ due to property (2) and (3) there.

Fix an $x\in \Lambda$ and an arbitrary $\eta>0.$   Let $F_k=\mathcal{B}(\alpha_k,  \zeta_k)$ for each $k\in\mathbb{N}$.   Since $(X,  f)$ satisfies $s$-limit-shadowing property and uniform separation property,   by  Corollary  \ref{Maincor-mu-gamma-uniform}    there is an $\varepsilon^*=\varepsilon^*(\eta)>0$ and $0<\varepsilon_k^*<\varepsilon^*$ such that for any $0<\varepsilon_k\leq\varepsilon_k^*$ and any $N\in\mathbb{N}$,   there is an $n_k\geq N$ and an $(n_k,  \frac{\varepsilon^*}{3})$-separated set $\Gamma^{\alpha_k}_{n_k}$ so that
\begin{description}
	\item[(a)] $\Gamma^{\alpha_k}_{n_k}\subseteq X_{n_k,  F_k}\cap B(x,  \varepsilon_k)\cap f^{-n_k}B(x,  \varepsilon_k)$;
	\item[(b)] $\frac{\log |\Gamma^{\alpha_k}_{n_k}|}{n_k}>h_{\alpha_k}-\eta$;
	\item[(c)] $d_H(\{f^iy\}_{i=0}^{n_k-1},  \Lambda)<2\varepsilon_k$ for any $y\in\Gamma^{\alpha_k}_{n_k}.$ 
\end{description}

\textbf{Case 1:}
If $U\cap \Lambda \neq \emptyset,$ we take $x_0\in U\cap \Lambda.$ Let $d=\operatorname{dist}(x_0,  X\setminus U)$ and set $\varepsilon:=\min\{\frac{\varepsilon^*}{9}, \frac{d}{2}\}.$ By the $s$-limit-shadowing property,   there is a $0<\delta<\varepsilon$ such that any $\delta$-pseudo-orbit can be $\varepsilon$-shadowed and any $\delta$-limit-pseudo-orbit can be $\varepsilon$-limit-shadowed. Since $\Lambda$ is internally chain transitive,  there exists a $\frac{\delta}{2}$-chain $\mathfrak{C}_{x_0x}$ in $\Lambda$ with length $n_{0}'$ connecting $x_0$ and $x.$

\textbf{Case 2:}
If $U\cap \Lambda = \emptyset,$ we take $x_0\in U.$ Let $d=\operatorname{dist}(x_0,  X\setminus U)$ and set $\varepsilon:=\min\{\frac{\varepsilon^*}{9}, \frac{d}{2}\}.$ By the $s$-limit-shadowing property,   there is a $0<\delta<\varepsilon$ such that any $\delta$-pseudo-orbit can be $\varepsilon$-shadowed and any $\delta$-limit-pseudo-orbit can be $\varepsilon$-limit-shadowed. Since $(X,  f)$ is topologically transitive, then $X$ is internally chain transitive, and thus there exist a $\frac{\delta}{2}$-chain $\mathfrak{C}_{x_0x}$ in $X$ with length $n_{0}'$ connecting $x_0$ and $x.$

We set $\varepsilon_k=\min\{2^{-k}\delta,\varepsilon_k^*\}$ for each $k\in\mathbb{N}$. Then we have $\Gamma^{\alpha_k}_{n_k}$ satisfying property (a), (b) and (c).
Now choose a strictly increasing positive integer sequence $N_k$ such that
\begin{equation}\label{equa-AA}
	\lim_{k\rightarrow\infty}\frac{n_k}{\sum_{j=1}^{k-1}N_jn_j+n_0'}=0, 
\end{equation}
\begin{equation}\label{equa-AB}
	\lim_{k\rightarrow\infty}\frac{N_kn_k}{\sum_{j=1}^{k-1}N_jn_j+n_0'}=+\infty.
\end{equation}
Moreover,   we define the stretched sequences $\{n_{j}'\}$,   $\{\varepsilon_{j}'\}$ and $\{\Gamma_{j}'\}$,   by setting for
$$j=N_{1}+\cdots N_{k-1}+q~\textrm{with}~1\leq q\leq N_{k},  \,\,\,\, n_{j}':=n_{k},\,\,  \, \varepsilon_{j}':=\varepsilon_{k},\,\,  \,  \Gamma_j':=\Gamma_{n_k}^{\alpha_k}.  $$

Let $\mathcal{W}:=\prod_{j=1}^{+\infty}\Gamma_{j}'$
whose element is $w=(w_1w_2\cdots)$ with $w_j\in\Gamma_j'$,   $j\geq1$.
For each $w_j\in\Gamma_j'$,   denote
$$\mathfrak{C}_{w_j}=\langle w_j,  \cdots,  f^{n_j'-1}w_j\rangle.$$
Due to the property (a) of $\Gamma_{n_k}^{\alpha_k}$ and the fact that $\varepsilon_{i}+\varepsilon_j\leq\delta,  i,  j\in\mathbb{N}$,   we can freely concatenate $\mathfrak{C}_{\omega_i}$s to constitute a $\delta$-pseudo-orbit.   In particular,  for each $w\in\mathcal{W}$,   there corresponds a $\delta$-pseudo-orbit
$$O_w:=\mathfrak{C}_{x_0x}\mathfrak{C}_{w_1}\mathfrak{C}_{w_2}\cdots\mathfrak{C}_{w_k}\cdots\cdots.$$
The construction of $O_w$ and the fact that $\lim_{i\to\infty}\varepsilon_i=0$ shows that $O_w$ is a $\delta$-limit-pseudo-orbit. We denote $O_w=\{y_n\}_{n=0}^{+\infty},$ then
the following defined set
$$G(w):=\{y\in X: d(f^n(y),  y_n)<\varepsilon \text{ for any }n\in\mathbb{Z}^+ \text{ and } \lim_{n\to\infty}d(f^n(y),  y_n)=0\}$$
is nonempty.

Let
\begin{equation*}
	G=\bigcup_{w\in\mathcal{W}}G(w).
\end{equation*}
Therefore,   for any $y\in G$,   there is a $w=w(y)\in\mathcal{W}$ such that $y\in G(w)$. Moreover,   we have
\begin{Lem}\label{lem-G-disjoint}
	$G=\bigsqcup_{w\in\mathcal{W}}G(w)$.   Here $\bigsqcup$ denotes the disjoint union.
\end{Lem}
\begin{proof}
	For any $u,  v\in\mathcal{W},  u\neq v,$ there exists an $i\in\mathbb{N}$ such that $u_i\neq v_i$.
	This implies that $$d_{n_i'}(u_i,  v_i)\geq \frac{\varepsilon^*}{3}\geq 3\varepsilon.$$
	For any $y_u\in G(u),$ $y_v\in G(v)$
	we have $d_{n_i'}(f^{M_{i-1}}(y_u),  u_i)<\varepsilon$ and $d_{n_i'}(f^{M_{i-1}}(y_v),  v_i)<\varepsilon$ where $M_{i-1}:=\sum_{l=0}^{i-1}n_{l}',$   which implies that
	$d_{M_i'}(y_u,  y_v)>\varepsilon.$ So $G(u)\cap G(v)=\emptyset.$
\end{proof}

Now we show that
\begin{equation}\label{G-subset-G-K-star}
	G\subseteq G_K^{\Lambda}\cap U.
\end{equation}
Indeed,   $G\subseteq U$ is clear by construction.   To show that for any $y\in G,  $ $\omega_f(y)=\Lambda$,   suppose $y$ $\varepsilon$-limit-shadows some $O_w=\{y_i\}_{i=0}^{+\infty}$ with $w\in\mathcal{W}$.   Define the $\omega$-limit set of the pseudo-orbit $\{y_i\}_{i\in\mathbb{N}}$ by
$$\omega(\{y_i\}_{i\in\mathbb{N}})=\bigcap_{n\in\mathbb{N}}\overline{\bigcup_{k\geq n}\{y_k\}}.$$
Due to property (c) of $\Gamma_{n_k}^{\alpha_k}$,   one sees that $\omega(\{y_i\}_{i=0}^{+\infty})=\Lambda$. 
The following lemma can be checked from the definition of limit points.
\begin{Lem}\label{lem-omega-pseudo-equal-real}
	If a pseudo-orbit  $\{y_i\}_{i\in\mathbb{N}}$ is limit-shadowed by $y\in X$,   then $\omega(\{y_i\}_{i\in\mathbb{N}})=\omega_f(y)$, where $\omega(\{y_i\}_{i\in\mathbb{N}})$ is the set of limit points of $\{y_i\}_{i\in\mathbb{N}}$.
\end{Lem}
Lemma \ref{lem-omega-pseudo-equal-real} ensures that $\omega_f(y)=\omega(\{y_i\}_{i=0}^{+\infty})=\Lambda$.
Then we are left to prove

\begin{Lem}
	For any $y\in G$,   $V_f(y)=K$.
\end{Lem}
\begin{proof}
	Since $\{\alpha_k\}_{k=1}^{\infty}$ is dense in $K$,   it suffices to prove that $\{\mathcal{E}_n(y)\}_{n=1}^{\infty}$ has the same limit-points as $\{\alpha_k\}_{k=1}^{\infty}$.
	To this end,   set $S_k:=\sum_{l=1}^kn_lN_l+n_0'$ with $S_0:=n_0'$ and define the stretched sequence $\{\alpha_m'\}_{m=S_0}^{\infty}$ by
	$$\alpha_m':=\alpha_k~\text{if}~S_{k-1}< m\leq S_k.  $$
	The sequence $\{\alpha_m'\}_{m=S_0}^{\infty}$ has the same limit-point set as the sequences $\{\alpha_k\}_{k=1}^{\infty}$.
	
	Now suppose $y$ $\varepsilon$-limit-shadows some  $O_w=\{y_n\}_{n=0}^{+\infty}$ with $w\in\mathcal{W}$,   then $\lim_{n\to\infty} d(f^n(y),  y_n)=0$.   Let $\beta_n=\frac{1}{n}\sum_{i=0}^{n-1}\delta_{y_i}.  $ Using Lemma \ref{lem:prohorov},   we have $\lim_{n\to\infty}\rho(\mathcal{E}_n(y),  \beta_n)=0$.   Hence,   $\mathcal{E}_n(y)$ shares the same limit-points as the sequence $\beta_n$.   So we are left to prove that
	$\lim_{n\to\infty}\rho(\alpha_n',  \beta_n)=0.  $ 
	In fact,   by Lemma \ref{lem:prohorov} and \eqref{equa-AA},   we only need to prove that
	$\lim_{k\to\infty}\rho(\alpha_{M_k}',  \beta_{M_k})=0.$
	Indeed,   Lemma \ref{lem:prohorov} and \eqref{equa-AB} indicates that
	\begin{equation}\label{beta-S-n-alpha-n}
		\lim_{n\to\infty}\rho(\beta_{S_n},  \alpha_n)=0.
	\end{equation}
	Now for each $k\geq N_1$,   there are $j\in\mathbb{N}$ and $1\leq r\leq N_{j+1}$ such that $M_k=S_j+rn_{j+1}$.   Then
	$$\alpha_{M_k}'=\alpha_{j+1}~\textrm{and}~\beta_{M_k}=\frac{S_j\beta_{S_j}+\sum_{p=1}^{r}n_{j+1}\mathcal{E}_{n_{j+1}}(w_{N_1+\cdots+N_k+p})}{M_k}.  $$
	However,   since $w_t\in\Gamma_{t}'$ for each $t\in\mathbb{N}$,
	$$\rho(\mathcal{E}_{n_{j+1}}(w_{N_1+\cdots+N_k+p}),  \alpha_{j+1})<\zeta_{j+1}~\textrm{for}~1\leq p\leq r$$
	Therefore,   one has
	\begin{eqnarray*}
		\rho(\beta_{M_k},  \alpha_{j+1}) &\leq& \frac{S_j}{M_k}[\rho(\beta_{S_j},  \alpha_j)+\rho(\alpha_j,  \alpha_{j+1})]+\frac{n_{j+1}}{M_k}\sum_{p=1}^r\rho(\mathcal{E}_{n_{j+1}}(w_{N_1+\cdots+N_k+p}),  \alpha_{j+1}) \\
		&\leq&  \rho(\beta_{S_j},  \alpha_j)+\zeta_{j+1}+\zeta_{j+1}.
	\end{eqnarray*}
	The proof will be completed if one notices \eqref{beta-S-n-alpha-n} and that $\lim_{i\to\infty}\zeta_i=0.  $
\end{proof}

Let $\widetilde{h}=\inf\{h_\mu:\mu\in K\}$ and $h=\widetilde{h}-2\eta$.   We shall prove that
\begin{equation}\label{entropy-of-G}
	h_{top}(f,  G)\geq h.
\end{equation}
Recall the definition in \eqref{definition-of-topological-entropy}.   Since $C(G;s,  \sigma,  f)$ is non-decreasing as $\sigma$ decreases,   it is enough to prove that there exists $\widetilde{\sigma}>0$ such that
\begin{equation}\label{estimation-of-C}
	C(G;h,  \widetilde{\sigma},  f)\geq 1.
\end{equation}
In fact,   we will prove that (\ref{estimation-of-C}) holds for $\widetilde{\sigma}=\frac{\varepsilon}{2}$. 
Due to \eqref{equa-AA},   there exists a $q\in\mathbb{N}$ such that
\begin{equation}\label{selection-of-q}
	\frac{M_r-n_0'}{M_{r+1}}\geq\frac{\widetilde{h}-2\eta}{\widetilde{h}-\eta}~\textrm{for any}~r\geq q.
\end{equation}
Now let $N=M_q$,   it suffices to prove that
\begin{equation*}\label{C-n-neq-N}
	C(G;h,  n,  \frac{\varepsilon}{2},  f)\geq1~\textrm{for any}~n\geq N.
\end{equation*}
Alternatively,   one should show that for any $\mathcal{C}\in\mathcal{G}_n(G,  \frac{\varepsilon}{2})$ with $n\geq N$,
\begin{equation*}
	\sum_{B_u(z,  \frac{\varepsilon}{2})\in\mathcal{C}}e^{-hu}\geq 1.
\end{equation*}
Moreover,   we can assume that if $B_u(z,  \frac{\varepsilon}{2})\in \mathcal{C}$ then $B_u(z,  \frac{\varepsilon}{2})\cap G\neq\emptyset$,   since otherwise we may remove this set from $\mathcal{C}$ and it still remains as cover of $G$.   In addition,   for each $\mathcal{C}\in\mathcal{G}_n(G,  \frac{\varepsilon}{2})$,   we define a cover $\mathcal {C}'$ in which each ball $B_{u}(z,  \frac{\varepsilon}{2})$ is replaced by $B_{M_{m}}(z,  \frac{\varepsilon}{2})$,   $M_{m}\leq u< M_{m+1}$,   $m\geq q$.   Then
\begin{equation}\label{contract-prolong}
	\sum_{B_{u}(z,  \frac{\varepsilon}{2})\in\mathcal {C}}e^{-hu}\geq\sum_{B_{M_{m}}(z,  \frac{\varepsilon}{2})\in\mathcal {C}'}e^{-hM_{m+1}}.
\end{equation}

Consider a specific $\mathcal{C}'$ and let $c$ be the largest number of $p$ for which there exists $B_{M_p}(z,  \varepsilon)\in\mathcal{\mathcal{C}'}$.
Define
\begin{equation*}
	\mathcal {V}_{c}:=\bigcup_{l=1}^{c}\mathcal {W}_{l},
\end{equation*}
where $\mathcal {W}_{l}:=\prod_{i=1}^{l}\Gamma_i'$ whose element is $\underline{w}=(w_1\cdots w_l)$ with $w_i\in\Gamma_i'$ for $1\leq i\leq l$.
For $1\leq j\leq k$,   we say $\underline{w}\in\mathcal{W}_j$ is a prefix of $\underline{v}\in\mathcal{W}_k$ if the first $j$ entries of $\underline{v}$ coincides with $\underline{w}$,   namely $w_i=v_i$ for $1\leq i\leq j$.
Note that each word of $\mathcal {W}_{m}$ is a prefix of exactly $|\mathcal {W}_{c}|/|\mathcal {W}_{m}|$ words of $\mathcal {W}_{c}$.
\begin{Lem}
	If $\mathcal {V}\subset\mathcal {V}_{c}$ contains a prefix of each word of $\mathcal {W}_{c}$,   then
	\begin{equation}\label{V-W-c-m}
		\sum_{m=1}^{c}|\mathcal {V}\cap\mathcal {W}_{m}|\frac{|\mathcal {W}_{c}|}{|\mathcal {W}_{m}|}\geq |\mathcal {W}_{c}|.
	\end{equation}
	Consequently,
	\begin{equation}\label{w}
		\sum_{m=1}^{c}|\mathcal {V}\cap\mathcal {W}_{m}|/|\mathcal {W}_{m}|\geq1.
	\end{equation}
\end{Lem}
\begin{proof}
	For any $\underline{w}\in\mathcal{W}_c$,   there is an $1\leq m\leq c$ such that $(w_1,  \cdots,  w_m)\in\mathcal{V}\cap\mathcal{W}_m$.   However,   for each $\underline{v}\in \mathcal{W}_m$,   the number of $\underline{w}\in\mathcal{W}_c$ such that $(w_1,  \cdots,  w_m)=\underline{v}$ does not exceed $|\mathcal {W}_{c}|/|\mathcal {W}_{m}|$.   This proves \eqref{V-W-c-m} and hence \eqref{w}.
\end{proof}

Each $z'\in B_{M_{m}}(z,  \frac{\varepsilon}{2})\cap G$ corresponds to a point in $\mathcal {W}_{m}$. The proof of Lemma \ref{lem-G-disjoint} implies that this point is uniquely defined.  So $\mathcal{C}'$ corresponds uniquely to a subset $\mathcal {V}\subset\mathcal{V}_c.$ Note that $\mathcal{C}'$ is a cover, then $\mathcal{V}$ contains a prefix of each word of $\mathcal {W}_{c}$. 
This combining with (\ref{w}) gives that
\begin{equation*}
	\sum_{B_{M_{m}}(z,  \frac{\varepsilon}{2})\in\mathcal {C}'}\frac{1}{|\mathcal{W}_{m}|}\geq1.
\end{equation*}
Since
$$|\mathcal{W}_m|=\prod_{j=1}^m|\Gamma_j'|\geq e^{\sum_{j=1}^mn_j'(\widetilde{h}-\eta)}= e^{(M_m-n_0')(\widetilde{h}-\eta)}$$
by the property (b) of $\Gamma_{n_k}^{\alpha_k},$
one has
$$\sum_{B_{M_m}(z,  \frac{\varepsilon}{2})\in\mathcal{C}'}e^{-(M_m-n_0')(\widetilde{h}-\eta)}\geq1.  $$
Moreover,   since $m\geq q$,   \eqref{selection-of-q} produces that
$$\sum_{B_{M_m}(z,  \frac{\varepsilon}{2})\in\mathcal{C}'}e^{-M_{m+1}(\widetilde{h}-2\eta)}\geq1.  $$
Note that $h=\widetilde{h}-2\eta$.   Using \eqref{contract-prolong},   one gets that
$$\sum_{B_{u}(z,  \frac{\varepsilon}{2})\in\mathcal {C}}e^{-hu}\geq1\Rightarrow C(G;h,  n,  \frac{\varepsilon}{2},  f)\geq1.  $$
Finally,   \eqref{G-subset-G-K-star},   \eqref{entropy-of-G} and the arbitrariness of $\eta$ lead to \eqref{eq-G-K-U-geq-inf}. Now we complete the proof of Theorem \ref{G-K-star}. \qed

By \textbf{Case 1} of the proof of Theorem \ref{G-K-star}, we have the following without topologically transitive assumption. In fact, in this case, it is enough to choose $x_0\in\Lambda$ directly and set $\varepsilon=\frac{\varepsilon^*}{9}$.
\begin{Prop}
	Suppose $(X,  f)$ satisfies s-limit-shadowing property and uniform separation property. 
	Then $f$ is  star-saturated.
\end{Prop}

\section{Topological entropy of \texorpdfstring{$\mathscr{C}(\alpha_1\alpha_2\cdots\alpha_6)$}{C(alpha1alpha2...alpha6)}: proof of Theorem \ref{Theorem B}}\label{section 7}
We divide the proof of Theorem \ref{Theorem B} into four parts: \textbf{Case} $(1\alpha_2\alpha_3\alpha_411)$; \textbf{Case} $(1\alpha_2\alpha_3\alpha_4\alpha_50)$; \textbf{Case} $(0\alpha_2\alpha_3\alpha_4\alpha_50)$ and \textbf{Case} $(011111)$.
\subsection{Proof of Theorem \ref{Theorem B}: \textbf{Case} $(1\alpha_2\alpha_3\alpha_411)$} Firstly, we prove some basic properties for $C_\Lambda^*$. 
\begin{Lem}\cite[Lemma 5.1]{KOR2016}\label{lemma-full-support}
	Suppose that $(X,  f)$ is a topological dynamical system, then there exists an invariant measure $\mu$ such that $S_\mu=C^*_X.$
\end{Lem}
\begin{Lem}\label{Lem-IrregularMultianalysis-11111}
	Suppose that $(X,  f)$ is a topological dynamical system. Let $\Lambda\subseteq X$ be a closed $f$-invariant subset,
	then for any $\eta>0,$ there exists an $f$-invariant measure $\omega$ on $\Lambda$ such that
	\begin{description}
		\item[(1)] $S_{\omega}=C^*_\Lambda$.
		\item[(2)] $h_{\omega}(f)>h_{top}(f,\Lambda)-\eta$.
	\end{description}
\end{Lem}
\begin{proof}
	By the variational principle,   we choose a $\nu\in M_{erg}(f,  \Lambda)$ such that $h_{\nu}(f)>h_{top}(f,\Lambda)-\frac{1}{2}\eta$. By Lemma \ref{lemma-full-support},  there is a $\lambda\in M(f,  \Lambda)$ such that $S_{\lambda}=C^*_\Lambda$. Choose $\theta\in(0,  1)$ close to $1$ such that $\omega:=\theta\nu+(1-\theta)\lambda$ satisfies that
	$$h_{\omega}(f)=\theta h_{\nu}(f)+(1-\theta)h_{\lambda}(f)\geq\theta h_{\nu}(f)>h_{top}(f,\Lambda)-\eta.$$
	Then $S_{\omega}=S_{\nu}\cup S_{\lambda}=C^*_\Lambda.$
\end{proof}
\begin{Lem}\label{Lem-IrregularMultianalysis-22222}
	Suppose that $(X,  f)$ is a topological dynamical system. Let $\Lambda\subseteq X$ be a closed $f$-invariant subset. If $\Lambda$ has refined entropy-dense property and is not uniquely ergodic,
	then for any $\eta>0$ and any $n\geq2$,  there exist $f$-invariant measures $\{\omega_k\}_{k=1}^n$ on $\Lambda$ such that 
	\begin{description}
		\item[(1)] $\{S_{\omega_k}\}_{k=1}^n$ are pairwise disjoint and consequently,   $S_{\omega_k}\neq C^*_\Lambda,  k=1,  \cdots,  n$.
		\item[(2)] $h_{\omega_k}(f)>h_{top}(f,\Lambda)-\eta,  k=1,  \cdots,  n.$
	\end{description}
\end{Lem}
\begin{proof}
	Fix $\eta>0$ and $n\geq2.$   By the variational principle,  we choose a $\omega\in M_{erg}(f,  \Lambda)$ such that $h_{\omega}(f)>h_{top}(f,\Lambda)-\frac{1}{2}\eta$.
	Since $\Lambda$  is not uniquely ergodic, there exists $f$-invariant measures $\nu$ on $\Lambda$ such that $\nu\neq \omega.$ For each $1\leq k\leq n$, choose $\theta_k\in(0,  1)$ close to $1$ such that $\nu_k:=\theta_k\omega+(1-\theta_k)\nu$ satisfies that
	$$h_{\nu_k}(f)=\theta_k h_{\omega}(f)+(1-\theta_k)h_{\nu}(f)\geq\theta_k h_{\omega}(f)>h_{top}(f,\Lambda)-\frac{3}{4}\eta,$$
	and $$\theta_1<\theta_2<\cdots<\theta_n.$$
	Then $\nu_i\neq \nu_j$ for any $1\leq i<j\leq n.$ Denote $\rho_0=\min\{\rho(\nu_i,\nu_j):1\leq i<j\leq n\}.$
	By the refined entropy-dense property, for each $k=1,  \cdots,  n$,   then there exists  $\omega_k\in M_{erg}(f,  \Lambda)$ with
	\begin{description}
		\item[(a)] $M(f,  S_{\omega_k})\subseteq \mathcal{B}(\nu_k,  \frac{1}{3}\rho_0),  k=1,  \cdots,  n.  $
		\item[(b)] $h_{\omega_k}(f)>h_{top}(f,\Lambda)-\eta,  k=1,  \cdots,  n.  $
	\end{description}
	Note that $\{\mathcal{B}(\nu_k,  \frac{1}{3}\rho_0)\}_{k=1}^n$ are pairwise disjoint.   So item(a) indicates that
	$\{M(f,  S_{\omega_k})\}_{k=1}^n$ are pairwise disjoint.   As a result,
	$\{S_{\omega_k}\}_{k=1}^n$ are pairwise disjoint and consequently,   $S_{\omega_k}\neq C^*_\Lambda,  k=1,  \cdots,  n$.
\end{proof}

For any $m\in\mathbb{N}$ and $\{\nu_i\}_{i=1}^m \subseteq M(X)$,   we write $\operatorname{cov}\{\nu_i\}_{i=1}^m$ for the convex combination of $\{\nu_i\}_{i=1}^m$,   namely,
$$\operatorname{cov}\{\nu_i\}_{i=1}^m=\operatorname{cov}(\nu_1,\cdots,\nu_m):=\left\{\sum_{i=1}^mt_i\nu_i:t_i\in[0,  1],  1\leq i\leq m~\textrm{and}~\sum_{i=1}^mt_i=1\right\}.$$

\begin{Lem}\label{Lemma-AA-New}
	Suppose $(X,  f)$ is a topological dynamical system. Let $\Lambda\subseteq X$ be a closed $f$-invariant subset, and assume that $\Lambda$ has refined entropy-dense property and is not uniquely ergodic.
	\begin{description}
		\item[(1)] If $C^*_\Lambda=\Lambda,$ then for any $\eta>0$ and any $i\in\{1,2,\cdots,6\}$, there is $K_i\subseteq M(f,  \Lambda)$ such that 
		$$G_{K_1}^\Lambda\cap Rec(f)\subseteq\mathscr{C}(101111), G_{K_2}^\Lambda\cap Rec(f)\subseteq\mathscr{C}(101011), G_{K_3}^\Lambda\cap Rec(f)\subseteq\mathscr{C}(110111),$$ 	$$G_{K_4}^\Lambda\cap Rec(f)\subseteq\mathscr{C}(100111), G_{K_5}^\Lambda\cap Rec(f)\subseteq\mathscr{C}(110011), G_{K_6}^\Lambda\cap Rec(f)\subseteq\mathscr{C}(100011),$$
		$$G_{K_1}^\Lambda\cap NRec(f)\subseteq\mathscr{C}(101110), G_{K_2}^\Lambda\cap NRec(f)\subseteq\mathscr{C}(101010), G_{K_3}^\Lambda\cap NRec(f)\subseteq\mathscr{C}(110110),$$ 	$$G_{K_4}^\Lambda\cap NRec(f)\subseteq\mathscr{C}(100110), G_{K_5}^\Lambda\cap NRec(f)\subseteq\mathscr{C}(110010), G_{K_6}^\Lambda\cap NRec(f)\subseteq\mathscr{C}(100010)$$		
        and $$\inf\{h_\mu (f):  \mu\in K_i\}>h_{top}(f,\Lambda)-\eta.$$
		\item[(2)] If $C^*_\Lambda\subsetneq \Lambda,$ then for any $\eta>0$ and any $i\in\{1',2',\cdots,6'\}$, there is $K_i\subseteq M(f,  \Lambda)$ such that 
		$$G_{K_1}^\Lambda\cap Rec(f)\subseteq\mathscr{C}(101101), G_{K_2}^\Lambda\cap Rec(f)\subseteq\mathscr{C}(101001), G_{K_3}^\Lambda\cap Rec(f)\subseteq\mathscr{C}(110101),$$ 	$$G_{K_4}^\Lambda\cap Rec(f)\subseteq\mathscr{C}(100101), G_{K_5}^\Lambda\cap Rec(f)\subseteq\mathscr{C}(110001), G_{K_6}^\Lambda\cap Rec(f)\subseteq\mathscr{C}(100001),$$
		$$G_{K_1}^\Lambda\cap NRec(f)\subseteq\mathscr{C}(101100), G_{K_2}^\Lambda\cap NRec(f)\subseteq\mathscr{C}(101000), G_{K_3}^\Lambda\cap NRec(f)\subseteq\mathscr{C}(110100),$$ 	$$G_{K_4}^\Lambda\cap NRec(f)\subseteq\mathscr{C}(100100), G_{K_5}^\Lambda\cap NRec(f)\subseteq\mathscr{C}(110000), G_{K_6}^\Lambda\cap NRec(f)\subseteq\mathscr{C}(100000)$$		
		and $$\inf\{h_\mu (f):  \mu\in K_i\}>h_{top}(f,\Lambda)-\eta.$$
	\end{description}
\end{Lem}
\begin{proof}
	By Lemma \ref{Lem-IrregularMultianalysis-11111}, there exists an $f$-invariant measure $\omega_0$ on $\Lambda$ such that
	\begin{description}
		\item[(a)] $S_{\omega_0}=C^*_\Lambda$.
		\item[(b)] $h_{\omega_0}(f)>h_{top}(f,\Lambda)-\eta$.
	\end{description}
	By Lemma \ref{Lem-IrregularMultianalysis-22222}, there exist $f$-invariant measures $\{\omega_k\}_{k=1}^4$ on $\Lambda$ such that 
	\begin{description}
		\item[(c)] $\{S_{\omega_k}\}_{k=1}^4$ are pairwise disjoint and consequently,   $S_{\omega_k}\neq C^*_\Lambda,  k=1,  \cdots,  4.$
		\item[(d)] $h_{\omega_k}(f)>h_{top}(f,\Lambda)-\eta,  k=1,  \cdots,  4.$
	\end{description}
	By item (c), we have $S_{\omega_1}\cup S_{\omega_2}\cup S_{\omega_3}\subseteq C^*_\Lambda\setminus S_{\omega_4}  \subsetneq C^*_\Lambda.$
	Define 
	$$\nu_1=\frac{1}{3}\omega_{0}+\frac{2}{3}\omega_{1},\ \nu_2=\frac{1}{3}\omega_{1}+\frac{2}{3}\omega_{2},$$
	$$\nu_3=\frac{1}{2}\omega_{1}+\frac{1}{2}\omega_{2},\ \nu_4=\frac{1}{2}\omega_{1}+\frac{1}{2}\omega_{3},$$
	Then we have
	\begin{itemize}
		\item $\nu_1\neq\omega_0,$ and $S_{\nu_1}=S_{\omega_0}=C^*_\Lambda.$
		\item $\nu_2\neq\nu_3,$ and $S_{\nu_2}=S_{\nu_3}=S_{\omega_1}\cup S_{\omega_2}\subsetneq C^*_\Lambda.$
		\item $\nu_3\neq\nu_4,$ and $S_{\nu_3}\cap S_{\nu_4}=S_{\omega_1}$ and $S_{\nu_3}\cup S_{\nu_4}=S_{\omega_1}\cup S_{\omega_2}\cup S_{\omega_3}\subsetneq C^*_\Lambda.$ 
	\end{itemize}
	Let
	$$K_1:=\operatorname{cov}\{\omega_0,\nu_1\};$$
	$$K_2:=\operatorname{cov}\{\nu_2,\nu_3\};$$
	$$K_3:=\operatorname{cov}\{\omega_k\}_{k=0}^2;$$
	$$K_4:=\operatorname{cov}\{\omega_k\}_{k=0}^1;$$
	$$K_5:=\operatorname{cov}\{\omega_k\}_{k=1}^2;$$
	$$K_6:=\operatorname{cov}\{\nu_k\}_{k=3}^4.$$
	Then we have $$\inf\{h_\mu (f):  \mu\in K_i\}\geq\inf\{h_{\omega_k} (f):  k=0,1,2,3\} >h_{top}(f,\Lambda)-\eta,\ i\in\{1,2,\cdots,6\},$$
	For any $x\in G_{K_i}^{\Lambda}$,   $\omega_f(x)=\Lambda$ by definition.   So by item (3) and item (4) of Theorem \ref{thm-density-basic-property}, 
	\begin{equation}\label{eq-X-e}
		\omega_{\underline{B}}(x)={\bigcap_{\mu\in M(f,  \omega_f(x))}S_{\mu}}= {\bigcap_{\mu\in M(f,  \Lambda)}S_{\mu}}=S_{\omega_1}\cap S_{\omega_2}=\emptyset.
	\end{equation}
	\begin{equation}\label{eq-X-f}
		\omega_{\overline{B}}(x)=C^*_x=\overline{\bigcup_{\mu\in M(f,  \omega_f(x))}S_{\mu}}=\overline{\bigcup_{\mu\in M(f,  \Lambda)}S_{\mu}}=C^*_\Lambda.
	\end{equation}
	Moreover,   since $V_f(x)=K_i$ by definition, item (1) and item (2) of Theorem \ref{thm-density-basic-property} gives that
	\begin{equation}\label{eq-X-d}
		\omega_{\underline{d}}(x)=\bigcap_{\mu\in V_f(x)}S_{\mu}=\bigcap_{\mu\in K_i}S_{\mu}~\textrm{and}~\omega_{\overline{d}}(x)=\overline{\bigcup_{\mu\in V_f(x)}S_{\mu}}=\overline{\bigcup_{\mu\in K_i}S_{\mu}}.
	\end{equation}
	Therefore,   a convenient use of \eqref{eq-X-e}, \eqref{eq-X-f} and \eqref{eq-X-d} yields item (1) and item (2).
\end{proof}

Now, we give the proof of Theorem \ref{Theorem B}: \textbf{Case} ($1\alpha_2\alpha_3\alpha_411$).

\emph{Proof of of Theorem \ref{Theorem B}: \textbf{Case} $(1\alpha_2\alpha_3\alpha_411)$.} Suppose that $(X,f)$ is topologically transitive topologically expanding or topologically transitive topologically Anosov, then $(X,f)$ is not uniquely ergodic, by Lemma \ref{lem-shadowing-full-support}, $C_X^*=X$ and by Corollary \ref{prop-entropy-dense-for-shadowing}, $(X,f)$ has the refined entropy-dense property. Denote $\Lambda=X$, then we always have that $\{x\in X:\omega_f(x)=\Lambda\}\subseteq Rec(f)$. Hence, by Lemma \ref{Lemma-AA-New},  for any $\eta>0$ and any $i\in\{1,2,\cdots,6\}$, there is $K_i\subseteq M(f,  X)$ such that 
$$G_{K_1}^X\subseteq\mathscr{C}(101111), G_{K_2}^X \subseteq\mathscr{C}(101011), G_{K_3}^X \subseteq\mathscr{C}(110111),$$ 	$$G_{K_4}^X\subseteq\mathscr{C}(100111), G_{K_5}^X \subseteq\mathscr{C}(110011), G_{K_6}^X\subseteq\mathscr{C}(100011)$$
and $$\inf\{h_\mu (f):  \mu\in K_i\}>h_{top}(f)-\eta.$$
For any nonempty open set $U\subset X$, by Lemma \ref{lem-shadowing-s-limit-shadowing}, Lemma \ref{lem-shadowing-s-limit-shadowing-2} and Theorem \ref{G-K-star}, we have that $$h_{top} (f,  G_{K_i}^{X}\cap U)=\inf\{h_\mu (f):  \mu\in K_i\}>h_{top}(f)-\eta.$$
Hence, for any $1\alpha_2\alpha_3\alpha_411\in\mathfrak{A}$, we have that $$h_{top}(f,\mathscr{C}(1\alpha_2\alpha_3\alpha_411))>h_{top}(f)-\eta.$$
Finally, by the arbitrariness of $\alpha$, we finish the proof. \qed

Before give the proof of the remaining three parts, we show the following auxiliary lemmas, which will assist us in transitioning from  \textbf{Case} $(\alpha_1\alpha_2\alpha_3\alpha_410)$ to \textbf{Case} $(\alpha_1\alpha_2\alpha_3\alpha_400)$.

\begin{Lem}\label{lem-omega-A}
	Suppose that $(X,  f)$ is topologically transitive topologically expanding. 
	Let $\Lambda\subsetneq X$ be a nonempty closed invariant subset of $X$.   Then there exists an $x\in X$ such that
	\begin{equation}\label{A-f-A}
		\Lambda\subseteq\omega_f(x)\subseteq \bigcup_{l=0}^{\infty}f^{-l}\Lambda.
	\end{equation}
	In particular,
	\begin{equation}\label{omega-f-x-B}
		\overline{\bigcup_{y\in\omega_f(x)}\omega_f(y)}\subseteq \Lambda\subseteq\omega_f(x)\neq X.
	\end{equation}
\end{Lem}
\begin{proof}
	Suppose that $(X,f)$ is topologically transitive topologically expanding. Let $c$ be the expansive constant.   By the shadowing property, there exists a $\delta>0$ such that any $\delta$-pseudo orbit can be $c$-shadowed.   Since $\Lambda$ is compact,   we cover $\Lambda$ by finite number of balls $\{B(y_i, \frac{\delta}{2})\}_{i=1}^p$.   By the transitivity,   for any two open balls $B(y_i, \frac{\delta}{2})$ and $B(y_j,  \frac{\delta}{2})$,   there exists a $w_{i,j}\in B(y_i, \frac{\delta}{2})$ such that $f^{l_{i,j}}(w_{i,j})\in B(y_j, \frac{\delta}{2})$ for some $l_{i,j}\in \mathbb{N}$.
	Let
	\begin{equation}\label{definition-of-L}
		L=\max\{l_{i,j}:1\leq i,  j\leq p\}.
	\end{equation}
	
	Meanwhile,   as a compact subset of compact metric space,   $\Lambda$ is separable,   i. e.   there exist $\{x_n\}_{n\geq1}\subseteq \Lambda$ which is dense in $\Lambda$.   Now let us construct the $\delta$-pseudo orbit as following.

	For $r\geq 1$,   one can find $n\in\mathbb{N}$ and $1\leq t\leq n$ such that $r=\frac{n(n-1)}{2}+t$.  Then let $z_r=x_t$ and choose the orbit segment $\mathfrak{O}_{r}=\langle z_r,  \cdots,  f^{r-1}(z_r)\rangle$. Since $\Lambda\subseteq\cup_{i=1}^{p} B(y_i, \frac{\delta}{2}),$ there exist $i_{r,1},i_{r,2}\in\{1,2,\dots,p\}$ such that $z_r\in B(y_{i_{r,1}}, \frac{\delta}{2})$ and $f^{r}(z_r)\in B(y_{i_{r,2}}, \frac{\delta}{2}).$  Let $\mathfrak{C}_r=\langle w_{i_{r,2},i_{r+1,1}},  \cdots,  f^{l_{i_{r,2},i_{r+1,1}}-1}(w_{i_{r,2},i_{r+1,1}})\rangle.$

	Hence we obtain the following pseudo orbit:
	$$\mathfrak{O}=\mathfrak{O}_1\mathfrak{C}_1\cdots\mathfrak{O}_r\mathfrak{C}_r\cdots.  $$
	It is clear that $\mathfrak{O}$ constitutes a $\delta$-pseudo orbit.   So there exists an $x\in B(z_1,  c)$ such that $x$ $c$-shadows $\mathfrak{O}$.   Put $j_1=0,$ $k_1=1$ and $j_m=k_{m-1}+|\mathfrak{C}_{m-1}|,  ~k_m=j_{m-1}+m$ inductively for $m\geq 2$.
	
	We now prove that $\Lambda\subseteq \omega_f(x)$.
	Indeed,   it is sufficient to prove that for any $n$,   $x_n\in\omega_f(x)$.   For then the denseness of $\{x_n\}_{n\geq1}$ in $\Lambda$ and the closeness of $\Lambda$ yield the result.   In fact,   $x_n$ occurs infinitely in $\mathfrak{O}$ at the place $\{j_{\frac{k(k-1)}{2}+n}\}_{k=n}^{\infty}$.   By the compactness of $X$,   we suppose (by taking a subsequence if necessary) that $\lim_{k\to\infty}f^{j_{\frac{k(k-1)}{2}+n}}(x)=u\in\omega_f(x)$.   Then $d(f^i(x_n),  f^i(u))=\lim_{k\to\infty}d(f^i(x_n),  f^{j_{\frac{k(k-1)}{2}+n}+i}(x))<c$.   By the positive expansiveness,   we have $u=x_n$.   Hence,   $x_n\in\omega_f(x)$.
	
	On the other hand,   for any $v\in\omega_f(x)$,   there exists a strictly increasing sequence $n_s$ such that $v=\lim_{s\to\infty}f^{n_s}(x)$.   The following discussion splits into two cases.
	
	Case 1.   For any $p\geq 1$,   there exists an $n_{s_p}$ such that $j_{m_p}\leq n_{s_p}<n_{s_p}+p\leq k_{m_p}$ for some $m_p\in\mathbb{N}$.   Suppose (take a subsequence if necessary) $\lim_{p\to\infty}f^{n_{s_p}-j_{m_p}}(z_{j_{m_p}})=a$.   Then $a\in \Lambda$ since $\Lambda$ is invariant and closed.   Since $d(f^{n_{s_p}+i}(x),  f^{n_{s_p}-j_{m_p}+i}(z_{j_{m_p}}))<c$ for $0\leq i\leq p$ and all $p$.   So we have $d(f^i(v),  f^i(a))\leq c$ for all $i$.   By the positive expansiveness,   we have $v=a\in \Lambda$.   In particular,   $\omega_f(v)\subseteq \Lambda$.
	
	Case 2.   There is $\tilde{p}$ such that for all $s\geq1$,   there exists no $m\geq1$ such that $j_m\leq n_s<n_s+\tilde{p}\leq k_m$.   In other words,   for any $s\geq1$,   there is some $m_s\geq \tilde{p}$ such that $k_{m_{s}}-\tilde{p}<n_s<j_{m_{s}+1}$.    Using pigeonhole principle,   we suppose (by taking a subsequence if necessary) that $n_s=j_{m_s+1}-l$ for some $0<l\leq \tilde{p}+L$ where $L$ is defined in \eqref{definition-of-L}.   Since $d(f^{j_{m_s+1}+i}(x),  f^i(z_{j_{m_s+1}}))<c$ for  $0\leq i\leq m_s+1$ and all $s$.   By the compactness of $\Lambda$,   we suppose (by taking a subsequence if necessary) that $\lim_{s\to\infty}z_{j_{m_s+1}}=z\in \Lambda$.   Then  we have $d(f^{i+l}(v),  f^i(z))<c$ for all $i$.   By the positive expansiveness,   we have $f^l(v)=z\in \Lambda$.
	
	We thus have proved \eqref{A-f-A},   which immediately implies \eqref{omega-f-x-B}.
\end{proof}

\begin{Lem}\label{Lemma 6.2}
	Suppose that $(X,f)$ is a topologically transitive dynamical system, $f$ is a homeomorphism and $(X, f)$ has the two-sided s-limit-shadowing property. Let $\emptyset\neq\Lambda\subseteq X$ be closed $f$-invariant and internally chain transitive. Then for any nonempty open set $U\subseteq X$, there exists $x\in U$ with $\omega_f(x)=\omega_{f^{-1}}(x)=\Lambda$.
\end{Lem}
\begin{proof}
	Choose $x_0\in \Lambda$, $y_0\in U$ and $\varepsilon_1>0$ such that $B(y_0,\varepsilon_1)\subseteq U$, then by Lemma \ref{Mainlemma-prop-mu-gamma-n}, there exists $\varepsilon_2^*$ such that for any $0<\varepsilon\leq\varepsilon_2^*,\varepsilon_1$, there exists $n\in\mathbb{N}$ and $y\in B(x_0,\varepsilon)\cap f^{-n}B(x_0,\varepsilon)$ such that $d_H(\{f^i(y)\}_{i=0}^{n-1},\Lambda)<2\varepsilon$. Since $(X, f)$ has the two-sided s-limit-shadowing property, there exists $0<\delta<\frac{\varepsilon_2^*}{2}$ such that any $\delta$-limit-pseudo-orbit can be $\varepsilon_1$-limit-shadowed. Denote $\delta_k=\frac{\delta}{2^k}$ for any $k\in\mathbb{N}$, then there exists $n_k\in\mathbb{N}$ and $y_k\in B(x_0,\delta_k)\cap f^{-n_k}B(x_0,\delta_k)$ such that $d_H(\{f^i(y_k)\}_{i=0}^{n-1},\Lambda)<2\delta_k$. Since $(X,f)$ is topologically transitive, there exist a $\frac{\delta}{2}$-chain $\mathfrak{C}_{x_0y_0}$ connecting $x_0$ and $y_0$ and a  $\frac{\delta}{2}$-chain $\mathfrak{C}_{y_0x_0}$ connecting $y_0$ and $x_0$. Denote $\mathfrak{C}_i=<y_i,\cdots,f^{n_k-1}(y_i)>$, then $$\mathfrak{C}=\cdots\mathfrak{C}_3\mathfrak{C}_2\mathfrak{C}_1\mathfrak{C}_{x_0y_0}\mathfrak{C}_{y_0x_0}\mathfrak{C}_1\mathfrak{C}_2\mathfrak{C}_3\cdots$$ is a $\delta$-limit-pseudo-orbit. Hence, there exists $x\in X$ such that $\mathfrak{C}$ is $\varepsilon_1$-limit-shadowed by $x$ and $x\in B(y_0,\varepsilon_1)\subseteq U$. As a result, we have that $\omega_f(x)=\omega_{f^{-1}}(x)=\Lambda$.
\end{proof}

\begin{Lem}\label{lemma-AA}
	Suppose $f:X\to X$ is a homeomorphism on compact metric space. If $(X,  f)$ is topologically transitive and satisfies s-limit-shadowing property and uniform separation property, 
	then for any nonempty internally chain transitive closed invariant subsets $\Lambda \subsetneq X,$ there exists an internally chain transitive closed invariant subsets $\tilde{\Lambda}\subsetneq X$ such that
	\begin{equation*}
		\overline{\bigcup_{y\in\tilde{\Lambda}}\omega_f(y)}\subseteq \Lambda\subsetneq \tilde{\Lambda}.
	\end{equation*}
\end{Lem}
\begin{proof}
	Denote $U=X\setminus\Lambda$, by Lemma \ref{Lemma 6.2}, we can take $x\in U$ such that  $\omega_{f}(x)=\omega_{f^{-1}}(x)=\Lambda.$ Let $\tilde{\Lambda}=\Lambda\cup\{f^n(x)\}_{n=-\infty}^{+\infty}.$ Then $\tilde{\Lambda}$ is a closed invariant subset of $X.$ Since $x\in U=X\setminus\Lambda,$ we have $\Lambda\subsetneq \tilde{\Lambda}.$ Note that $\omega_{f}(f^{n}(x))=\omega_{f}(x)=\Lambda$ for any $n\in\mathbb{Z},$ then $\overline{\bigcup_{y\in\tilde{\Lambda}}\omega_f(y)}\subseteq \Lambda.$ This implies $\tilde{\Lambda}\subsetneq X.$ Otherwise, $\Lambda\supseteq \overline{\bigcup_{y\in X}\omega_f(y)}=X$ since $(X,f)$ is topologically transitive. This contradicts $\Lambda \subsetneq X.$
	
	Now let us prove that $\tilde{\Lambda}$ is internally chain transitive. Let $x_1,x_2\in \tilde{\Lambda}$ and $\varepsilon>0.$
	
	Case (1): if $x_1,x_2\in \Lambda,$  there is an $\varepsilon$-chain $\mathfrak{C}$ in $\Lambda$ connecting $x_1$ and $x_2$ since $\Lambda$ is internally chain transitive.
	
	Case (2): if $x_1\in \Lambda,$ $x_2=f^{n_2}(x)$ for some $n_2\in\mathbb{Z},$  there exists $N<0$ such that $f^N(x_2)=f^{n_2+N}(x)\in B(f(x_1),\varepsilon)$ since $\omega_{f^{-1}}(x)=\Lambda.$ Then $\mathfrak{C}=\langle x_1, f^{N}(x_2), f^{N+1}(x_2) \cdots,  f^{-1}(x_2)\rangle$ is an $\varepsilon$-chain in $\tilde{\Lambda}$ connecting $x_1$ and $x_2.$
	
	Case (3): if $x_2\in \Lambda,$ $x_1=f^{n_1}(x)$ for some $n_1\in\mathbb{Z},$  there exists $N>0$ such that $f^N(x_1)=f^{n_1+N}(x)\in B(x_2,\varepsilon)$ since $\omega_{f}(x)=\Lambda.$Then $\mathfrak{C}=\langle x_1, f(x_1), \cdots,  f^{N}(x_1)\rangle$ is an $\varepsilon$-chain in $\tilde{\Lambda}$ connecting $x_1$ and $x_2.$
	
	Case (4): if $x_1=f^{n_1}(x),x_2=f^{n_2}(x)$ for some $n_1,n_2\in\mathbb{Z}$ with $n_1<n_2,$  then $$\mathfrak{C}=\langle f^{n_1}(x), f^{n_1+1}(x), \cdots,  f^{n_2-1}(x)\rangle$$ is an $\varepsilon$-chain in $\tilde{\Lambda}$ connecting $x_1$ and $x_2.$
	
	Case (5): if $x_1=f^{n_1}(x),x_2=f^{n_2}(x)$ for some $n_1,n_2\in\mathbb{Z}$ with $n_1\geq n_2,$ take $x_3\in\Lambda,$ then there exists $N_1>0>N_2$ such that $f^{N_1}(x_1)=f^{n_1+N_1}(x)\in B(x_3,\varepsilon)$ and $f^{N_2}(x_2)=f^{n_2+N_2}(x)\in B(f(x_3),\varepsilon)$ since $\omega_{f}(x)=\omega_{f^{-1}}(x)=\Lambda.$
	Then $$\mathfrak{C}=\langle f^{n_1}(x), f^{n_1+1}(x), \cdots,  f^{n_1+N_1}(x),x_3,f^{n_2+N_2}(x),f^{n_2+N_2+1}(x),\cdots, f^{n_2-1}(x)\rangle$$ is an $\varepsilon$-chain in $\tilde{\Lambda}$ connecting $x_1$ and $x_2.$
\end{proof}

\begin{Lem}\label{lem-Lambda-Lambda-0}
	Suppose $(X,  f)$ is a topological dynamical system.
	Let $\Lambda,  \Lambda_0$ be two closed $f$-invariant subset of $X$ with $\Lambda_0\subseteq \Lambda$.   If for any $x\in\Lambda$,   $\omega_f(x)\subseteq \Lambda_0,$  then
	$$M(f,  \Lambda)=M(f,  \Lambda_0).$$
\end{Lem}
\begin{proof}
	Of course,   $M(f,  \Lambda)\supseteq M(f,  \Lambda_0)$.   We now prove that $M(f,  \Lambda)\subseteq M(f,  \Lambda_0)$.   In fact,   by the convexity of $M(f,  \Lambda)$ and $M(f,  \Lambda_0)$,   it is sufficient to prove that for any $\mu\in M_{erg}(f,  \Lambda)$,   $\mu\in M(f,  \Lambda_0)$.   Indeed,   choose an arbitrary generic point $x\in\Lambda$ of $\mu$.   Then $S_{\mu}\subseteq \omega_f(x)\subseteq\Lambda_0$ which yields that $\mu\in M(f,  \Lambda_0)$.   The proof is completed.
\end{proof}

The following lemma play an important role in transitioning from  \textbf{Case} $(\alpha_1\alpha_2\alpha_3\alpha_410)$ to \textbf{Case} $(\alpha_1\alpha_2\alpha_3\alpha_400)$.
\begin{mainlemma}\label{maintheorem-abstract-4}
	Suppose $(X,  f)$ is topologically transitive topologically expanding or topologically transitive topologically Anosov.  Let $\Lambda\subsetneq X$ be a nonempty internally chain transitive closed invariant subset.   Then there exists an internally chain transitive closed invariant subset $\Theta\subsetneq X$ such that $\Lambda\subsetneq \Theta$ and $M(f,  \Theta)=M(f,  \Lambda).$
\end{mainlemma}
\begin{proof}
	\textbf{Case 1:} Suppose that $(X,  f)$ is topologically transitive topologically expanding. By $\Lambda\subsetneq X$ there is an open set $U\subseteq X$ such that $\Lambda\cap U=\emptyset.$
	By Lemma \ref{lem-shadowing-s-limit-shadowing-2} and Theorem \ref{G-K-star}, we can find a $z\in U$ such that $\omega_f(z)=\Lambda.$   Let $\Lambda'=\{f^n(z)\}_{n=0}^{\infty}\cup \Lambda$.   Then $\Lambda'\supsetneq \Lambda$ is closed and $f$-invariant.   So by Lemma \ref{lem-omega-A},   there is a point $x\in X$ such that
	$$\Lambda'\subseteq\omega_f(x)\subseteq \cup_{l=0}^{\infty}f^{-l}\Lambda'.$$
	In particular,
	$$\overline{\bigcup_{y\in\omega_f(x)}\omega_f(y)}\subseteq \Lambda'\subseteq\omega_f(x)\neq X.  $$
	Let $\Theta=\omega_f(x)$.   Then $\Theta$ is internally chain transitive  by Lemma \ref{Lem-omega-naturally-in-ICT}.   By Lemma \ref{lem-Lambda-Lambda-0},
	$M(f,  \Theta)=M(f,  \Lambda')=M(f, \Lambda).$
	
	\textbf{Case 2:} Suppose that $(X,  f)$ is topologically transitive topologically Anosov, then by Lemma \ref{lem-shadowing-s-limit-shadowing} and Lemma \ref{lemma-AA}, there exists an internally chain transitive closed invariant subset $\Theta\subsetneq X$ such that $$\overline{\bigcup_{y\in\Theta}\omega_f(y)}\subseteq \Lambda\subsetneq\Theta.  $$ By Lemma \ref{lem-Lambda-Lambda-0},
	$M(f,  \Theta)=M(f, \Lambda).$
\end{proof}

\subsection{Proof of Theorem \ref{Theorem B}: \textbf{Case} $(1\alpha_2\alpha_3\alpha_4\alpha_50)$} First of all, let's recall a basic property of topologically transitive dynamical system.
\begin{Lem}\label{lemma-B}
	Suppose $(X,  f)$ is a transitive topological dynamical system. Let $U\subseteq X$ be a nonempty open set. If $\Lambda\subsetneq X$ is a closed invariant subset, then there is a nonempty open set $V$ in $X$ such that $V\subseteq U$ and $V\cap \Lambda=\emptyset.$
\end{Lem}
\begin{proof}
	Since $(X,  f)$ is transitive, there is $x\in U$ such that $\omega_f(x)=X.$ By $\Lambda\subsetneq X,$ we have $x\not\in \Lambda.$ Otherwise, $\Lambda \supseteq \omega_f(x)=X.$ Denote $d=\min\{\operatorname{dist}(x,\Lambda),\operatorname{dist}(x,X\setminus U)\}>0.$ Take $V=B(x,\frac{d}{2}).$
\end{proof}
Next, we prove the following lemma, which will help us find a proper internally chain transitive closed invariant subset with sufficiently good properties and large entropy.
\begin{mainlemma}\label{Lemma B New}
	Suppose that $(X,  f)$ is topologically transitive topologically expanding or topologically transitive topologically Anosov. Then for any $0<\alpha<h_{top}(f)$, there exists a nonempty subset $\Lambda\subsetneq X$ such that
	\begin{enumerate}[(1)]
		\item $(\Lambda,f)$ is not uniquely ergodic;
		\item $(\Lambda,f)$ is internally chain transitive;
		\item $(\Lambda,f)$ has refined entropy-dense property;
		\item $C^*_\Lambda=\Lambda.$
		\item $h_{top}(f,\Lambda)>h_{top}(f)-\alpha$.
	\end{enumerate}
\end{mainlemma}
To prove Lemma \ref{Lemma B New}, we need some preparations. To begin with, we prove a lemma similar to Lemma \ref{horseshoe} and Lemma \ref{horseshoe2}.
\begin{Lem}\label{Lemma 6.9}
	Suppose that $(X,  f)$ is topologically expanding (resp. topologically Anosov) with $h_{top}(f)>0$. Then for any nonempty open set $U\subset X$, for any $0<\alpha<h_{top}(f)$,   there are $m,  k\in \mathbb{N}$,   $\frac{\log m}{k}>\alpha$ and a closed set $\Lambda\subset X$
	invariant under $f^k$ such that 
	\begin{enumerate}[(1)]
		\item $\Lambda\subseteq U$, $\bigcup_{i=0}^{k-1}f^i(\Lambda)\subsetneq X$ and $f^i(\Lambda)\cap f^j(\Lambda)$ for any $0\leq i<j\leq k-1$;
		\item there is a conjugate map $$\pi: (\Lambda,   f^k)\to (\Sigma_{m}^+,  \sigma)\text{ }(\text{resp. }(\Sigma_{m},  \sigma)).$$
	\end{enumerate}
\end{Lem}

\begin{proof}
	Choose $\eta>0$ such that $0<\alpha<h_{top}(f)-\frac{9\eta}{8}$. By Lemma  \ref{Lemma 6.5}, there exists $\nu\in M(f,  X)$ such that $(S_\mu,f)$ is minimal and $h_{\mu}>h_{top}(f)-\frac{\eta}{8}.$ It is clear that $\{\mu\}\neq M(f,  X)$ and thus $d_H(\{\mu\},  M(f,  X))>0.$ Fix $x\in S_\mu$, $\varepsilon_0>0$ and $0<\zeta<d_H(\{\mu\},  M(f,  X))$.

Choose $x\in U$ and $\hat{\varepsilon}>0$ such that $B(x,\hat{\varepsilon})\subset U$. By Corollary \ref{Maincor-mu-gamma-n-expansive},  there exist $\varepsilon^*>0$ and $0<\tilde{\varepsilon}^*<\varepsilon^*$  such that for any $0<\frac{\delta}{2}<\tilde{\varepsilon}^*$, there exists an $n_\mu\in\mathbb{N}$ such that for any $p\in\mathbb{N}$,   there exists an $(pn_\mu,  \frac{\varepsilon^*}{3})$-separated set $\Gamma_{pn_\mu}^{\mu}$ with
\begin{description}
	\item[(a)]\label{lem-B-aBasic} $\Gamma_{pn_\mu}^{\mu}\subseteq P_{pn_\mu}(f)\cap X_{pn_\mu,  \mathcal{B}(\mu,  \frac{\zeta}{4})}\cap B(x,  \frac{\delta}{2})$;
	\item[(b)]\label{lem-B-bBasic} $\frac{\log |\Gamma_{pn_\mu}^{\mu}|}{pn_\mu}>h_{\mu}-\frac{\eta}{8}$.
\end{description}
We can assume that $\frac{\varepsilon^*}{3}<\frac{c}{4}$ where $c>0$ is the expansive constant. Let $s(n,\frac{\varepsilon^*}{3})$ denote the largest cardinality of any $(n,\frac{\varepsilon^*}{3})$-separated set of $X$, then by \cite[Theorem 7.11]{Walters} one has
\begin{equation}
	h_{top}(f)=\limsup_{n\to\infty}\frac{1}{n}\log s(n,\frac{\varepsilon^*}{3}).
\end{equation}
Then there exists $N\in\mathbb{N}$ such that for any $n\geq N,$ one has 
\begin{equation}\label{equation-AB}
	s(n,\frac{\varepsilon^*}{3})<e^{n(h_{top}(f)+\frac{\eta}{4})}.
\end{equation}

Set $\varepsilon=\min\{\frac{\zeta}{4},  \frac{\tilde{\varepsilon}^*}{27},\frac{\varepsilon_0}{2},\frac{\hat{\varepsilon}}{3}\}.$  Then there exists a $0<\delta<\varepsilon$ such that any $\delta$-pseudo-orbit can be $\varepsilon$-shadowed by some point in $X.$ Set $n=p_0n_\mu$ where $p_0$ is large enough such that 
\begin{equation}\label{equation-AA}
	n\geq 2N,\ e^{n(h_{\mu}-\frac{\eta}{8})}-n\geq e^{n(h_{\mu}-\frac{\eta}{4})}\ \text{ and } e^{n(h_{top}(f)-\frac{\eta}{4})}>\lceil\frac{n}{2}\rceil e^{\lceil\frac{n}{2}\rceil(h_{top}(f)+\frac{\eta}{4})}+\sum_{m=1}^{N-1}| P_{m}^*(f)|
\end{equation} 
where $P_{m}^*(f)$ is the set of periodic points with minimal period $m.$
Then $P_{n_\mu}(f)\subseteq P_n(f)$ by definition and furthermore,   we can obtain an $(n,  \frac{\varepsilon^*}{3})$-separated set $\Gamma_n$ with
\begin{description}
	\item[(a)]\label{lem-B-aBasic2} $\Gamma_{n}\subseteq P_{n}(f)\cap X_{n,  \mathcal{B}(\mu,  \frac{\zeta}{4})}\cap B(x, \frac{\delta}{2})$;
	\item[(b)]\label{lem-B-bBasic2} $\frac{\log |\Gamma_{n}|}{n}>h_{\mu}-\frac{\eta}{8}$.
\end{description}
Since periodic points in $\Gamma_n$ with same period $l_0$ for some $l_0\in\mathbb{N}$ are $(l_0,\frac{\varepsilon^*}{3})$-separated, by \eqref{equation-AB} and \eqref{equation-AA} we have
\begin{equation*}
	\begin{split}
		\sum_{i=1}^{\lceil\frac{n}{2}\rceil}| P_{i}^*(f)\cap \Gamma_n|\leq &\sum_{i=N}^{\lceil\frac{n}{2}\rceil} s(i,\frac{\varepsilon^*}{3})+\sum_{i=1}^{N-1}| P_{i}^*(f)|\\
		<&\sum_{i=N}^{\lceil\frac{n}{2}\rceil} e^{{i}(h_{top}(f)+\frac{\eta}{4})}+\sum_{i=1}^{N-1}| P_{i}^*(f)|\\
		\leq&{\lceil\frac{n}{2}\rceil}e^{{\lceil\frac{n}{2}\rceil}(h_{top}(f)+\frac{\eta}{4})}+\sum_{i=1}^{N-1}| P_{i}^*(f)|\\
		< &e^{n(h_{top}(f)-\frac{\eta}{4})}<e^{n(h_{\mu}-\frac{\eta}{8})}<|\Gamma_{n}|.
	\end{split}
\end{equation*}
Thus there exists $x_0\in\Gamma_n$ with minimal period $n.$ Together with $\frac{\varepsilon^*}{3}<\frac{c}{4},$ the only sub-intervals of length $n$ of $\langle x_0,  f(x_0),  \cdots,  f^{n-1}(x_0),x_0,  f(x_0),  \cdots,  f^{n-1}(x_0)\rangle$ that are $\frac{\varepsilon^*}{9}$-shadowed by $\langle x_0,  f(x_0),  \cdots,  f^{n-1}(x_0)\rangle$ are the initial and the final sub-intervals.
By the separation assumption we have $$|\{y\in\Gamma_n:d_n(y,f^j(x_0))<\frac{\varepsilon^*}{9} \text{ for some }0\leq j\leq n-1\}|\leq n.$$ Consequently, by \eqref{equation-AA} one can find a subset $\widetilde{\Gamma}_n\subset \Gamma_n$ with $|\widetilde{\Gamma}_n|>e^{n(h_{\mu}-\frac{\eta}{4})}$ such that $d_n(y,f^j(x_0))\geq \frac{\varepsilon^*}{9}$ for any $y\in\widetilde{\Gamma}_n$ and $0\leq j\leq n-1.$ 
Denote $r=|\widetilde{\Gamma}_n|.$ Enumerate the elements of each $\widetilde{\Gamma}_n$ by $\widetilde{\Gamma}_n=\{p_1,  \cdots,  p_{r}\}$.  
Take $l$ large enough such that
\begin{equation}\label{equation-AC}
	\frac{1}{l}<\frac{\zeta}{16} \text{ and }\frac{(l-2)\log|\widetilde{\Gamma}_n|}{nl}>h_{\mu}-\eta>\alpha.
\end{equation}
Denote $$k=nl\text{ and }m=r^{l-2}.$$

Now let $\Gamma=\widetilde{\Gamma}_n\times\widetilde{\Gamma}_n\times\cdots\times\widetilde{\Gamma}_n$ whose element is $\underline{y}=(y_1,  \cdots,  y_{l-2})$ with $y_j\in\widetilde{\Gamma}_n$ for $1\leq j\leq l-2.$  For any $y\in X$,  let $\mathfrak{C}_y^n=\langle y,  fy,  \cdots,  f^{n-1}y\rangle$. Then for $\underline{y}\in \Gamma$  we define the following pseudo-orbit:
$$\mathfrak{C}_{\underline{y}}=\mathfrak{C}_{x_0}^n\mathfrak{C}_{x_0}^n\mathfrak{C}_{y_1}^n\mathfrak{C}_{y_2}^n\cdots\mathfrak{C}_{y_{l-2}}^n.$$
It is clear that $\mathfrak{C}_{\underline{y}}$ is a $\delta$-pseudo-orbit.   Moreover,   one notes that we can freely concatenate such $\mathfrak{C}_{\underline{y}}$s to constitutes a $\delta$-pseudo-orbit. We write $\mathfrak{C}_{\underline{y}}=\langle \omega_0,  \omega_1,  \cdots,  \omega_{nl-1}\rangle.$ If $d(f^k(z),\omega_{k+1})\leq \varepsilon$ for any $0\leq k\leq ln-1$ then by Lemma \ref{lem:prohorov} and \eqref{equation-AC} one has 
\begin{equation}\label{eq-AA}
	\begin{split}
		\rho(\mathcal{E}_{ln}(z),  \mu)\leq &\rho(\mathcal{E}_{ln}(z), \frac{1}{n(l-2)} \sum_{k=1}^{n(l-2)}\delta_{\omega_k})+\rho(\frac{1}{n(l-2)} \sum_{k=1}^{n(l-2)}\delta_{\omega_k},  \mu)\\
		\leq &\varepsilon+\frac{4}{l}+\frac{1}{l-2}\sum_{j=1}^{l-2}\rho(\mathcal{E}_{n}(y_j),  \mu)\\
		<&\varepsilon+\frac{4}{l}+\frac{\zeta}{4}<\frac34\zeta.
	\end{split}
\end{equation}
Now we define
$$\Sigma_{m}^+=\Sigma_{r^{l-2}}^+:=\theta_{0}\theta_{1}\theta_{2}\dots:\theta_{j}=(\theta_{j,1},  \cdots,  \theta_{j,l-2})\in \Gamma \text{ for any } j\in\mathbb{Z}^+ \}.$$ 

$$(\text{resp. }\Sigma_{m}=\Sigma_{r^{l-2}}:=\{\theta=\dots\theta_{-2}\theta_{-1}\theta_{0}\theta_{1}\theta_{2}\dots:\theta_{j}=(\theta_{j,1},  \cdots,  \theta_{j,l-2})\in \Gamma \text{ for any } j\in\mathbb{Z} \}.)$$ 
Then $(\Sigma_{m}^+,\sigma)$ (resp. $(\Sigma_{m},\sigma)$) is a full shift and for each $\theta=\theta_{0}\theta_{1}\theta_{2}\dots$ (resp. $\theta=\dots\theta_{-2}\theta_{-1}\theta_{0}\theta_{1}\theta_{2}\dots$) in $\Sigma_{m}^+$ (resp. $\Sigma_{m}$)

$$\mathfrak{C}_{\theta}=\mathfrak{C}_{\theta_{0}}\mathfrak{C}_{\theta_{1}}\mathfrak{C}_{\theta_{2}}\dots\text{ }(\text{resp. }\mathfrak{C}_{\theta}=\dots\mathfrak{C}_{\theta_{-2}}\mathfrak{C}_{\theta_{-1}}\mathfrak{C}_{\theta_{0}}\mathfrak{C}_{\theta_{1}}\mathfrak{C}_{\theta_{2}}\dots)$$ is a $\delta$-pseudo-orbit. We write $\mathfrak{C}_{\theta}=\omega_{0}\omega_{1}\omega_{2}\dots$ (resp. $\mathfrak{C}_{\theta}=\dots\omega_{-2}\omega_{-1}\omega_{0}\omega_{1}\omega_{2}\dots$),  by the shadowing property,
$$Y_{\theta}=\{z\in X:d(f^{j}(z),  \omega_j)\leq\varepsilon, \text{ for any }j\in\mathbb{Z}^+\}$$
$$(\text{resp. }Y_{\theta}=\{z\in X:d(f^{j}(z),  \omega_j)\leq\varepsilon,  \text{ for any }j\in\mathbb{Z}\})$$
is nonempty and closed. Since $\omega_0=x_0$, one has $$Y_\theta\subset\overline{B(x_0,\varepsilon)}\subseteq B(x_0,\frac{\hat{\varepsilon}}{2})\subseteq B(x, \frac{\delta}{2}+\frac{\hat{\varepsilon}}{2})\subseteq B(x, \hat{\varepsilon})\subseteq U.$$

We claim that $Y_{\theta}\cap Y_{\theta'}=\emptyset$ for any $\theta\neq\theta'$ in $\Sigma_{m}^+$ (resp. $\Sigma_{m}$). 
By $\theta\neq\theta'$ there is $t\in\mathbb{Z}^+$ (resp. $t\in\mathbb{Z}$) and $1\leq s\leq l-2$ such that $\theta_{t,s}\neq\theta'_{t,s}.$ Since $\theta_{t,s}$ and $\theta'_{t,s}$ are $(n,\frac{\varepsilon^*}{3})$-separated, we have $d_n(f^{lnt+sn+n}(z),f^{lnt+sn+n}(z'))>\varepsilon^*/3-2\varepsilon>\varepsilon^*/9$ for any $z\in Y_{\theta}$ and $z'\in Y_{\theta'}.$ So $Y_{\theta}\cap Y_{\theta'}=\emptyset.$
Then we can define the following disjoint union:
$$\Lambda=\bigsqcup_{\theta\in \Sigma_{m}^+}Y_{\theta}\text{ }(\text{resp. }\Lambda=\bigsqcup_{\theta\in \Sigma_{m}}Y_{\theta}).$$
Note that $\Lambda\subseteq U$, $f^{nl}(Y_{\theta})\subset Y_{\sigma(\theta)}$ and $k=nl$.   Then $f^k(\Lambda)\subset\Lambda.$
Therefore,   if we define $\pi:\Lambda\to \Sigma_{m}^+\text{ (resp. }\Sigma_{m})$ as
\begin{equation*}
	\pi(x):=\theta \textrm{ for all } x\in Y_{\theta}~\textrm{with}~\theta\in \Sigma_{m}^+\text{ (resp. }\Sigma_{m}),
\end{equation*}
then $\pi$ is surjective by the shadowing property.   Follow the same line of Lemma \ref{horseshoe2}, we can prove $\pi$ is continuous. Since both $\Lambda$ and $\Sigma_m^+\text{ }(\text{resp. }\Sigma_m)$ are compact metric space, we have that $\pi$ is a homeomorphism and thus $\pi$ is a conjugate map. Hence, $f^k(\Lambda)=\Lambda$.

Next we show that $f^{k''}(\Lambda)\cap f^{k'}(\Lambda)=\emptyset$ for any $0\leq k''<k'\leq k-1.$ It is enough to show that $\Lambda\cap f^{\tilde{k}}(\Lambda)=\emptyset$ for any $1\leq\tilde{k}\leq k-1$. In fact, if $f^{i}(\Lambda)\cap f^{j}(\Lambda)\neq\emptyset$ for some $0\leq i<j\leq k-1$, then $$\emptyset\neq f^{N-i}(f^{i}(\Lambda)\cap f^{j}(\Lambda))\subset f^N(\Lambda)\cap f^{j+N-i}(\Lambda)=\Lambda\cap f^{j-i}(\Lambda).$$ Now, suppose that $\Lambda\cap f^{\tilde{k}}(\Lambda)\neq\emptyset$ for some $1\leq\tilde{k}\leq k-1$ then for any $z\in \Lambda\cap f^{\tilde{k}}(\Lambda),$ there exist $\theta,\  \theta'\in \Sigma_{m}^+\text{ }(\text{resp. }\Sigma_m)$ such that 
$$d(f^j(z),  \omega_{j})\leq\varepsilon \text{ and } d(f^j(z),  \omega'_{j+\tilde{k}})\leq\varepsilon,\text{ for any } j\in\mathbb{Z}^+\text{ (resp. }\mathbb{Z})$$
where $$\mathfrak{C}_{\theta}=\omega_{0}\omega_{1}\omega_{2}\dots\text{ and }\mathfrak{C}_{\theta'}=\omega'_{0}\omega'_{1}\omega'_{2}\dots.$$ $$(\text{resp. } \mathfrak{C}_{\theta}=\dots\omega_{-2}\omega_{-1}\omega_{0}\omega_{1}\omega_{2}\dots\text{ and }\mathfrak{C}_{\theta'}=\dots\omega'_{-2}\omega'_{-1}\omega'_{0}\omega'_{1}\omega'_{2}\dots.)$$ Then we have 
\begin{equation}\label{eq-AD}
	d(\omega_{j},  \omega'_{j+\tilde{k}})\leq2\varepsilon, \text{ for any } j\in\mathbb{Z}^+\text{ (resp. }\mathbb{Z}).
\end{equation}

Case (1): If $1\leq \tilde{k}\leq n-1,$ then \eqref{eq-AD} implies $d(\omega_{j},  \omega'_{j+\tilde{k}})\leq2\varepsilon<\frac{\varepsilon^*}{9}\text{ for any }0\leq j\leq n-1.$ Note that $\omega_{0}\omega_{1}\omega_{2}\dots\omega_{2n-1}=\omega'_{0}\omega'_{1}\omega'_{2}\dots\omega'_{2n-1}=\langle x_0,  fx_0,  \cdots,  f^{n-1}x_0,x_0,  fx_0,  \cdots,  f^{n-1}x_0\rangle,$ this contradicts that the minimal period of $x_0$ is $n.$ 

Case (2): If $\tilde{k}= n,$ then \eqref{eq-AD} implies $d(\omega_{j+n},  \omega'_{j+2n})\leq2\varepsilon<\frac{\varepsilon^*}{9}\text{ for any }0\leq j\leq n-1.$ Note that $\omega_{n}\omega_{n+1}\dots\omega_{2n-1}=\langle x_0,  fx_0,  \cdots,  f^{n-1}x_0\rangle$ and $\omega'_{2n}\in \widetilde{\Gamma}_n,$ this contradicts that $d_n(y,f^j(x_0))\geq \frac{\varepsilon^*}{9}$ for any $y\in\widetilde{\Gamma}_n$ and $0\leq j\leq n-1.$ 

Case (3): If $n<\tilde{k}\leq n(l-1),$ then $\tilde{k}=tn+s$ for some $1\leq t\leq l-2$ and $0\leq s\leq n-1.$ Thus \eqref{eq-AD} implies $d(\omega_{n-s+j},  \omega'_{(t+1)n+j})\leq2\varepsilon\text{ for any }0\leq j\leq n-1.$ Note that $\omega_{0}\omega_{1}\omega_{2}\dots\omega_{2n-1}=\langle x_0,  fx_0,  \cdots,  f^{n-1}x_0,x_0,  fx_0,  \cdots,  f^{n-1}x_0\rangle,$ and $\omega'_{(t+1)n}\in \widetilde{\Gamma}_n,$ this contradicts that $d_n(y,f^j(x_0))\geq \frac{\varepsilon^*}{9}$ for any $y\in\widetilde{\Gamma}_n$ and $0\leq j\leq n-1.$ 

Case (4): If $n(l-1)< \tilde{k}\leq nl-1,$ then \eqref{eq-AD} implies $d(\omega_{nl-\tilde{k}+j},  \omega'_{nl+j})\leq2\varepsilon\text{ for any }0\leq j\leq n-1.$ Note that $\omega_{0}\omega_{1}\omega_{2}\dots\omega_{2n-1}=\omega'_{nl}\omega'_{nl+1}\omega'_{nl+2}\dots\omega'_{(l+2)n-1}=\langle x_0,  fx_0,  \cdots,  f^{n-1}x_0,x_0,  fx_0,  \cdots,  f^{n-1}x_0\rangle,$ and $1\leq nl-\tilde{k}<n,$ this contradicts that the minimal period of $x_0$ is $n.$ 

So we have $\Lambda\cap f^{\tilde{k}}(\Lambda)=\emptyset$ for any $1\leq\tilde{k}\leq k-1$ and thus $f^{k''}(\Delta)\cap f^{k'}(\Delta)=\emptyset$ for any $0\leq k''<k'\leq k-1.$ 

For any ergodic measure $\xi\in M(f,\bigcup_{i=0}^{k-1}f^i(\Lambda))$,   pick an arbitrary generic point $z$ of $\xi$ in $\Lambda$. 
Then $$\rho(\mathcal{E}_{ln}(f^{tln}(z)),  \mu)<\frac34\zeta\text{ for any } t\in\mathbb{N}$$ by \eqref{eq-AA}.
In addition,   we have $\xi=\lim_{j\to\infty}\mathcal{E}_j(z)=\lim_{t\to\infty}\mathcal{E}_{tln}(z)$.   So we have $$\rho(\xi,  \mu)=\lim_{t\to\infty}\rho(\mathcal{E}_{tln}(z),  \mu)\leq \frac34\zeta.$$   By the ergodic decomposition theorem,   we obtain that $$M(f,  \bigcup_{i=0}^{k-1}f^i(\Lambda))\subseteq  \mathcal{B}(\mu,  \zeta).$$
As a result,   $\bigcup_{i=0}^{k-1}f^i(\Lambda)\subsetneq X$.   For otherwise,   $$d_H(\{\mu\},  M(f,  \bigcup_{i=0}^{k-1}f^i(\Lambda)))=d_H(\{\mu\},  M(f,  X))>\zeta,$$   it is a contradiction.
\end{proof} 

To know the property of $(\bigcup_{i=0}^{k-1}f^i(\Lambda),f)$ for $(\Lambda,f^k)$ in Lemma \ref{Lemma 6.9}, we will show the following lemma.
\begin{Lem}\label{relative-specification}
	Suppose $(X,  f)$ is a topological dynamical system. Consider $\Delta\subseteq X$ which satisfies $f^n(\Delta)\subseteq \Delta$ for some $n\in\mathbb{N}$ and let $\Lambda=\bigcup_{i=0}^{n-1}f^i(\Delta).$ If $(\Delta,  f^n)$ has periodic specification property,   then $(\Lambda,  f)$ has periodic gluing orbit property.
\end{Lem}
\begin{proof}
	Fix $\varepsilon>0.$  Let $\eta \in(0, \varepsilon)$ be such that for every $y,\ z \in X$ and every $i=0, \cdots, n-1$ one has $d\left(f^{i}(y), f^{i}(z)\right) \leq \varepsilon$ provided $d(y, z) \leq \eta$. We claim that it is enough to set $K_\varepsilon=nM_\eta+3n,$ where $M_\eta$ is  defined in Definition \ref{definition of specification}. Fix $x_{1}, \dots, x_{k}\in X$ and $n_{1}, \dots, n_{k} \geq 1.$ 
	For any $1\leq i\leq k$ let $\tilde{x}_{i} \in \Delta$ be such that there exists $m_{i} \in\{0, \ldots, n-1\}$ such that $f^{m_{i}}\left(\tilde{x}_{i}\right)=x_{i}$. And for any $1\leq i\leq k$ there is $q_i\geq 0$ and $r_{i} \in\{0, \ldots, n-1\}$ such that $n_i=q_in+r_i.$ Let $\tilde{n}_i=q_i+2.$
	By periodic specification property of $(\Delta,  f^n)$ there is $\tilde{x}$ in $\Delta$ such that $d(f^{jn}(\tilde{x}),f^{jn}(\tilde{x}_{1}))\leq \eta$ for every $0\leq j \leq \tilde{n}_{1}-1$ and $$d(f^{(j+\tilde{n}_{1}+M_\eta+\dots +\tilde{n}_{i-1}+M_\eta)n}(\tilde{x}),f^{jn}(\tilde{x}_{i}))\leq \eta$$ for all $0\leq j \leq \tilde{n}_{i}-1,$ $2 \leq i \leq k,$ and $f^{(\tilde{n}_{1}+M_\eta+\dots +\tilde{n}_{k}+M_\eta)n}(\tilde{x})=\tilde{x}.$
	Then we have $d(f^{j}(\tilde{x}),f^{j}(\tilde{x}_{1}))\leq \varepsilon$ for every $0\leq j \leq \tilde{n}_{1}n-1$ and $$d(f^{j+(\tilde{n}_{1}+M_\eta+\dots +\tilde{n}_{i-1}+M_\eta)n}(\tilde{x}),f^{j}(\tilde{x}_{i}))\leq \varepsilon$$ for all $0\leq j \leq \tilde{n}_{i}n-1,$ $2 \leq i \leq k.$
	Set $x=f^{m_1}(\tilde{x}),$ $p_i=2n-m_i-r_i+nM_\eta+m_{i+1}$ for any $1 \leq i \leq k-1,$ and $p_k=2n-m_k-r_k+nM_\eta+m_1.$ Then $d(f^{j}(x),f^{j}(x_{1}))\leq \varepsilon$ for every $0\leq j \leq n_{1}-1$ and $$d(f^{j+n_{1}+p_{1}+\dots +n_{i-1}+p_{i-1}}(x),f^{j}(x_{i}))\leq \varepsilon$$ for all $0\leq j \leq n_{i}-1$ and $2 \leq i \leq k,$ and $f^{n_{1}+p_{1}+\dots +n_{k}+p_{k}}(x)=x.$ So by $p_i\leq 2n+nM_\eta+n=K_\varepsilon$ we completes the proof.
\end{proof}

Now, we give the proof of Lemma \ref{Lemma B New}.

\emph{Proof of Lemma \ref{Lemma B New}.} Suppose that $(X,f)$ is topologically transitive topologically expanding (resp. topologically transitive topologically Anosov) and fix $\alpha>0$, then by Lemma \ref{Lemma 6.9}, there are $m,  k\in \mathbb{N}$,   $\frac{\log m}{k}>h_{top}(f)-\alpha$ and a closed set $\Delta\subset X$
invariant under $f^k$ such that $\bigcup_{i=0}^{k-1}f^i(\Delta)\subsetneq X$ and there is a conjugate map $$\pi: (\Delta,   f^k)\to (\Sigma_{m}^+,  \sigma)\text{ }(\text{resp. }(\Sigma_{m},  \sigma)).$$ Since $(\Sigma_{m}^+,  \sigma)\text{ }(\text{resp. }(\Sigma_{m},  \sigma))$ has periodic specification property by \cite[Proposition 21.2]{DGS},  one has that $(\Delta,  f^k)$ also has periodic specification property. Hence, by Lemma \ref{relative-specification}, $(\Lambda,  f)$ has periodic gluing orbit property where $\Lambda=\bigcup_{i=0}^{k-1}f^i(\Delta)\subsetneq X.$ As a result, $(\Lambda,f)$ is topologically transitive and thus internally chain transitive. By Corollary \ref{prop-entropy-dense-for-shadowing}, $(\Lambda,  f)$ has the refined entropy-dense property. It can be checked that periodic points are dense in $\Lambda,$ then there is a $\nu\in M(f,  \Lambda)$ with $S_\nu=\Lambda$ \cite[Proposition 21.12]{DGS}.  This implies that $C^*_\Lambda=\Lambda.$
Since $\pi$ is a conjugation,
$$h_{top}(f,  \Lambda)=\frac1kh_{top}(f^k,  \Delta)=\frac1kh_{top}(\sigma,  \Sigma_m^+)=\frac{\log m}{k}>h_{top}(f)-\alpha>0.$$
Therefore, $(\Lambda,f)$ is not uniquely ergodic. \qed

Now, we give the proof of Theorem \ref{Theorem B}: \textbf{Case} $(1\alpha_2\alpha_3\alpha_4\alpha_50)$.

\emph{Proof of Theorem \ref{Theorem B}: \textbf{Case} $(1\alpha_2\alpha_3\alpha_4\alpha_50)$.} Given $0<\eta<h_{top}(f)$, by Lemma \ref{Lemma B New}, there exists a nonempty subset $\Lambda\subsetneq X$ such that
\begin{enumerate}[(1)]
	\item $(\Lambda,f)$ is not uniquely ergodic;
	\item $(\Lambda,f)$ is internally chain transitive;
	\item $(\Lambda,f)$ has refined entropy-dense property;
	\item $C^*_\Lambda=\Lambda.$
	\item $h_{top}(f,\Lambda)>h_{top}(f)-\eta$.
\end{enumerate} 
By Lemma \ref{maintheorem-abstract-4}, there exists an internally chain transitive closed invariant subset $\Theta\subsetneq X$ such that $\Lambda\subsetneq \Theta$ and $M(f,  \Theta)=M(f,  \Lambda).$ As a result, $C_\Theta^*=\Lambda\subsetneq \Theta$. For $\Lambda$, by item (1) of Lemma \ref{Lemma-AA-New}, for any $i\in\{1,2,\cdots,6\}$, there is $K_i\subseteq M(f,  \Lambda)$ such that 
$$G_{K_1}^\Lambda\cap NRec(f)\subseteq\mathscr{C}(101110), G_{K_2}^\Lambda\cap NRec(f)\subseteq\mathscr{C}(101010), G_{K_3}^\Lambda\cap NRec(f)\subseteq\mathscr{C}(110110),$$ 	$$G_{K_4}^\Lambda\cap NRec(f)\subseteq\mathscr{C}(100110), G_{K_5}^\Lambda\cap NRec(f)\subseteq\mathscr{C}(110010), G_{K_6}^\Lambda\cap NRec(f)\subseteq\mathscr{C}(100010)$$		
and $$\inf\{h_\mu (f):  \mu\in K_i\}>h_{top}(f,\Lambda)-\eta>h_{top}(f)-2\eta.$$ For $\Theta$, by item (2) of Lemma \ref{Lemma-AA-New}, for any $i\in\{1',2',\cdots,6'\}$, there is $K_i\subseteq M(f,  \Theta)$ such that 
$$G_{K_1}^\Theta\cap NRec(f)\subseteq\mathscr{C}(101100), G_{K_2}^\Theta\cap NRec(f)\subseteq\mathscr{C}(101000), G_{K_3}^\Theta\cap NRec(f)\subseteq\mathscr{C}(110100),$$ 	$$G_{K_4}^\Theta\cap NRec(f)\subseteq\mathscr{C}(100100), G_{K_5}^\Theta\cap NRec(f)\subseteq\mathscr{C}(110000), G_{K_6}^\Theta\cap NRec(f)\subseteq\mathscr{C}(100000)$$		
and $$\inf\{h_\mu (f):  \mu\in K_i\}>h_{top}(f,\Theta)-\eta>h_{top}(f)-2\eta.$$

Given a nonempty open set $U\subset X$, since $(X,f)$ is topologically transitive and $\Theta\subsetneq X$, by Lemma \ref{lemma-B}, there is a nonempty open set $V$ in $X$ such that $V\subseteq U$ and $V\cap \Theta=\emptyset.$ For any $i\in\{1,2,\cdots6,1',2',\cdots,6'\}$, by Lemma \ref{lem-shadowing-s-limit-shadowing}, Lemma \ref{lem-shadowing-s-limit-shadowing-2} and Theorem \ref{G-K-star}, we have that $$h_{top} (f,  G_{K_i}^{\Lambda}\cap V)=h_{top} (f,  G_{K_i}^{\Theta}\cap V)=\inf\{h_\mu (f):  \mu\in K_i\}>h_{top}(f)-2\eta.$$
Since $V\cap\Theta=\emptyset$ and $V\subseteq U$, we have that $$G_{K_i}^\Lambda\cap V\subset G_{K_i}^\Lambda\cap NRec(f)\cap U$$ and $$G_{K_i}^\Theta\cap V\subset G_{K_i}^\Theta\cap NRec(f)\cap U.$$ Hence, for any $1\alpha_2\alpha_3\alpha_4\alpha_50$, we have that $$h_{top}(f,NRec(f)\cap\mathscr{C}_i\cap U)> h_{top}(f)-2\eta.$$ Finally, by the arbitrariness of $\eta$, we finish the proof. \qed

\subsection{Proof of Theorem \ref{Theorem B}: \textbf{Case} $(0\alpha_2\alpha_3\alpha_4\alpha_50)$} To give the proof, we will construct a dynamical system with large topological entropy, having a unique minimal subsystem $Y$, and three invariant measures $\mu_1$, $\mu_2$ and $\mu_3$ such that $S_{\mu_i}$ is internally chain transitive and $Y\subsetneq S_{\mu_1}\subsetneq S_{\mu_2}\subsetneq S_{\mu_3}$.  Firstly, we will construct such dynamical systems in full shifts.  Let's recall basic notations. A word $A=a_1a_2\cdots a_s$ is said to be appeared in a word $B=b_1b_2\cdots b_t$, denoted by $a_1a_2\cdots a_s\prec b_1b_2\cdots b_t$ or $A\prec B$ if there exists $1\leq l\leq t-s+1$ such that $a_1a_2\cdots a_s=b_lb_{l+1}\cdots b_{l+s-1}$. We use $AB$ to denote $a_1a_2\cdots a_sb_1b_2\cdots b_t$ and $A^m$ to denote $A^{m-1}A$ for $m\geq2$. Given $x\in(\Sigma_m^+,\sigma)$ (resp. $(\Sigma_m,\sigma)$), a word $A=a_1a_2\cdots a_s$ is said to be appeared in $x$ if there exists $j\in\mathbb{Z}^+$ (resp. $\mathbb{Z}$) such that $A\prec x_jx_{j+1}\cdots x_{j+s-1}$, we also say that $A$ is a word of $x$. The length of a word $A=a_1a_2\cdots a_s$ is denoted as $|A|=s$.
The following lemma is due to the definition of almost periodic points. 
\begin{Lem}\cite[Theorem 26.7]{DGS}\label{Lemma 6.13}
	Given $x\in(\Sigma_m^+,\sigma)$ (resp. $(\Sigma_m,\sigma)$), then $x$ is almost periodic if and only if for any $j\geq0$, there exists $N\geq1$ such that for any $q\geq 0$ (resp. $q\in\mathbb{Z}$), $x_0x_1\cdots x_j\prec x_qx_{q+1}\cdots x_{q+N-1}$ (resp. $x_{-j}\cdots x_{-1}x_0x_1\cdots x_j\prec x_qx_{q+1}\cdots x_{q+N-1}$).
\end{Lem}

\begin{Lem}\label{Lemma-unique-minimal}
	Given an integer $m>1$, a almost periodic point $z$ of $(\Sigma_m^+,\sigma)$ (resp. $(\Sigma_m,\sigma)$)and a recurrent point $x$ of $(\Sigma_m^+,\sigma)$ satisfying that for any $j\geq0$, there exists $N\geq 1$ such that for any $m\geq0$, we have $$z_0z_1\cdots z_j\prec x_{m}x_{m+1}\cdots x_{m+N-1}.$$ 
	$$(\text{resp. }z_{-j}\cdots z_{-1}z_0z_1\cdots z_j\prec x_{m}x_{m+1}\cdots x_{m+N-1}.)$$Then $(\overline{\operatorname{orb}(x,\sigma)},\sigma)$ has a unique minimal subsystem $(\overline{\operatorname{orb}(z,\sigma)},\sigma)$.
\end{Lem}
\begin{proof}
	Given $j\geq 0$ and $y\in\overline{\operatorname{orb}(x,\sigma)}$ (resp. $(\Sigma_m,\sigma)$), by assumptions, there exists there exists $N\geq 1$ such that for any $m\geq0$, we have $$z_0z_1\cdots z_j\prec x_{m}x_{m+1}\cdots x_{m+N-1}.$$ 
	$$(\text{resp. }z_{-j}\cdots z_{-1}z_0z_1\cdots z_j\prec x_{m}x_{m+1}\cdots x_{m+N-1}.)$$Since $y\in\overline{\operatorname{orb}(x,\sigma)}$, there exists $L\geq 0$ such that $$y_0y_1\cdots y_{N-1}=x_Lx_{L+1}\cdots x_{L+N-1}.$$ Hence, $$z_0z_1\cdots z_j\prec y_{0}y_{1}\cdots y_{N-1}.$$ 
	$$(\text{resp. }z_{-j}\cdots z_{-1}z_0z_1\cdots z_j\prec y_{0}y_{1}\cdots y_{N-1}.)$$ By the arbitrariness of $j$, one has $z\in\overline{\operatorname{orb}(y,\sigma)}$ and thus $\overline{\operatorname{orb}(z,\sigma)}\subset\overline{\operatorname{orb}(y,\sigma)}$. Therefore, $(\overline{\operatorname{orb}(x,\sigma)},\sigma)$ has a unique minimal subsystem $(\overline{\operatorname{orb}(z,\sigma)},\sigma)$.
\end{proof}

\begin{Lem}\label{Theorem-unique-minimal}
	Given an integer $m>1$, a almost periodic point $z$ of $(\Sigma_m^+,\sigma)$ (resp. $(\Sigma_m,\sigma)$) and a point $x$ of $(\Sigma_m^+,\sigma)$ (resp. $(\Sigma_m,\sigma)$) satisfying that for any $j\geq0$, there exists $N\geq 1$ such that for any $m\geq0$, we have $$z_0z_1\cdots z_j\prec x_{m}x_{m+1}\cdots x_{m+N-1}.$$ 
	$$(\text{resp. }z_{-j}\cdots z_{-1}z_0z_1\cdots z_j\prec x_{m}x_{m+1}\cdots x_{m+N-1}.)$$ Then there exists a recurrent point $y$ of $(\Sigma_m^+,\sigma)$ (resp. $(\Sigma_m,\sigma)$) such that 
	\begin{enumerate}[(1)]
		\item $\overline{\operatorname{orb}(x,\sigma)}\subsetneq\omega_\sigma(y)=C_y^*$;
		\item for any $j\geq0$, there exists $L\geq 1$ such that for any $m\geq0$, we have $$z_0z_1\cdots z_j\prec y_{m}y_{m+1}\cdots y_{m+L-1}.$$ 
		$$(\text{resp. }z_{-j}\cdots z_{-1}z_0z_1\cdots z_j\prec y_{m}y_{m+1}\cdots y_{m+L-1}.)$$ 
	\end{enumerate}
\end{Lem} 
\begin{proof}
	By Lemma \ref{Lemma-unique-minimal}, $(\overline{\operatorname{orb}(x,\sigma)},\sigma)$ has a unique minimal subsystem $(\overline{\operatorname{orb}(z,\sigma)},\sigma)$ and thus we can choose a word $C_0=a_0\cdots a_{m_0-1}$ for some $m_0\in\mathbb{N}$ such that $C_0$ do not appear in $x$. For any $k\in\mathbb{N}$, denote $$B_k=x_0x_1\cdots x_{k-1},$$ then $|B_k|=k$. Given a sequence $\{b_n\}_{n=1}^\infty$ of $\mathbb{N}$, which will be determined later, denote $$C_1=(C_0C_0B_1)^{b_1}C_0C_0, C_2=(C_1C_1B_2)^{b_2}C_1C_1\text{ and }C_{n+1}=(C_nC_nB_{n+1})^{b_{n+1}}C_nC_n$$ for any $n\in\mathbb{Z}^+$. Now, we choose $\{b_n\}_{n=1}^\infty$ such that $$\frac{|B_n|}{|C_{n-1}|}\leq\frac{1}{2^{n-1}} \text{ for any }n\in\mathbb{N}^+.$$ Given $m\in\mathbb{Z}^+,n\in\mathbb{N}$, denote $\gamma_{m,n}$ the frequency of the word $C_mC_m$ appearing in the word $C_{m+n}$, which is defined as $\gamma_{m,n}=\frac{p_{m,n}}{|C_{m+n}|}$, where $p_{m,n}$ is the number of times of the word $C_mC_m$ appearing in the word $C_{m+n}$. Then $\gamma_{m,1}\geq\frac{b_{m+1}+1}{|C_{m+1}|}>0$ and for any $m\geq 0$ and $n\geq2$, we have that 
	\begin{align*}
		\gamma_{m,n}&\geq\frac{2(b_{m+n}+1)|C_{m+n-1}|}{b_{m+n}(2|C_{m+n-1}|+|B_{m+n}|)+2|C_{m+n-1}|}\gamma_{m,n-1}\\&\geq\frac{1}{1+\frac{|B_{m+n}|}{2|C_{m+n-1}|}}\gamma_{m,n-1}\\&\geq\frac{1}{\prod_{i=1}^\infty(1+\frac{1}{2^i})}\gamma_{m,1}\\&>0.
	\end{align*}
	$$$$
	We define $$y=y_0y_1\cdots\text{ as } y_0y_1\cdots y_{|C_n|-1}=C_n\text{ for any }n\in\mathbb{N}^+$$ $$(\text{resp. } y=\cdots y_{-1}y_{0}y_1\cdots\text{ as } y_{-|C_n|}\cdots y_{-1}y_0y_1\cdots y_{|C_n|-1}=C_nC_n\text{ for any }n\in\mathbb{N}^+)$$ and the cylinder $$[C_nC_n]=\{w\in\Sigma_m^+:w_0w_1\cdots w_{2|C_n|-1}=C_nC_n\}.$$ $$(\text{resp. }[C_nC_n]=\{w\in\Sigma_m:w_{-|C_n|}\cdots w_{-1}w_0w_1\cdots w_{|C_n|-1}=C_nC_n\}.)$$Then $y$ is a recurrent point and $\overline{\operatorname{orb}(x,\sigma)}\subsetneq\omega_\sigma(y)$. For any $m\geq 0$, we have  $$\liminf_{n\to\infty}\frac{1}{|C_n|}\sum_{i=0}^{|C_n|-1}\delta_{\sigma^i(y)}([C_mC_m])\geq\frac{1}{\prod_{i=1}^\infty(1+\frac{1}{2^i})}\gamma_{m,1}>0.$$ As a result, we have that $y\in C_y^*$ and thus $\omega_\sigma(y)=C_y^*$. 
	
	Now, given $j\geq 0$, there exists $N\geq 1$ such that for any $m\geq0$, we have $$z_0z_1\cdots z_j\prec x_{m}x_{m+1}\cdots x_{m+N-1}.$$ 
	$$(\text{resp. }z_{-j}\cdots z_{-1}z_0z_1\cdots z_j\prec x_{m}x_{m+1}\cdots x_{m+N-1}.)$$ Denote $L=|C_N|$, from the definition of $y$, for any $m\geq 0$, the word $y_my_{m+1}\cdots y_{m+L-1}$ contains a word of $x$ with length $|B_N|=N$ and thus $$z_0z_1\cdots z_j\prec y_{m}y_{m+1}\cdots y_{m+L-1}.$$ 
	$$(\text{resp. }z_{-j}\cdots z_{-1}z_0z_1\cdots z_j\prec y_{m}y_{m+1}\cdots y_{m+L-1}.)$$ 
\end{proof}
\begin{Lem}\label{Lemma 7.15}
	Given an integer $m>1$ and $\alpha<\log m$, there exists a subsystem $(\Lambda,\sigma)$ of $(\Sigma_m^+,\sigma)$ (resp. $(\Sigma_m,\sigma)$) such that
	\begin{enumerate}[(1)]
		\item $(\Lambda,\sigma)$ has a unique minimal subsystem $(\Delta,\sigma)$ with $\Delta\subsetneq\Lambda$;
		\item $h_{top}(\sigma,\Lambda)\geq h_{top}(\sigma,\Delta)>\log m-\alpha$;
		\item there exist three invariant measures of $(\Lambda,\sigma)$, denoted by $\mu_1$, $\mu_2$, $\mu_3$, such that $$\Delta\subsetneq S_{\mu_1}\subsetneq S_{\mu_2}\subsetneq S_{\mu_3}=\Lambda\subsetneq\Sigma_m^+~(\text{resp. }\Sigma_m);$$
		\item $S_{\mu_1}$, $S_{\mu_2}$ and $S_{\mu_3}$ are all transitive;
		\item $\min\{h_{\mu_1},h_{\mu_2},h_{\mu_3}\}>\log m-\alpha$.
	\end{enumerate}
\end{Lem}
\begin{proof}
	Given $0<\alpha<\log m$, since $(\Sigma_m^+,\sigma)$ (resp. $(\Sigma_m,\sigma)$) satisfies the shadowing property, by Lemma \ref{Lemma 6.5}, there exists a minimal subsystem $(\Delta,\sigma)$ such that $h_{top}(\sigma,\Delta)>h_{top}(\sigma)-\alpha=\log m-\alpha$. Choose $x\in\Delta$, then by Lemma \ref{Lemma 6.13}, for any $j\geq0$, there exists $N\geq 1$ such that for any $m\geq0$, we have $$x_0x_1\cdots x_j\prec x_{m}x_{m+1}\cdots x_{m+N-1}.$$ 
	$$(\text{resp. }x_{-j}\cdots x_{-1}x_0x_1\cdots x_j\prec x_{m}x_{m+1}\cdots x_{m+N-1}.)$$ By Lemma \ref{Theorem-unique-minimal} and Lemma \ref{Lemma-unique-minimal}, we can find three recurrent points $y_1,y_2,y_3$ such that $$\Delta=\overline{\operatorname{orb}(x,\sigma)}\subsetneq\omega_\sigma(y_1)=C_{y_1}^*\subsetneq\omega_\sigma(y_2)=C_{y_2}^*\subsetneq\omega_\sigma(y_3)=C_{y_3}^*$$
	and $\Lambda:=C_{y_3}^*$ has a unique minimal subsystem $\Delta$. It is clear that  $\Lambda\subsetneq\Sigma_m^+~(\text{resp. }\Sigma_m)$. Since $y_1,y_2,y_3$ are all recurrent, we have that $C_{y_1}^*$, $C_{y_2}^*$ and $C_{y_3}^*$ are all transitive. By Lemma \ref{Lem-IrregularMultianalysis-11111}, we can find $\mu_1\in M(f,C_{y_1}^*)$, $\mu_2\in M(f,C_{y_2}^*)$ and $\mu_3\in M(f,C_{y_3}^*)$ such that $S_{\mu_1}=C_{y_1}^*$, $S_{\mu_2}=C_{y_2}^*$, $S_{\mu_3}=C_{y_3}^*$ and $$\min\{h_{\mu_1},h_{\mu_2},h_{\mu_3}\}>\log m-\alpha.$$
\end{proof}

\begin{Lem}\label{Lemma-iteration-unique-minimal}
	Suppose that $(X,f)$ is a dynamical system, $n\geq1$ and $(\Delta_1,f^n)$ has a unique minimal subsystem $(\Delta_2,f^n)$, then $(\bigcup_{i=0}^{n-1}f^i(\Delta_1),f)$ has a unique minimal subsystem $(\bigcup_{i=0}^{n-1}f^i(\Delta_2),f)$.
\end{Lem}
\begin{proof}
	Denote $\Lambda_1=\bigcup_{i=0}^{n-1}f^i(\Delta_1)$ and $\Lambda_2=\bigcup_{i=0}^{n-1}f^i(\Delta_2)$. Given $x\in\Lambda_1$, then there exists $0\leq s\leq k-1$ such that $f^s(x)\in\Delta_1$. Since $(\Delta_1,f^n)$ has a unique minimal subsystem $(\Delta_2,f^n)$, we have that $$\Delta_2\subset\overline{\operatorname{orb}(f^s(x),f^n)}.$$ As a result, $$\Lambda_2=\bigcup_{i=0}^{n-1}f^i(\Delta_2)\subset\bigcup_{i=0}^{n-1}f^i(\overline{\operatorname{orb}(f^s(x),f^n)})\subset\bigcup_{i=0}^{n-1}\overline{f^i(\operatorname{orb}(f^s(x),f^n))}\subset\overline{\bigcup_{i=0}^{n-1}f^i(\operatorname{orb}(f^s(x),f^n))}\subset\overline{\operatorname{orb}(x,f)}.$$ Therefore, $(\Lambda_1,f)$ has a unique minimal subsystem $(\Lambda_2,f)$.
\end{proof}

\begin{Lem}\label{Lemma 7.17}
	Suppose that $(X,f)$ is a dynamical system, $n\geq 1$, $\emptyset\neq\Delta\subset X$ is compact, $f^n$-invariant. We define a map $\mathfrak{F}:M(f^n,\Delta)\to\mathcal{M}(f,\Lambda)$ by $$\mathfrak{F}(\mu)=\frac{1}{n}\sum_{i=0}^{n-1}\mu\circ f^{-i},$$where $\Lambda=\bigcup_{i=0}^{n-1}f^i(\Delta)$. Then for any $\mu\in M(f^n,\Delta)$, we have
	\begin{enumerate}[(1)]
		\item $h_{\mathfrak{F}(\mu)}(f)=\frac{1}{n}h_{\mathfrak{F}(\mu)}(f^n)=\frac{1}{n}h_\mu(f^n)$;
		\item $S_{\mathfrak{F}(\mu)}=\bigcup_{i=0}^{n-1}f^i(S_\mu)$. In particular, if $S_\mu$ is $f^n$-transitive, then $S_{\mathfrak{F}(\mu)}$ is $f$-transitive.
	\end{enumerate}
\end{Lem}
\begin{proof}
	Item (1) is directly from \cite[Lemma 18.5]{DGS} and the fact that $h_\nu(f^n)=nh_\nu(f)$ for any $nu\in M(f,X)$. For item (2), given $\mu\in M(f^n,\Delta)$, on the one hand, since $S_\mu\subseteq S_{\mathfrak{F}(\mu)}$ and $S_{\mathfrak{F}(\mu)}$ is $f$-invariant, we have that $\bigcup_{i=0}^{n-1}f^i(S_\mu)\subseteq S_{\mathfrak{F}(\mu)}$. On the other hand, since $\mathfrak{F}(\mu)(\bigcup_{i=0}^{n-1}f^i(S_\mu))=1$ and $\bigcup_{i=0}^{n-1}f^i(S_\mu)$ is compact, we have that $S_{\mathfrak{F}(\mu)}\subseteq\bigcup_{i=0}^{n-1}f^i(S_\mu)$. As a result, $S_{\mathfrak{F}(\mu)}=\bigcup_{i=0}^{n-1}f^i(S_\mu)$. Now, suppose that $S_\mu$ is $f^n$-transitive, then there exists $x\in S_\mu$ with $\omega_{f^n}(x)=S_\mu$. Hence, for any $0\leq i\leq n-1$, $$f^i(S_\mu)=f^i(\omega_{f^n}(x))\subseteq\omega_{f^n}(f^ix)\subseteq\omega_f(x)\subseteq S_{\mathfrak{F}(\mu)}.$$
	As a result, $\omega_f(x)=S_{\mathfrak{F}(\mu)}$ and thus $S_{\mathfrak{F}(\mu)}$ is $f$-transitive.
\end{proof}
Next, we generalize Lemma \ref{Lemma 7.15} from full shifts to topologically transitive topologically expanding or topologically transitive topologically Anosov dynamical system.
\begin{mainlemma}\label{Lemma C}
	Suppose that $(X,f)$ is topologically transitive topologically expanding or topologically transitive topologically Anosov. Then for any $0<\alpha<h_{top}(f)$, there exists a subsystem $(\Lambda,f)$ such that 
	\begin{enumerate}[(1)]
		\item $(\Lambda,f)$ has a unique minimal subsystem $(\Delta,f)$ with $\Delta\subsetneq\Lambda$;
		\item $h_{top}(f,\Lambda)\geq h_{top}(f,\Delta)>h_{top}(f)-\alpha$;
		\item there exist three invariant measures of $(\Lambda,f)$, denoted by $\mu_1$, $\mu_2$, $\mu_3$, such that $$\Delta\subsetneq S_{\mu_1}\subsetneq S_{\mu_2}\subsetneq S_{\mu_3}=\Lambda\subsetneq X;$$
		\item $S_{\mu_1}$, $S_{\mu_2}$ and $S_{\mu_3}$ are all transitive;
		\item $\min\{h_{\mu_1},h_{\mu_2},h_{\mu_3}\}>h_{top}(f)-\alpha$.
	\end{enumerate}
\end{mainlemma}
\begin{proof}
	Suppose that $(X,  f)$ is topologically transitive topologically expanding (resp. topologically transitive topologically Anosov). Given $0<\alpha<h_{top}(f)$, then by Lemma \ref{Lemma 6.9}, there are $m,  k\in \mathbb{N}$,   $\frac{\log m}{k}>h_{top}(f)-\frac{\alpha}{2}$ and a closed set $\Lambda'\subset X$
	invariant under $f^k$ such that 
	\begin{enumerate}[(1)]
		\item $\bigcup_{i=0}^{k-1}f^i(\Lambda')\subsetneq X$ and $f^i(\Lambda')\cap f^j(\Lambda')$ for any $0\leq i<j\leq k-1$;
		\item there is a conjugate map $$\pi: (\Lambda',   f^k)\to (\Sigma_{m}^+,  \sigma)\text{ }(\text{resp. }(\Sigma_{m},  \sigma)).$$
	\end{enumerate}
Choose $\tilde{\alpha}<\log m$ such that $\frac{\log m-\tilde{\alpha}}{k}>h_{top}(f)-\alpha$. By Lemma \ref{Lemma 7.15}, there exists a subsystem $(\tilde{\Lambda},\sigma)$ of $(\Sigma_m^+,\sigma)$ (resp. $(\Sigma_m,\sigma)$) such that
\begin{enumerate}[(1)]
	\item $(\tilde{\Lambda},\sigma)$ has a unique minimal subsystem $(\tilde{\Delta},\sigma)$ with $\tilde{\Delta}\subsetneq\tilde{\Lambda}$;
	\item $h_{top}(\sigma,\tilde{\Lambda})\geq h_{top}(\sigma,\tilde{\Delta})>\log m-\tilde{\alpha}$;
	\item there exist three invariant measures of $(\tilde{\Lambda},\sigma)$, denoted by $\nu_1$, $\nu_2$, $\nu_3$, such that $$\tilde{\Delta}\subsetneq S_{\nu_1}\subsetneq S_{\nu_2}\subsetneq S_{\nu_3}=\tilde{\Lambda}\subsetneq\Sigma_m^+~(\text{resp. }\Sigma_m);$$
	\item $S_{\nu_1}$, $S_{\nu_2}$ and $S_{\nu_3}$ are all transitive;
	\item $\min\{h_{\nu_1},h_{\nu_2},h_{\nu_3}\}>\log m-\tilde{\alpha}$.
\end{enumerate}
As a result, there exists a subsystem $(\Lambda_1,f^k)$ of $(\Lambda',f^k)$ such that
\begin{enumerate}[(1)]
	\item $(\Lambda_1,f^k)$ has a unique minimal subsystem $(\Delta_1,f^k)$ with $\Delta_1\subsetneq\Lambda_1$;
	\item $h_{top}(f^k,\Lambda_1)\geq h_{top}(f^k,\Delta_1)>\log m-\tilde{\alpha}$;
	\item there exist three invariant measures of $(\Lambda_1,f^k)$, denoted by $\omega_1$, $\omega_2$, $\omega_3$, such that $$\Delta_1\subsetneq S_{\omega_1}\subsetneq S_{\omega_2}\subsetneq S_{\omega_3}=\Lambda_1\subsetneq\Lambda';$$
	\item $S_{\omega_1}$, $S_{\omega_2}$ and $S_{\omega_3}$ are all transitive;
	\item $\min\{h_{\omega_1}(f^k),h_{\omega_2}(f^k),h_{\omega_3}(f^k)\}>\log m-\tilde{\alpha}$.
\end{enumerate}
Denote $\Lambda=\bigcup_{i=0}^{k-1}f^i(\Lambda_1)$ and $\Delta=\bigcup_{i=0}^{k-1}f^i(\Delta_1)$. Then by Lemma \ref{Lemma-basic-property-Bowen-entropy-1}, $$h_{top}(f,\Lambda)\geq h_{top}(f,\Delta)\geq h_{top}(f,\Delta_1)=\frac{1}{k}h_{top}(f^k,\Delta_1)>\frac{\log m-\tilde{\alpha}}{k}>h_{top}(f)-\alpha$$ and  by Lemma \ref{Lemma-iteration-unique-minimal}, $(\Lambda,\sigma)$ has a unique minimal subsystem $(\Delta,\sigma)$ with $\Delta\subsetneq\Lambda$. Finally, we denote $\mu_1=\mathfrak{F}(\omega_{1})$, $\mu_2=\mathfrak{F}(\omega_{2})$  $\mu_3=\mathfrak{F}(\omega_{3})$, then by Lemma \ref{Lemma 7.17}, we have
\begin{enumerate}[(1)]
	\item $\Delta\subsetneq S_{\mu_1}\subsetneq S_{\mu_2}\subsetneq S_{\mu_3}=\Lambda\subsetneq X$;
	\item $S_{\mu_1}$, $S_{\mu_2}$ and $S_{\mu_3}$ are all transitive;
	\item $\min\{h_{\mu_1},h_{\mu_2},h_{\mu_3}\}>\frac{\log m-\tilde{\alpha}}{k}>h_{top}(f)-\alpha$.
\end{enumerate}
\end{proof}

Now, we give the proof of Theorem \ref{Theorem B}: \textbf{Case} $(0\alpha_2\alpha_3\alpha_4\alpha_50)$.

\emph{Proof of Theorem \ref{Theorem B}: \textbf{Case} $(0\alpha_2\alpha_3\alpha_4\alpha_50)$.}  Let $U\subseteq X$ be a nonempty open set. Given $0<\alpha<h_{top}(f)$, by Lemma \ref{Lemma C}, there exists a subsystem $(\Lambda,f)$ such that 
\begin{enumerate}[(1)]
	\item $(\Lambda,f)$ has a unique minimal subsystem $(\Delta,f)$ with $\Delta\subsetneq\Lambda$;
	\item $h_{top}(f,\Lambda)\geq h_{top}(f,\Delta)>h_{top}(f)-\alpha$;
	\item there exist three invariant measures of $(\Lambda,f)$, denoted by $\mu_1$, $\mu_2$, $\mu_3$, such that $$\Delta\subsetneq S_{\mu_1}\subsetneq S_{\mu_2}\subsetneq S_{\mu_3}=\Lambda\subsetneq X;$$
	\item $S_{\mu_1}$, $S_{\mu_2}$ and $S_{\mu_3}$ are all transitive;
	\item $\min\{h_{\mu_1},h_{\mu_2},h_{\mu_3}\}>h_{top}(f)-\alpha$.
\end{enumerate}
By Lemma \ref{maintheorem-abstract-4}, there exist internally chain transitive closed invariant subset $\Theta_0,\Theta_1,\Theta_2,\Theta_3\subsetneq X$ such that 
\begin{enumerate}[(1)]
	\item $\Delta\subsetneq \Theta_0$ and $M(f, \Theta_0)=M(f,  \Delta)$;
	\item for any $1\leq i\leq 3$, $S_{\mu_i}\subsetneq \Theta_i$ and $M(f, \Theta_i)=M(f,  S_{\mu_i})$.
\end{enumerate}
Since $(X,f)$ is topologically transitive, by Lemma \ref{lemma-B}, there are nonempty open sets $V_i$ in $X$ such that $V_i\subseteq U$ and $V_i\cap \Theta_i=\emptyset.$ for any $0\leq i\leq 3$. By the variational principle, we can choose $\mu_0\in M(f,\Delta)$ with $h_{\mu_0}>h_{top}(f)-\alpha$, then $S_{\mu_0}=\Delta$. Then we have
\begin{enumerate}[(1)]
	\item $\{x\in G_{\mu_0}:\omega_f(x)=\Delta\}\cap V_0\subset\mathscr{C}(011110)\cap U$ and $\{x\in G_{\mu_0}:\omega_f(x)=\Theta_0\}\cap V_0\subset\mathscr{C}(011100)\cap U$;
	\item $\{x\in G_{\mu_1}:\omega_f(x)=S_{\mu_1}\}\cap V_1\subset\mathscr{C}(001110)\cap U$ and $\{x\in G_{\mu_1}:\omega_f(x)=\Theta_1\}\cap V_1\subset\mathscr{C}(001100)\cap U$;
	\item $\{x\in G_{\mu_1}:\omega_f(x)=S_{\mu_2}\}\cap V_2\subset\mathscr{C}(001010)\cap U$ and $\{x\in G_{\mu_1}:\omega_f(x)=\Theta_2\}\cap V_2\subset\mathscr{C}(001000)\cap U$;
	\item $$\{x\in G_{\operatorname{cov}\{\mu_0,\mu_1\}}:\omega_f(x)=S_{\mu_1}\}\cap V_1\subset\mathscr{C}(010110)\cap U$$ and $$\{x\in G_{\operatorname{cov}\{\mu_0,\mu_1\}}:\omega_f(x)=\Theta_1\}\cap V_1\subset\mathscr{C}(010100)\cap U;$$
	\item $$\{x\in G_{\operatorname{cov}\{\mu_1,\mu_2\}}:\omega_f(x)=S_{\mu_2}\}\cap V_2\subset\mathscr{C}(000110)\cap U$$ and $$\{x\in G_{\operatorname{cov}\{\mu_1,\mu_2\}}:\omega_f(x)=\Theta_2\}\cap V_2\subset\mathscr{C}(000100)\cap U;$$
	\item $$\{x\in G_{\operatorname{cov}\{\mu_0,\mu_1\}}:\omega_f(x)=S_{\mu_2}\}\cap V_2\subset\mathscr{C}(010010)\cap U$$ and $$\{x\in G_{\operatorname{cov}\{\mu_0,\mu_1\}}:\omega_f(x)=\Theta_2\}\cap V_2\subset\mathscr{C}(010000)\cap U;$$
	\item $$\{x\in G_{\operatorname{cov}\{\mu_1,\mu_2\}}:\omega_f(x)=S_{\mu_3}\}\cap V_3\subset\mathscr{C}(000010)\cap U$$ and $$\{x\in G_{\operatorname{cov}\{\mu_1,\mu_2\}}:\omega_f(x)=\Theta_3\}\cap V_3\subset\mathscr{C}(000000)\cap U;$$
	\item $\{x\in G_{\mu_0}:\omega_f(x)=S_{\mu_1}\}\cap V_1\subset\mathscr{C}(011010)\cap U$ and $\{x\in G_{\mu_0}:\omega_f(x)=\Theta_1\}\cap V_1\subset\mathscr{C}(011000)\cap U$.
\end{enumerate}By Lemma \ref{lem-shadowing-s-limit-shadowing}, Lemma \ref{lem-shadowing-s-limit-shadowing-2} and Theorem \ref{G-K-star}, for any $0\alpha_2\alpha_3\alpha_4\alpha_50\in\mathfrak{A}$, we have that $$h_{top}(f,\mathscr{C}(0\alpha_2\alpha_3\alpha_4\alpha_50)\cap U)>h_{top}(f)-\alpha.$$ Finally, by the arbitrariness of $\alpha$, we finish the proof. \qed

\subsection{Proof of Theorem \ref{Theorem B}: \textbf{Case} $(011111)$}Suppose that $(X,  f)$ is topologically transitive topologically expanding (resp. topologically transitive topologically Anosov). Given a nonempty open set $U\subset X$ and  $0<\alpha<h_{top}(f)$, then by Lemma \ref{Lemma 6.9}, there are $m,  k\in \mathbb{N}$,   $\frac{\log m}{k}>h_{top}(f)-\alpha$ and a closed set $\Lambda\subset X$
invariant under $f^k$ such that 
\begin{enumerate}[(1)]
	\item $\Lambda\subseteq U$, $\bigcup_{i=0}^{k-1}f^i(\Lambda)\subsetneq X$ and $f^i(\Lambda)\cap f^j(\Lambda)$ for any $0\leq i<j\leq k-1$;
	\item there is a conjugate map $$\pi: (\Lambda,   f^k)\to (\Sigma_{m}^+,  \sigma)\text{ }(\text{resp. }(\Sigma_{m},  \sigma)).$$
\end{enumerate}
By Lemma \ref{Lemma 6.5}, one has $h_{top}(f^k, AP(f^k)\cap\Lambda)=h_{top}(f^k,\Lambda)=\log m$. As a result, 
\begin{align*}
	h_{top}(f, \mathscr{C}(011111)\cap U)&=h_{top}(f, AP(f)\cap U)=h_{top}(f, AP(f^k)\cap U)\\ &\geq h_{top}(f, AP(f^k)\cap \Lambda)=\frac{1}{k}h_{top}(f^k, AP(f^k)\cap\Lambda)\\ &=\frac{\log m}{k}>h_{top}(f)-\alpha.
\end{align*}
Finally, by the arbitrariness of $\alpha$, we finish the proof. \qed

\section{Lebesgue measure of \texorpdfstring{$\mathscr{C}(\alpha_1\alpha_2\cdots\alpha_6)$}{C(alpha1alpha2...alpha6)}: proof of Theorem \ref{Theorem D}}\label{section leb}

\subsection{ Rotation on $\mathbb{S}^1$}
Consider the circle  $\mathbb{S}^1= \mathbb{R} / \mathbb{Z} .$ Given $\alpha\in\mathbb{R}$ let $f_\alpha:\mathbb{S}^1\to\mathbb{S}^1$ denote the rotation of angle $\alpha$ given by $$f_\alpha(x)=x+\alpha ~(\mathrm{mod} ~1)$$ for every $x\in\mathbb{S}^1.$ 

\begin{Thm}\label{Thm-rotation}
	For every rotation $f_\alpha:\mathbb{S}^1\to\mathbb{S}^1,$  we have $\mathbb{S}^1\subseteq \mathscr{C}(011111)$  and thus $\mathscr{C}(011111)$ has full Lebesgue measure.
\end{Thm}
\begin{proof}
	When $\alpha\in\mathbb{R}\setminus\mathbb{Q},$ $f_\alpha$ is a irrational rotation.  It is well known that every irrational rotation is minimal and thus $x\in \mathscr{C}(011111)$ for  any   $x\in \mathbb{S}^1.$
	When $\alpha\in \mathbb{Q},$ every $x\in\mathbb{S}^1$ is periodic and thus  $x\in \mathscr{C}(011111)$.  
\end{proof}

\subsection{Morse-Smale diffeomorphism}
A diffeomorphism $f:M\to M$ is said to be a Morse-Smale diffeomorphism if the following conditions are satisfied:
\begin{enumerate}
	\item the non-wandering set $\Omega_f$ is finite and hence consists only of periodic points;
	\item all periodic points are hyperbolic;
	\item for any $p,q\in\Omega_f,$ the stable manifold of $p$ and the unstable manifold of $p$ are transversal.
\end{enumerate}
\begin{Thm}
	For every Morse-Smale diffeomorphism $f:M\to M,$  we have $M\setminus \Omega_f\subseteq  \mathscr{C}(011110)$ and thus $\mathscr{C}(011110)$ has full Lebesgue measure.
\end{Thm}
\begin{proof}
	The definition of Morse-Smale diffeomorphism implies that for any $x\in M$, $\omega_f(x)=\operatorname{orb}(p,f)$ for some $p\in \Omega_f.$ Note that if $x\in M\setminus \Omega_f$, then $\omega_{f}(x)\subsetneq\overline{\operatorname{orb}(x,f)},$ and thus $$\emptyset\subsetneq \omega_{\underline{B}}(x)= \omega_{\underline{d}}(x)= \omega_{\overline{d}}(x)= \omega_{\overline{B}}(x)= \omega_{f}(x)=\operatorname{orb}(p,f)\subsetneq \overline{\operatorname{orb}(x,f)}$$ for some $p\in \Omega_f,$ which implies $x \in \mathscr{C}(011110)$. 
\end{proof}

\subsection{Bowen's eye and smooth reparameterizations of irrational linear ﬂows of the two torus with two stopping points}
\begin{figure}[h]\caption{Bowen’s eye}\label{fig-2}
	\begin{center}

	\tikzset{every picture/.style={line width=0.75pt}} 
	
	\begin{tikzpicture}[x=0.75pt,y=0.75pt,yscale=-1,xscale=1]
		
		\draw    (171,202) .. controls (248,60) and (400,59) .. (493,199) ;
		\draw    (170,112) .. controls (275,281) and (378,287) .. (489,115) ;
		\draw    (318,156) .. controls (343,154) and (350,206) .. (307,189) ;
		\draw    (307,189) .. controls (258,163) and (296,107) .. (341,137) ;
		\draw    (341,137) .. controls (383,166) and (374,218) .. (324,216) ;
		\draw    (324,216) .. controls (197,215) and (217,64) .. (384,115) ;
		\draw   (296.59,125.92) .. controls (304.34,127.91) and (311.7,128.4) .. (318.67,127.39) .. controls (312.09,129.89) and (305.9,133.91) .. (300.1,139.41) ;
		\draw   (249.01,196.57) .. controls (249.11,189.55) and (247.87,183.5) .. (245.26,178.42) .. controls (249.23,182.52) and (254.56,185.64) .. (261.24,187.79) ;
		\draw   (452.51,175.19) .. controls (445.11,178.24) and (438.92,182.24) .. (433.95,187.23) .. controls (437.72,181.28) and (440.27,174.36) .. (441.63,166.48) ;
		\draw   (221.69,191.44) .. controls (219.67,183.7) and (216.54,177.02) .. (212.27,171.42) .. controls (217.65,175.96) and (224.16,179.43) .. (231.79,181.84) ;
		\draw   (226.91,121.74) .. controls (234.9,121.28) and (242.06,119.51) .. (248.39,116.42) .. controls (242.88,120.81) and (238.21,126.52) .. (234.35,133.53) ;
		\draw   (418.49,114.08) .. controls (421.65,121.44) and (425.75,127.57) .. (430.8,132.47) .. controls (424.8,128.79) and (417.84,126.33) .. (409.94,125.09) ;
		\draw   (171.88,185.89) .. controls (178.99,182.22) and (184.81,177.69) .. (189.34,172.3) .. controls (186.09,178.55) and (184.14,185.66) .. (183.47,193.63) ;
		\draw   (183.39,123.61) .. controls (185.37,131.36) and (188.47,138.05) .. (192.7,143.68) .. controls (187.35,139.11) and (180.85,135.61) .. (173.24,133.16) ;
		\draw   (475.45,185.48) .. controls (473.58,177.7) and (470.58,170.97) .. (466.42,165.28) .. controls (471.71,169.93) and (478.16,173.52) .. (485.74,176.07) ;
		\draw   (488.85,131.48) .. controls (481.17,133.73) and (474.59,137.06) .. (469.11,141.49) .. controls (473.49,135.97) and (476.77,129.36) .. (478.95,121.67) ;
		
		\draw (175,149.4) node [anchor=north west][inner sep=0.75pt]    {$A$};
		\draw (468,146.4) node [anchor=north west][inner sep=0.75pt]    {$B$};
		\draw (312,68.4) node [anchor=north west][inner sep=0.75pt]    {$\Sigma _{1}$};
		\draw (314,247.4) node [anchor=north west][inner sep=0.75pt]    {$\Sigma _{2}$};

	\end{tikzpicture}
	
	\end{center}
\end{figure}

Let $\Phi=\{\varphi^{t}\}_{t\in\mathbb{R}}$ be a continuous flow on a compact metric space $X.$ Then $\varphi^{t}$ is  the time-t map. We denote the set  of all the $\Phi$-invariant Borel probability measures  by $M(\Phi).$ 
For any point $x\in X$  and $T\geq 0$,   we define the empirical measure of $x$ as
\begin{equation*}
	\mathcal{E}_{T}(x):=\frac{1}{T}\int_{0}^{T}\delta_{\varphi_t(x)}dt.
\end{equation*}
We denote the set of limit points of $\{\mathcal{E}_{T}(x)\}_{T\geq 0}$ by $V_{\Phi}(x)$. For any $\mu\in M(\Phi),$ denote 
$G_{\mu}(\Phi)=\{x\in X:V_\Phi(x)=\{\mu\}\}.$
\begin{Lem}\cite[Proposition 2.2 and Theorem 2.4]{PS2022}\label{Lemma-pacific}
	Let $\Phi=\{\varphi^{t}\}_{t\in\mathbb{R}}$ be a continuous flow on a compact metric space $X.$ 
	\begin{enumerate}
		\item If $\mu\in M(\Phi),$ then $G_{\mu}(\varphi^{t})\subseteq G_{\mu}(\Phi)$ for all $t\in\mathbb{R}\setminus\{0\}.$
		\item For any $t\in\mathbb{R}\setminus\{0\}$ and  $x\in X,$ $\sharp V_{\Phi}(x)=1$ if and only if $\sharp V_{\varphi^t}(x)=1.$
	\end{enumerate}
\end{Lem}
Following the argument of \cite[Proposition 2.2]{PS2022}, we have 
\begin{Lem}\label{Lemma-pacific-2}
	Let $\Phi=\{\varphi^{t}\}_{t\in\mathbb{R}}$ be a continuous flow on a compact metric space $X.$  If $\mu\in M(\Phi)$, $t\in\mathbb{R}\setminus\{0\}$ and $x\in X$, then $\mu\in V_{\varphi^t}(x)$ implies  $\mu\in V_{\Phi}(x).$
\end{Lem}
Next, we prove the following lemma.
\begin{Lem}\label{Lemma-pacific-3}
	Let $\Phi=\{\varphi^{t}\}_{t\in\mathbb{R}}$ be a continuous flow on a compact metric space $X.$  If $p$ is a fixed point of $\Phi,$ then for any $t\in\mathbb{R}$ and $x\in X$, $\delta_p\in V_{\Phi}(x)$ implies $\delta_p\in V_{\varphi^t}(x).$
\end{Lem}
\begin{proof}
	It is enough to verify for $t = 1.$ Assume that there is $x\in X$ such that $\delta_p\in V_{\Phi}(x)$ and $\delta_p\not\in V_{\varphi^1}(x).$ Let $\psi(y)=d(y,p)$ for any $y\in X$. Denote $$\underline{\psi}(x)=\liminf_{n\to\infty}\frac{1}{n}\sum_{j=0}^{n-1}\psi(\varphi^j(x)).$$
	Then $\psi(p)=0<\underline{\psi}(x).$   Take $0<\varepsilon<\frac{\underline{\psi}(x)}{4}.$ Then there is $N\in\mathbb{N}$ such that for any $n\geq N$ one has $\frac{1}{n}\sum_{j=0}^{n-1}\psi(\varphi^j(x))>3\varepsilon.$  Denote $M_\psi=\sup_{x\in X}\psi(x).$ Then we have 
	$$\sharp\{1\leq j\leq n-1: \psi(\varphi^j(x))>2\varepsilon\}>n\frac{\varepsilon}{M-2\varepsilon}.$$
	Take $\delta>0$ such that $|\psi(y)-\psi(\varphi^{t}(y))|<\varepsilon$ for any $y\in X$ and $-\delta<t<\delta.$ 
     Then 
	$$\mathrm{Leb}(\{0\leq t\leq n: \psi(\varphi^t(x))>\varepsilon\})>2\delta n\frac{\varepsilon}{M-2\varepsilon}.$$
	So for any $T\geq N$ we have 
	\begin{equation*}
		\begin{split}
			\frac{1}{T}\int_{0}^{T}\psi(\varphi^t(x))dt\geq  \frac{1}{T}\int_{0}^{\lfloor T\rfloor}\psi(\varphi^t(x))dt
			&\geq \frac{1}{T}2\delta \lfloor T\rfloor\frac{\varepsilon^2}{M-2\varepsilon},
		\end{split}
	\end{equation*}
	which implies $\liminf\limits_{T\to\infty}\frac{1}{T}\int_{0}^{T}\psi(\varphi^t(x))dt>2\delta \frac{\varepsilon^2}{M-2\varepsilon}>0.$ So $\delta_p\not \in V_{\Phi}(x).$ 
\end{proof}

We now consider the well-known example of a planar flow with divergent time averages attributed to Bowen $($see \cite{Takens1995}$)$, which is often referred to as Bowen’s eye. In the example, there are two fixed
hyperbolic saddle points A, B, and heteroclinic orbits connecting the two saddle points. We denote the expanding and contracting eigenvalues of the linearized vector field in $A$ by $\alpha_{+}$ and $-\alpha_{-}$ and in $B$ by $\beta_{+}$ and $-\beta_{-}$. The condition on the eigenvalues which makes the cycle attracting is that the contracting eigenvalues dominate: $\alpha_{-} \beta_{-}>$ $\alpha_{+} \beta_{+}.$
Denote
$$
\lambda=\alpha_{-} / \beta_{+} \text { and } \sigma=\beta_{-} / \alpha_{+},
$$
then their values are positive and their product is bigger than 1. Denote  the flow by $\Phi=\{\phi^{t}\}_{t\in\mathbb{R}}$. 
From \cite{Takens1995}, for any point $x$ inside the eye, $$V_\Phi(x)=\{t\mu_{1}+(1-t)\mu_{2}: t\in[0,1]\}$$ where $\mu_1=\frac{\sigma}{1+\sigma} \delta_A+\frac{1}{1+\sigma} \delta_B$ and $\mu_{2}=\frac{\lambda}{1+\lambda} \delta_B+\frac{1}{1+\lambda} \delta_A.$ Since $\lambda\sigma>1,$ then $\mu_{1}\neq\mu_{2}.$ Since $M(\varphi^1,X)=M(\Phi)=\{t\delta_A+(1-t)\delta_B: t\in[0,1]\},$
then by Lemma \ref{Lemma-pacific}(2) we have $\sharp V_{\varphi^1}(x)>1$, by Lemma \ref{Lemma-pacific-2} we have $V_{\varphi^1}(x)\subseteq V_\Phi(x).$ So $$V_{\varphi^1}(x)=\{t\mu_{1}+(1-t)\mu_{2}: t\in[t_1,t_2]\}$$ for some $0\leq t_1<t_2\leq 1.$
Thus for any point $x$ inside the eye we have $$\emptyset= \omega_{\underline{B}}(x)=\{A\}\cap \{B\}\subsetneq  \omega_{\underline{d}}(x)= \omega_{\overline{d}}(x)= \omega_{\overline{B}}(x)=\{A\}\cup\{B\} \subsetneq \omega_{\varphi^1}(x)\subset\Sigma_1\cup\Sigma_{2}\subsetneq  \overline{\operatorname{orb}(x,\varphi^1)},$$ and thus $x\in \mathscr{C}(101100).$
\begin{Thm}
	For Bowen's eye,   $\mathrm{Leb}(\mathscr{C}(101100))>0.$
\end{Thm}
In \cite{AG2022}, Andersson and Guihéneuf studies the behavior of the Birkhoff averages of continuous functions along the orbits of a smooth reparametrization of an irrational
flow on the 2-torus $\mathbb{T}^{2}$ with two singularities of quadratic order. More precisely, the reparameterized flow $\{\phi^{t}\}_{t\in\mathbb{R}}$ is generated by a vector field
$X=\phi X_0$, where $X_0=(1,\alpha)$ is the infinitesimal generator of the irrational flow on $\mathbb{T}^{2}$ of slope  $\alpha\in\mathbb{R}\setminus\mathbb{Q}$ and $\phi:\mathbb{T}^2\to[0,+\infty)$ is a $C^3$ function which
vanishes at two distinct points $p,q\in\mathbb{T}^2$ and is positive elsewhere so that the Hessians of $\phi$ at $p$ and $q$ are positive definite.
Denote 
$$\mu_\infty=\frac{\sqrt{d_q}}{\sqrt{d_p}+\sqrt{d_q}}\delta_p+\frac{\sqrt{d_p}}{\sqrt{d_p}+\sqrt{d_q}}\delta_q$$
where $d_p$ and $d_q$ are the determinants of the Hessians and $\delta_\mathrm{p}$ and $\delta_q$ are the Dirac point measures at $p$ and $q,$ respectively. From \cite{AG2022} there are subsets $\mathcal{D}, \mathcal{R}  \subset \mathbb{R} \times \mathbb{T} ^2\times \mathbb{T} ^2$ and $\mathcal{A}\subset \mathbb{R}$ with the following characteristics:  $\mathcal{A}$ is of full Lebesgue measure, $\mathcal{R}$ is a dense $G_{\delta }$  set, and $\mathcal{D}$ is dense (but not $G_{\delta }$) ,  such
that for any $(\alpha, p, q) \in \mathbb{T} \times \mathbb{T} ^2\times \mathbb{T} ^2$ and any $\Phi$   be a reparameterized linear flow satisfying
$(SH)$   with angle $\alpha$  and stopping points at $p$ and $q.$  Then $\Phi$ has

$\bullet$ $V_\Phi(x)=\{\mu_\infty\}$ for $\mathrm{Leb}$-a.e. $x$  if $(\alpha, p, q) \in \mathcal{D}$,

$\bullet$ $V_\Phi(x)=\{t\delta_p+(1-t)\mu_\infty: t\in[0,1]\}$ for $\mathrm{Leb}$-a.e. $x$ if $\alpha \in \mathcal{A}$ and $q = p + r(1, \alpha) \mod \mathbb{Z}^2$ for some $r > 0,$

$\bullet$ $V_\Phi(x)=\{t\delta_p+(1-t)\delta_q: t\in[0,1]\}$ for $\mathrm{Leb}$-a.e. $x$ if $(\alpha, p, q) \in \mathcal{R} .$

Similar as Bowen's eye,  by Lemma \ref{Lemma-pacific},  \ref{Lemma-pacific-2} and \ref{Lemma-pacific-3} we have

$\bullet$ $V_{\varphi^1}(x)=\{\mu_\infty\}$ for $\mathrm{Leb}$-a.e. $x$  if $(\alpha, p, q) \in \mathcal{D}$,

$\bullet$ for $\mathrm{Leb}-a.e.$ $x$ there is $0\leq t^*< 1$ such that $V_{\varphi^1}(x)=\{t\delta_p+(1-t)\mu_\infty: t\in[t^*,1]\}$  if $\alpha \in \mathcal{A}$ and $q = p + r(1, \alpha) \mod \mathbb{Z}^2$ for some $r > 0,$

$\bullet$ $V_{\varphi^1}(x)=\{t\delta_p+(1-t)\delta_q: t\in[0,1]\}$ for $\mathrm{Leb}$-a.e. $x$ if $(\alpha, p, q) \in \mathcal{R}.$

Then we have the following result.
\begin{Thm}
	Let $\Phi$   be a reparameterized linear flow satisfying
	$(SH)$   with angle $\alpha$  and stopping points at $p$ and $q.$
	\begin{enumerate}
		\item $x\in \mathscr{C}(101100)$ for $\mathrm{Leb}$-a.e. $x$  if $(\alpha, p, q) \in \mathcal{D}$,
		\item $x\in \mathscr{C}(100100)$ for $\mathrm{Leb}$-a.e. $x$ if $\alpha \in \mathcal{A}$ and $q = p + r(1, \alpha) \mod \mathbb{Z}^2$ for some $r > 0,$
		\item $x\in \mathscr{C}(110100)$ for $\mathrm{Leb}$-a.e. $x$ if $(\alpha, p, q) \in \mathcal{R} .$
	\end{enumerate}   
\end{Thm}

\subsection{Figure-eight attractor}
\begin{figure}[h]\caption{The figure-eight attractor}\label{fig-1}
	\begin{center}

		\tikzset{every picture/.style={line width=0.75pt}} 
		
		\begin{tikzpicture}[x=0.75pt,y=0.75pt,yscale=-1,xscale=1]
			
			\draw   (263.1,191.97) .. controls (242.77,215.07) and (207.78,217.5) .. (184.96,197.41) .. controls (162.13,177.31) and (160.12,142.3) .. (180.45,119.2) .. controls (200.79,96.11) and (235.77,93.67) .. (258.6,113.76) .. controls (286.14,138.02) and (299.92,150.15) .. (299.92,150.15) .. controls (299.92,150.15) and (287.65,164.09) .. (263.1,191.97) -- cycle ;
			\draw   (340.07,111.51) .. controls (362.24,90.17) and (397.31,90.64) .. (418.4,112.55) .. controls (439.48,134.46) and (438.6,169.52) .. (416.43,190.86) .. controls (394.25,212.19) and (359.18,211.73) .. (338.1,189.82) .. controls (312.65,163.38) and (299.92,150.15) .. (299.92,150.15) .. controls (299.92,150.15) and (313.31,137.26) .. (340.07,111.51) -- cycle ;
			\draw    (205,161) .. controls (190,147) and (215,121) .. (224,149) ;
			\draw    (205,161) .. controls (218,179) and (260,163) .. (243,130) ;
			\draw    (176,154) .. controls (171,121) and (211,92) .. (243,130) ;
			\draw    (265,127) .. controls (297,205) and (182,216) .. (176,154) ;
			\draw   (224.79,119.08) .. controls (221.39,115.42) and (217.3,112.88) .. (212.5,111.46) .. controls (217.99,111.75) and (224.15,110.93) .. (231.01,108.98) ;
			\draw   (267.33,154.54) .. controls (269.25,149.93) and (269.89,145.15) .. (269.19,140.2) .. controls (271.19,145.32) and (274.49,150.59) .. (279.1,156.03) ;
			\draw    (391.81,144.19) .. controls (398.35,163.64) and (363.99,174.62) .. (369.39,145.71) ;
			\draw    (391.81,144.19) .. controls (388.93,122.17) and (344.38,116.28) .. (343.64,153.39) ;
			\draw    (413.99,164.14) .. controls (402.7,195.54) and (353.73,202.04) .. (343.64,153.39) ;
			\draw    (332.32,157.51) .. controls (341.26,73.67) and (438.2,106.74) .. (413.99,164.14) ;
			\draw   (358.21,173.99) .. controls (359.45,178.83) and (361.81,183.02) .. (365.34,186.56) .. controls (360.67,183.68) and (354.86,181.45) .. (347.91,179.87) ;
			\draw   (358,116) .. controls (353.37,116.48) and (349.19,118.12) .. (345.5,120.91) .. controls (348.72,116.96) and (351.48,111.84) .. (353.76,105.58) ;
			\draw   (229.68,205.42) .. controls (235.99,206.48) and (241.77,206.29) .. (247.03,204.8) .. controls (242.3,207.55) and (238.1,211.58) .. (234.41,216.87) ;
			\draw   (372.23,200.08) .. controls (377.75,203.32) and (383.23,205.19) .. (388.67,205.67) .. controls (383.27,206.56) and (377.91,208.83) .. (372.58,212.47) ;
			\draw    (251,81) .. controls (406,56) and (460,99) .. (442,179) ;
			\draw    (442,179) .. controls (392,269) and (145,274) .. (146,153) ;
			\draw    (146,153) .. controls (144.2,134.6) and (150.67,108.74) .. (172.27,87.27) .. controls (198.79,60.92) and (248.13,41.18) .. (333,50) ;
			\draw   (206,73) .. controls (199.52,71.55) and (193.66,71.54) .. (188.38,72.96) .. controls (193.05,70.1) and (197.1,65.8) .. (200.55,60.06) ;
			\draw   (295.93,236.56) .. controls (301.22,240.58) and (306.55,243.02) .. (311.94,243.91) .. controls (306.51,244.57) and (301.04,246.8) .. (295.52,250.59) ;
			
			\draw (294,156.4) node [anchor=north west][inner sep=0.75pt]    {$p$};
			\draw (330,195.4) node [anchor=north west][inner sep=0.75pt]    {$L_{2}$};
			\draw (263,96.4) node [anchor=north west][inner sep=0.75pt]    {$L_{1}$};

		\end{tikzpicture}
		
	\end{center}
\end{figure}

Now we consider the 'figure-eight attractor': Figure \ref{fig-1} shows a flow with a stationary saddle point $p$ whose stable and unstable manifolds coincide. Let $f$ be the time-1 map, and assume that $|\det Df(p)| < 1.$ Then every point in the basin of this attractor is generic with respect to its unique invariant measure $\delta_p$.
Then for any $x\neq p,$ one has $V_f(x)=\{\delta_p\},$ and $\omega_{f}(x)=L_1, L_2\text{ or } L_1\cup L_2,$ and $\omega_{f}(x)\subsetneq \overline{\operatorname{orb}(x,f)}.$ This implies that for any $x\neq p$  $$\emptyset\subsetneq  \omega_{\underline{B}}(x)=  \omega_{\underline{d}}(x)= \omega_{\overline{d}}(x)= \omega_{\overline{B}}(x)=\{p\}\subsetneq \omega_{f}(x)\subsetneq  \overline{\operatorname{orb}(x,f)},$$ and thus $x\in \mathscr{C}(011100).$
\begin{Thm}
	In the figure-eight attractor,   $x\in \mathscr{C}(011100)$ for any $x\neq p$, and thus $\mathscr{C}(011100)$ has full  Lebesgue measure.
\end{Thm}

\subsection{Axiom A system}
Let $f: M \to M$ be a diffeomorphism on a compact Riemannian manifold $M.$  $f$ is said to  satisfies \emph{Axiom A} if the non-wandering set $\Omega(f)$ is hyperbolic and $\Omega(f)=\overline{\{x\in M: x \text { is periodic}\}}$. From Smale's spectral decomposition theorem (see, for example, \cite[Theorem 23.14]{DGS}), $\Omega$ is a finite union of disjoint closed $f$-invariant sets $\Omega_j$ $(1 \leq  j \leq  s)$ such that $f: \Omega_j\to \Omega_j$ is topologically transitive. Such sets $\Omega_j$ are called \emph{basic sets}. $f$ is said to be \emph{Anosov} if $M$ is hyperbolic. Every Anosov diffeomorphism  satisfies Axiom A.
  A compact $f$-invariant set $\Lambda\subset M$ is called an \emph{attractor} if there is a neighborhood $U$ of $\Lambda$ called its basin such that $f^nx\to \Lambda$ for every $x\in U.$   
From \cite{Bowen2}, if $\Omega_j$ is an attractor and $f$ is a $C^2$ diffeomorphism, there is a unique $f$-invariant Borel probability measure $\mu$ on $\Omega_j$  such that $S_\mu=\Omega_j$ and  for $\mathrm{Leb}$-a.e. $x\in U$ one has $V_f(x)=\{\mu\}$ where $U$ is the basin of $\Omega_j.$ 
 It's well known that periodic points are dense in $\Omega_j$. So if $\Omega_j$ is not consist of a single periodic orbit,  then there is more than one periodic orbit and thus $\omega_{\underline{B}}(x)=\emptyset$ for any $x$ with $\omega_{f}(x)=\Omega_j.$
Then  for $\mathrm{Leb}$-a.e. $x\in \Omega_j$ one has 
$$\emptyset= \omega_{\underline{B}}(x)\subsetneq  \Omega_j=S_\mu=  \omega_{\underline{d}}(x)= \omega_{\overline{d}}(x)= \omega_{\overline{B}}(x)= \omega_{f}(x)=  \overline{\operatorname{orb}(x,f)},$$ and thus $x\in \mathscr{C}(101111),$
for $\mathrm{Leb}-a.e.$ $x\in U\setminus\Omega_j$ one has 
$$\emptyset= \omega_{\underline{B}}(x)\subsetneq  \Omega_j=S_\mu=  \omega_{\underline{d}}(x)= \omega_{\overline{d}}(x)= \omega_{\overline{B}}(x)= \omega_{f}(x)\subsetneq  \overline{\operatorname{orb}(x,f)},$$ and thus $x\in \mathscr{C}(101110).$

\begin{Thm}\label{Thm-leb-axiom}
	Let $f: M \to M$ be a $C^2$ Axiom A diffeomorphism and $\Omega_j$ be a basic set which is not consist of a single periodic orbit. Assume that  $\Omega_j$ is  an attractor with basin $U$.
	\begin{enumerate}
		\item  If $\mathrm{Leb}(\Omega_j)>0,$ then  $\mathrm{Leb}(\mathscr{C}(101111))>0.$
		\item If $\mathrm{Leb}(U\setminus\Omega_j)>0,$ then  $\mathrm{Leb}(\mathscr{C}(101110))>0.$
	\end{enumerate}  In particular, if $f$ is a transitive Anosov diffeomorphism, then $M=U=\Omega=\Omega_j$ for any $1\leq j\leq s$ and thus $x\in \mathscr{C}(101111)$ for $\mathrm{Leb}$-a.e. $x\in M.$ 
\end{Thm}
\begin{Rem}
	Theorem \ref{Thm-leb-axiom}(2) can be applied to the solenoid  introduced by Smale in \cite{Smale1967}.
	Consider the solid torus $\mathcal{T}=S^1\times D^2$, where $S^1=[0,1] \mod 1$ and $D^2=\{(x,y)\in\mathbb{R}^2: x^2+y^2\leq1\}.$ Fix $\lambda\in(0,1/2)$, and define $F: \mathcal{T}\to\mathcal{T}$ by
	$$F(\phi,x,y)=\begin{pmatrix}2\phi,\lambda x+\frac12\cos2\pi\phi,\lambda y+\frac12\sin2\pi\phi\end{pmatrix}.$$
	The set $S=\cap_{n=0}^\infty F^n(\mathcal{T})$ is called a solenoid. It's known that $S$ is the non-wandering set of $F: \mathcal{T}\to\mathcal{T}$, $F: S\to S$ is hyperbolic,  topologically mixing and locally maximal,  and  $S$ is an attractor with basin $\mathcal{T},$  see, for example, \cite{BS2004}. This implie that $F: \mathcal{T}\to\mathcal{T}$ is a  Axiom A diffeomorphism, $S$ is the unique basic set and an attractor with basin $\mathcal{T}$. It's easy to see that $\Omega_j$  is not consist of a single periodic orbit and  $\mathcal{T}\setminus S\neq\emptyset$. So $\mathrm{Leb}(\mathcal{T}\setminus S)>0,$ and thus $\mathrm{Leb}(\mathscr{C}(101110))>0.$
\end{Rem}
\subsection{Almost Anosov system}
Almost Anosov system was introduced by Hu and Young in \cite{Hu-Young-1995}.
Let $f: M \to M$ be a $C^2$ diffeomorphism on a two-dimensional compact Riemannian manifold $M.$
$f$ is said to be almost Anosov if  $f$ satisfes the following two conditions:
\begin{enumerate}
	\item Assumption I.
	\begin{enumerate}
		\item $f$ has a fixed point $p$, i.e. $f(p)=p.$
		\item There exist a constant $\kappa^s<1$, a continuous function $\kappa^\mathrm{u}$ with
		$$\kappa^u(x)\begin{cases}=1,&\text{ at }x=p,\\>1,&\text{ elsewhere,}\end{cases}$$
		and a decomposition of the tangent space $T_xM$ at every $x\in M$ into
		$$T_xM=E_x^u\oplus E_x^s$$
		such that
		$$|Df_xv|\leq\kappa^s|v|,\quad\forall v\in E_x^s,$$
		$$|Df_xv|\geq\kappa^u(x)|v|,\quad\forall v\in E_x^u,$$
		and
		$$|Df_pv|=|v|,\quad\forall v\in E_p^u.$$
	\end{enumerate}
	\item Assumption II. $f$ is topologically transitive on $M.$
\end{enumerate}
\begin{Thm}
	Let $f: M \to M$ ba an almost Anosov system. Then $x\in \mathscr{C}(101011)$ for $\mathrm{Leb}-a.e.$ $x\in M.$ 
\end{Thm}
\begin{proof}
	We following the argument of \cite[Theorem B]{Hu-Young-1995}. For any $n\in \mathbb{N},$ let $U_n$ be a rectangle of the type in \cite[Lemma 5.1]{Hu-Young-1995} with $\mathrm{diam}(U)<\frac{1}{n}$. 	Let $g_n:M\setminus U_n\to M\setminus U_n$  be the first return map. That is,  for $x\in M\setminus U_n$, if $\tau(x)$ is the smallest positive integer with $f^{\tau(x)}x\in M\setminus P$, then $g_n(x)=f^{\tau(x)}(x).$
	Let  $\mu_n$ be the $g_n$-invariant measure as in \cite[Lemma 5.2]{Hu-Young-1995}. Then from the proof of   \cite[Lemma 5.2]{Hu-Young-1995},    conditional measures of $\mu_n$ on the unstable manifolds is equivalent to Lebesgue measure, and thus $M\setminus U_n\subseteq S_{\mu_n}$. By \cite[Lemma 5.2]{Hu-Young-1995},  $\mu_n$ is $g_n$-ergodic. Apply the
	Birkhoff ergodic theorem to $(g_n,\mu_n),$ one has $\omega_{g_n}(x)\supseteq S_{\mu_n}\supseteq M\setminus U_n$ for $\mu_n$-a.e. $x\in M\setminus U_n.$ This implies  $\omega_f(x)\supseteq  M\setminus U_n$ for $\mathrm{Leb}$-a.e. $x\in M\setminus U_n.$ From \cite[Lemma 5.1]{Hu-Young-1995}, $U_n$ can be  chosen such that $U_{n+1}\subseteq U_n$ for any $n\in\mathbb{N}.$ Then for any $m,n\in\mathbb{N}$ one has $\omega_f(x)\supseteq  M\setminus U_{n+m}\supseteq  M\setminus U_{n}$ for $\mathrm{Leb}$-a.e. $x\in M\setminus U_{n+m}.$ 
	This implies that for any $n\in\mathbb{N}$ one has $\omega_f(x)\supseteq  M\setminus U_{n}$ for $\mathrm{Leb}$-a.e. $x\in M$ and thus   $\omega_f(x)=M$ for $\mathrm{Leb}$-a.e. $x\in M.$  
	From \cite[Theorem B]{Hu-Young-1995}, $V_f(x)=\{\delta_p\}$ for $\mathrm{Leb}$-a.e. $x\in M.$  Hu and Young stated that the almost Anosov system can be made to be topologically conjugate to the original toral hyperbolic automorphism, and thus  $ \omega_{\underline{B}}(x)=\emptyset,$ $\omega_{\overline{B}}(x)=M$ if $\omega_f(x)=M$. So  for $\mathrm{Leb}$-a.e. $x\in M$ we have
	$$\emptyset= \omega_{\underline{B}}(x)\subsetneq  \{p\}=  \omega_{\underline{d}}(x)= \omega_{\overline{d}}(x)\subsetneq M=\omega_{\overline{B}}(x)= \omega_{f}(x)=  \overline{\operatorname{orb}(x,f)},$$  which implies $x\in\mathscr{C}(101011).$
\end{proof}

From \cite[Theorem B]{Hu-Young-1995}, given an almost Anosov system, $V_f(x)=\{\delta_p\}$ for $\mathrm{Leb}$-a.e. $x\in M.$ We pose the following question. If the following question has a positive answer, then  there will be  more cases can be observable  in the sense of  Lebesgue measure.
\begin{mainquestion}
	Given a toral hyperbolic automorphism $g:\mathbb{T}^2\to \mathbb{T}^2$ and a nonempty compact connected set $K \subseteq M (g,  \mathbb{T}^2),$ is there a dynamical system $f:\mathbb{T}^2\to \mathbb{T}^2$ which is topologically conjugate to $g:\mathbb{T}^2\to \mathbb{T}^2$ such that $V_f(x)=K$ for $\mathrm{Leb}$-a.e. $x\in M$?
\end{mainquestion}

\subsection{Volume-preserving diffeomorphisms}
Denote $\mathrm{Diff}_{\mathrm{vol}}^1(M)$ the space of $C^1$ volume-preserving diffeomorphisms of a compact manifold $M.$ By Theorem \ref{thm-density-basic-property}(5) we have the following results.

\begin{Thm}\label{Thm-volume}
	Assume that $f\in \mathrm{Diff}_{\mathrm{vol}}^1(M)$. Then for $\mathrm{Leb}$-a.e. $x\in M$ there is $$\alpha_1\alpha_21111\in\{001111, 101111, 011111 \}$$ such that $x\in \mathscr{C}(\alpha_1\alpha_21111).$ If further  volume is $f$-ergodic, then there is $$\alpha_1\alpha_21111\in\{001111, 101111, 011111 \}$$ such that $x\in \mathscr{C}(\alpha_1\alpha_21111)$ for $\mathrm{Leb}$-a.e. $x\in M.$
\end{Thm}
\begin{Rem}
	(1) It's well known that every rotation on $\mathbb{S}^1$ is volume-preserving. By Theorem \ref{Thm-rotation} one has $x\in \mathscr{C}(011111)$ for $\mathrm{Leb}$-a.e. $x\in M.$
	
	(2) Every hyperbolic toral automorphism is ergodic with respect to Lebesgue measure.  By Theorem \ref{Thm-leb-axiom}, one has $x\in \mathscr{C}(101111)$ for $\mathrm{Leb}$-a.e. $x\in M.$
	
	(3) By \cite[Theorem B]{ACW}, $C^1$ generically, a volume-preserving diffeomorphism $f$ of a compact connected manifold $M$ with positive metric entropy has ergodic hyperbolic Lebesgue measure and admits a non-uniformly hyperbolic dominated splitting .  Using \cite[Theorem 3.17]{ST2015}, there is a horseshoe and thus $\cap_{\mu\in M(f, M)}S_{\mu}=\emptyset.$ So by Theorem \ref{Thm-volume} one has $x\in \mathscr{C}(101111)$ for $\mathrm{Leb}$-a.e. $x\in M.$
\end{Rem}

\subsection{Systems with physical measure}
Let $f: M \to M$ be a diffeomorphism on a compact Riemannian manifold $M.$ An invariant probability measure $\mu$ is  called \emph{physical} if $\mathrm{Leb}(G_\mu)>0.$
\begin{Thm}
	Let $f: M \to M$ be a diffeomorphism on a compact Riemannian manifold $M$. If there is a physical measure,   then $\mathrm{Leb}(\cup_{\alpha_1\alpha_21\alpha_4\alpha_5\alpha_6\in\mathfrak{A}} \mathscr{C}(\alpha_1\alpha_21\alpha_4\alpha_5\alpha_6))>0.$ 
\end{Thm}

\section{Applications: proofs of Corollary \ref{Corollary A}, Corollary \ref{Corollary B} and Corollary \ref{Corollary C}}\label{section 8}
In this section, we discuss the applications of our main results to $\beta$-shifts, $C^{1+\alpha}$ surface diffeomorphisms, Mañé diffeomorphisms and prove Corollary \ref{Corollary A}, Corollary \ref{Corollary B} and Corollary \ref{Corollary C}.
\subsection{$\beta$-shifts} Let us recall the definition of $\beta$-shift $(\beta>1)$ in \cite[Chapter 7.3]{Walters}. If $\beta \geq 2$ is an integer, $\beta$-shift is the full shift of $\beta$ symbols. So we only need to recall the definition in the case that $\beta$ is not an integer. Consider the expansion of 1 in powers of $\beta^{-1}$, i.e. $1=\sum_{n=1}^{\infty} a_{n} \beta^{-n}$ where $a_{1}=[\beta]$ and $a_{n}=\left[\beta^{n}-\sum_{i=1}^{n-1} a_{i} \beta^{n-i}\right] .$ Here $[t]$ denotes the integral part of $t \in \mathbb{R} .$ Let $k=[\beta]+1$. Then $0 \leq a_{n} \leq k-1$ for all $n$ so we can consider $a=\left\{a_{n}\right\}_{1}^{\infty}$ as a point in the space $X=\prod_{n=1}^{+\infty} Y$ where $Y=\{0,1, \cdots, k-1\} .$ Consider the lexicographical ordering on $X$, i.e. $x=\left\{x_{n}\right\}_{1}^{\infty}<y=\left\{y_{n}\right\}_{1}^{\infty}$ if $x_{j}<y_{j}$ for the smallest $j$ with $x_{j} \neq y_{j}$. Let $\sigma: X \rightarrow X$ denote the one-sided shift transformation. Note that $\sigma^{n} a \leq a$ for all $n \geq 0$. Let
$$X_{\beta}:=\left\{x=\left\{x_{n}\right\}_{1}^{\infty}: x \in X \text{ and }\sigma^{n}x \leq a \text{ for all } n \geq 0\right\}.$$
Then $X_{\beta}$ is a closed subset of $X$ and $\sigma\left(X_{\beta}\right)=X_{\beta}$. Then $(X_{\beta}, \sigma)$ is called a $\beta$-shift. The topological entropy of a $\beta$-shift $(\beta>1)$ is $\log \beta$  \cite[Chapter 7.3]{Walters}. By the definition of $X_\beta$ above, $X_{\beta_1}\subsetneq X_{\beta_2}$ for $\beta_1<\beta_2$ \cite{P}. And every $\beta$-shift is topologically transitive.

Now, we give the proof of Corollary \ref{Corollary A}.

\emph{Proof of Corollary \ref{Corollary A}.} Suppose that $(X,f)=(X_\beta,\sigma)$ is a $\beta$-shift ($\beta>1$). From \cite{P}, see also \cite[Proposition 3.2]{Sm}, we have $\{\beta\in (1,+\infty):(X_{\beta},\sigma)\text{ is a subshift of finite type} \}$ is dense in $(1,+\infty)$. Recall that every subshift of finite type satisfies the shadowing property. Thus for any $\beta\in(1,+\infty)$, there exists an increase sequence $\{\beta_{i}\}_{i=1}^{\infty}$ such that $\lim\limits_{i\to\infty}\beta_{i}=\beta$ and $(X_{\beta_{i}},\sigma)$ satisfies shadowing property for any $i\geq 1$. Fix a nonempty open set $U\subseteq X$ and $\eta>0,$ there is $I>0$ such that $X_{\beta_{I}}\cap U\neq \emptyset$ and $h_{top}(f,X_{\beta_{I}})>h_{top}(f)-\eta.$ 
Let $U'=U\cap X_{\beta_{I}}.$ Then $U'$ is a nonempty open set in $X_{\beta_{I}}.$ Note for any  $\alpha_1\alpha_2\alpha_3\alpha_4\alpha_5\alpha_6\in\mathfrak{A}_h$, we have
$$\mathscr{C}(\alpha_1\alpha_2\alpha_3\alpha_4\alpha_5\alpha_6)\cap X_{\beta_{I}}\cap U'\subseteq \mathscr{C}(\alpha_1\alpha_2\alpha_3\alpha_4\alpha_5\alpha_6)\cap U.$$Hence, by Theorem \ref{Theorem B}, we have $$h_{top}(f,\mathscr{C}(\alpha_1\alpha_2\alpha_3\alpha_4\alpha_5\alpha_6)\cap U)\geq h_{top}(f,X_{\beta_{I}})>h_{top}(f)-\eta.$$
By the arbitrariness of $\eta$, we finish the proof.\qed 

\subsection{$C^{1+\alpha}$ surface diffeomorphisms} First of all, let us recall some basic notations for hyperbolic ergodic measures, basic sets and horseshoes. Let $f$ be a $C^1$ diffeomorphism over a compact Riemannian manifold $M.$ 
An ergodic measure is called {\it hyperbolic} if its all Lyapunov exponents are non-zero. A \emph{basic set} for $f$ is a topologically transitive, locally maximal hyperbolic set, if further it is totally disconnected and not finite, then we call it a \emph{horseshoe}.  Every basic set is expansive \cite[Corollary 6.4.10]{Katok1} and has the shadowing property \cite[Theorem 18.1.2]{Katok1}. Hence, every horseshoe is topologically transitive topologically Anosov. If further $f$ is $C^{1+\alpha}$, then it was proved in \cite{Katok,Katok1} that if the a hyperbolic ergodic measure has positive metric entropy, then its metric entropy can be approximated by the topological entropy of horseshoes.

Now, we give the proof of Corollary \ref{Corollary B}.

\emph{Proof of Corollary \ref{Corollary B}.} If $h_{top}(f)=0$, then it is clear the results in Corollary \ref{Corollary B} hold. Now we always assume that $h_{top}(f)>0$. For a surface diffeomorphism,  any ergodic measure with positive metric entropy should be hyperbolic by classical  Ruelle's inequality on metric entropy and Lyapunov exponents \cite{Ru}. In other words,  any  surface diffeomorphism $f$ satisfies that  $$ h_{top}(f)=\sup\{h_\mu: \mu\text{ is     ergodic  }\}=\sup\{h_\mu:\mu\text{ is   hyperbolic and ergodic}\}$$ and thus we have $$ h_{top}(f)=\sup\{h_{top}(f,\Lambda):\Lambda\text{ is a horseshoe}\}.$$ By Theorem \ref{Theorem B}, for any  $\alpha_1\alpha_2\alpha_3\alpha_4\alpha_5\alpha_6\in\mathfrak{A}_h$, we have $$h_{top}(f,\mathscr{C}(\alpha_1\alpha_2\alpha_3\alpha_4\alpha_5\alpha_6))=\sup\{h_{top}(f,\Lambda):\Lambda\text{ is a horseshoe}\}=h_{top}(f).$$
\qed

\subsection{Mañé diffeomorphisms} First of all, let's recall some basic notations for partially hyperbolic diffeomorphisms. Suppose that $f$ is a $C^1$ diffeomorphism on a compact $C^\infty$ Riemannian manifold $M$ and $\Lambda$ is a $f$-invariant set. For two $Df$-invariant bundles $E,F\subset TM|_\Lambda$, we say that $F$ is \emph{dominated} by $E$ if there exist constants $C>0$ and $\lambda\in(0,1)$ such that for any $x\in\Lambda$ and any $n\in\mathbb{N}$, we have $$\|Df^n|_{F(x)}\|\cdot\|Df^{-n}|_{E(f^nx))}\|\leq C\lambda^n.$$ And we denote by $F\oplus_\prec E$ when $F$ is dominated by $E$. A  $f$-invariant set $\Lambda$ is said to admit a \emph{dominated splitting} if there exists a non-trivial $Df$-invariant splitting $TM|_\Lambda=F\oplus_\prec E$. More generally, $\Lambda$ is said to admit a dominated splitting $TM|_\Lambda=E_1\oplus_\prec E_2\oplus_\prec\cdots\oplus_\prec E_k$ if for any $1\leq l\leq k-1$, we have that $$(E_1\oplus\cdots\oplus E_l)\oplus_\prec(E_{l+1}\oplus\cdots\oplus E_k).$$
A $Df$-invariant bundle $F\subset TM|_\Lambda$ is said to be \emph{uniformly contracted} (by $Df$), if there exist constants $C>0$ and $\lambda\in(0,1)$ such that for any $x\in\Lambda$ and any $n\in\mathbb{N}$, we have $\|Df^n|_{F(x)}\|\leq C\lambda^n$; is said to be \emph{uniformly expanded} (by $Df$) if it is uniformly contracted by $Df^{-1}$. We say that $\Lambda$ is said to be \emph{partially hyperbolic} if there exists a $Df$-invariant splitting $TM|_\Lambda=E^s\oplus_\prec E_1^c\oplus_\prec\cdots\oplus_\prec E_k^c\oplus_\prec E^u$ satisfying that
\begin{enumerate}[(1)]
	\item $E^s$ is uniformly contracted and $E^u$ is uniformly expanded;
	\item $E^s$ or $E^u$ can be trivial, but cannot be trivial simultaneously.
\end{enumerate} 
From the construction of \cite{Climenhaga-Fisher-Thompson-2019}, we know that every $f\in\mathcal{U}_{\rho,r}$ is partially hyperbolic with one-dimensional center bundle and one-dimensional unstable bundle. 

Suppose that $\mu$ is hyperbolic ergodic measure, has a positive Lyapunov exponent and a negative Lyapunov exponent. Let $E_1\oplus\cdots E_l\oplus E_{l+1}\cdots E_{l+s}$ be its Oseledets splitting with corresponding Lyapunov exponents $\lambda_1<\cdots<\lambda_l<0<\lambda_{l+1}<\cdots<\lambda_{l+s}$ defined for $\mu$-a.e. $x\in M$. Denote $E^-=E_1\oplus\cdots E_l$ and $E^+=E_{l+1}\cdots E_{l+s}$. We say that $\mu$ admits a dominated splitting corresponding to the stable/unstable subspaces of its Oseledets splitting if its support $S_\mu$ admits a dominated splitting $F\oplus_\prec E$ such that $F(x)=E^-(x)$ and $E(x)=E^+(x)$ for $\mu$-a.e. $x\in M$. By Ruelle's inequality on metric, every hyperbolic ergodic measure with positive metric entropy has a positive Lyapunov exponent and a negative Lyapunov exponent. From \cite{Gelfert-2016,Gelfert-2022}, we know that if a hyperbolic ergodic measure has positive metric entropy and admits a dominated splitting corresponding to the stable/unstable subspaces of its Oseledets splitting, then its metric entropy can be approximated by the topological entropy of  basic sets.

Now, we give the proof of Corollary \ref{Corollary C} by using the notations in \cite{Climenhaga-Fisher-Thompson-2019}.

\emph{Proof of Corollary \ref{Corollary C}.} Suppose that $g\in\mathcal{U}_{\rho,r}$ satisfies $r(h+\log L)+H(2r)<h_{top}(g)$ and is partially hyperbolic with $T\mathbb{T}^d=E^s\oplus_\prec E^c\oplus_\prec E^u$ with $\dim E^c=1$. By \cite[Theorem 3.3, Lemma 5.8]{Climenhaga-Fisher-Thompson-2019}, we have $P(\mathcal{C},0)\leq r(h+\log L)+H(2r)<h_{top}(g)$, where the definition of $P(\mathcal{C},0)$ can be found in \cite[section 3.3]{Climenhaga-Fisher-Thompson-2019}. Given $0<\eta<h_{top}(g)-P(\mathcal{C},0)$, by the variational principle, there exists a ergodic measure $\mu$ such that $h_\mu>h_{top}(g)-\eta>P(\mathcal{C},0)\geq0$. Hence, by \cite[Lemma 3.5]{Climenhaga-Fisher-Thompson-2019}, one has $\mu(A^+)=0$, where $$A^+=\{x\in\mathbb{T}^d:\text{there exists }K(x)\text{ such that }\frac{1}{n}S_n^g\chi(x)<r\text{ for any }n>K(x)\}.$$ Thus for $\mu$-a.e. $x\in\mathbb{T}^d$, there exists an increasing sequence of positive integers $\{n_i\}_{i\geq1}$ such that $\frac{1}{n_i}S_{n_i}^g\chi(x)\geq r$ for any $i\in\mathbb{N}$. Then by \cite[(4.1)]{Climenhaga-Fisher-Thompson-2019}, there exists $0<\theta_r(g)<1$, such that $\|Dg^{n_i}|_{E^{cs}(x)}\|\leq(\theta_r(g))^{n_i}$, where $E^{cs}=E^c\oplus E^s$. Hence, for $\mu$-a.e. $x\in\mathbb{T}^d$, all the Lyapunov exponents for $E^{cs}(x)$ are negative. As a result, we have that $\mu$ is hyperbolic. Since $g$ is partially hyperbolic with one-dimensional center bundle, one has that $\mu$ admits a dominated splitting corresponding to the stable/unstable subspaces of its Oseledets splitting. Thus $h_\mu$ can be approximated by the topological entropy of basic sets.  It is clear that a basic set with positive topological entropy has infinite points and thus is topologically transitive topologically Anosov. Therefore, by Theorem \ref{Theorem B}, for any  $\alpha_1\alpha_2\alpha_3\alpha_4\alpha_5\alpha_6\in\mathfrak{A}$, we have $$h_{top}(g,\mathscr{C}(\alpha_1\alpha_2\alpha_3\alpha_4\alpha_5\alpha_6))\geq h_\mu>h_{top}(g)-\eta.$$
By the arbitrariness of $\eta$, we finish the proof.\qed

\bigskip

\textbf{Acknowledgements.} The authors are grateful to Prof.  Yu Huang, Wenxiang Sun and Xiaoyi Wang for their numerous remarks and fruitful discussions.  The authors are supported by NSFC No. 12471182  and  Natural Science Foundation of Shanghai No. 23ZR1405800.

\end{document}